\documentclass[12pt, leqno]{article}
\usepackage{amsmath}
\usepackage{amsthm}
\usepackage{amssymb}
\usepackage{latexsym}
\allowdisplaybreaks[1]
\usepackage{amsfonts}
\usepackage{ascmac}
\usepackage{color}
\usepackage{enumerate}
\usepackage{graphicx}

\theoremstyle{definition}
\theoremstyle{plain}
\allowdisplaybreaks[1]
\date{}
\newtheorem{Thm}{Theorem}[section]
\newtheorem{Prop}[Thm]{Proposition}
\newtheorem{Lemma}[Thm]{Lemma}
\newtheorem{Cor}[Thm]{Corollary}

\newcommand{\graph}{\mbox{\rm graph}}

\newcommand{\dt}{\tau}
\newcommand{\dx}{h}

\newcommand{\dis}{\displaystyle}

\newcommand{\norm}{\parallel}

\newcommand{\Z}{{\mathbb Z}}

\newcommand{\T}{{\mathbb T}}
\newcommand{\N}{{\mathbb N}}
\newcommand{\R}{{\mathbb R}}

\newcommand{\G}{ \mathcal{G}}
\newcommand{\tG}{ \tilde{\mathcal{G}}}
\newcommand{\1}{{\bf 1} }

\newcommand{\ep}{\varepsilon }
\newcommand{\bH}{ \bar{H} }
\newcommand{\pr}{ {\rm pr.} }
\def\text#1{\mbox{#1 }}

\title{\bf Weak KAM theory \\for action minimizing random walks}
\author{Kohei Soga
\footnote{Department of Mathematics, Faculty of Science and Technology, Keio University, 3-14-1 Hiyoshi, Kohoku-ku, Yokohama, 223-8522, Japan. E-mail:  soga@math.keio.ac.jp 
}}
\begin{document}
\maketitle
\begin{abstract} 
\noindent We introduce a class of controlled random walks on a grid in $\T^d$ and investigate global properties of action minimizing random walks for a certain action functional together with Hamilton-Jacobi equations on the grid. This yields an analogue of weak KAM theory, which recovers a part of  original weak KAM theory through the hyperbolic scaling limit. 
\noindent 
 
\medskip\medskip\medskip\medskip

  \medskip

\end{abstract}
%
\tableofcontents

\setcounter{section}{0}
\setcounter{equation}{0}
\section{Introduction}

We introduce a class of controlled random walks and investigate action minimizing problems in the class based on the framework of weak KAM theory and Aubry-Mather theory. 

\subsection{Action minimizing random walk}

Let $G_h$ be the grid with a small unit length $0<h\ll1$ in $\T^d:=\R^d/\Z^d$, i.e.,  $G_h=h\Z^d\cap[0,1]^d$, $h^{-1}\in\N$ (even number) with identification of $0$ and $1$. 
Let $\{e_i\}$ be the standard basis of $\R^d$ and set $B=\{\pm e_i\}_{i=1,\ldots,d}$.  Consider the jump process from $x\in G_h$ to one of the points $\{x+h\omega\}_{\omega\in B}$ with a transition probability $\rho(\omega)$. 
Iteration of the jump process yields a random walk whose paths are given as  
$$\gamma^0=x\in G_h,\quad \gamma^{k+1}=\gamma^k+h\omega,$$     
where $k$ indicates the number of iteration. We associate $k$ with time in such a way that each jump process takes place within a small unit time $0<\tau\ll1$, $\tau^{-1}\in\N$ (even number). 
We introduce time-$1$-periodic inhomogeneous transition probabilities: for a given function $\xi:G_h\times\{\tau k\,|\,k\in\Z\} \to K\subset\R^d$ with $\xi(\cdot,t+1)=\xi(\cdot,t)$, define 
$$\rho^\pm(x,t;\omega):=\frac{1}{2d}\pm\frac{\tau}{2h}\xi(x,t)\cdot\omega,$$    
where $\rho^+(x,t;\omega)$ (resp. $\rho^-(x,t;\omega)$) stands for the transition probability of the jump from $(x,t)$ to $(x+h\omega,t+\tau)$ (resp.   from $(x,t)$ to $(x+h\omega,t-\tau)$). The paths $\gamma^k$ for $\rho^-$ are re-indexed as $\gamma^{-k}$ with $k\ge0$, since it is a time-backward process.   
We see that such a random walk is diffusive discretization of an ODE 
$$\mbox{$x'(s)=f(x(s),s)$, $x(0)=x$ \,\,\,with $f:\T^d\times\T\to\T^d$, Lipschitz.}$$ In fact, one can prove that, if $\xi=f|_{G_h\times\{\tau k\,|\,k\in\Z\}}$ and $\rho=\rho^+$ (resp.  $\rho=\rho^-$), the random walk tends to a solution $x(t):[0,T]\to\T^d $ (resp. $x(t):[-T,0]\to\T^d $) of the ODE as a consequence of the hyperbolic scaling limit, i.e., $h,\tau\to0$ with $0<\lambda_0\le \tau/h\le\lambda_1$ (the law of large numbers: see Soga \cite{Soga1}). 

 Let us relate the random walks with optimal control. We  regard each $\xi$ as a control and look for controls that maximize/minimize the cost functionals
\begin{eqnarray}\label{00001}
&&\mathcal{L}^{l}_+(\xi;v^+):=E\Big[\sum_{0\le k< l} -L^{(c)}(\gamma^k,\tau k+\tau, \xi(\gamma^k,\tau k))\tau +v^+(\gamma^{l}, \tau l)  \Big]\mbox{ \,\,\,for $\rho=\rho^+$},\\\label{00002}
&&\mathcal{L}^{l}_-(\xi;v^-):=E\Big[\sum_{-l<k\le 0} L^{(c)}(\gamma^k, \tau k-\tau, \xi(\gamma^k,\tau k))\tau +v^-(\gamma^{-l},-\tau l)  \Big]\mbox{\,\,\, for $\rho=\rho^-$},
\end{eqnarray} 
where $E$ stands for the average with respect to the probability measure of the random walk generated by $\xi$ in the above manner;  $L^{(c)}:=L(x,t,\zeta)-c\cdot\zeta$ with a constant $c\in\R^d$ and a Tonelli Lagrangian  $L:\T^d\times\T\times\R^d\to\R$;   $v^\pm:G_h\times\{\tau k\,|\,k\in\Z\} \to \R$ with $v^\pm(\cdot,t+1)=v^\pm(\cdot,t)$ are given functions. Later, we will set up the problem more precisely.    

In this paper, {\it we find appropriate families of functions $v^\pm$ via Hamilton-Jacobi type equations on the grid and investigate properties of the maximizers $\xi^{\ast+}$ of $\sup_{\xi}\mathcal{L}^{+l}(\xi;v^+)$ and the minimizers  $\xi^{\ast-}$ of $\inf_{\xi}\mathcal{L}^{l}_-(\xi;v^-)$ including the asymptotics of the maximizing/minimizing random walks as $l\to\infty$ in the  framework of weak KAM theory and Aubry-Mather theory; we also investigate the hyperbolic scaling limit of the issue to recover (a part of) exact weak KAM theory and Aubry-Mather theory.} 
The essential tool of our investigation is the value function of a time depending Hamilton-Jacobi equation on the grid, which is investigated in Soga \cite{Soga5} from a viewpoint of numerical analysis of viscosity solutions to initial value problems, i.e.,   we associate   $\mathcal{L}^{l}_+(\xi;v^+)$ with the value function of a discrete equation corresponding to     
\begin{eqnarray}\label{+}
v_t+H(x,t,c+v_x)=0,\quad x\in\T^d,\,\,\,t<0
\end{eqnarray}
and $\mathcal{L}^{l}_-(\xi;v^-)$ with that of  the discrete  equation corresponding to 
\begin{eqnarray}\label{-}
v_t+H(x,t,c+v_x)=0,\quad x\in\T^d,\,\,\, t>0,
\end{eqnarray}    
where $H$ is the Legendre transform of $L$ in the sense of
$$H(x,t,p)=\sup_{\zeta\in\R^d}(p\cdot\zeta-L(x,t,\zeta));$$ then, we investigate our action maximizing/minimizing problem with $v^\pm$ being time-$1$-periodic solutions of the discrete Hamilton-Jacobi equations. Hence, the first half of this paper is devoted to prove existence of time-$1$-periodic discrete solutions. 
We will see that our investigation of $\mathcal{L}^{l}_\pm(\xi;v^\pm)$ corresponds to that of the Lax-Oleinik semigroup $\mathcal{T}^t_\pm$ and weak KAM solutions $u_\pm$ (see the notation below). Unlike the standard weak KAM theory, our ``dynamical system'' has a diffusion effect that makes time-forward evolution and time-backward evolution distinct. Therefore, we treat the problem with $\mathcal{L}^{l}_+(\xi;v^+)$  and that with $\mathcal{L}^{l}_-(\xi;v^-)$ independently.  
Since the investigation with $\mathcal{L}^{l}_+(\xi;v^+)$ is quite parallel to that of  $\mathcal{L}^{l}_-(\xi;v^-)$,  we omit detailed discussion in this paper.   

In Section 2, we set up a class of controlled random walks and the corresponding Hamilton-Jacobi equation on a grid. Then, we investigate initial value problems of the Hamilton-Jacobi equation within the time interval $[0,1]$ through stochastic calculus of variations, to obtain necessary a priori estimates and convergence properties. In Section 3 and 4, we find time-$1$-periodic solutions of the Hamilton-Jacobi equation and investigate action minimizing random walks, to obtain an analogue of weak KAM theory and Aubry-Mather theory.           
\subsection{Weak KAM theory and related literature}

Before going into details, we give a short summary of  weak KAM theory established by Albert Fathi \cite{Fathi97-1}, \cite{Fathi97-2}, \cite{Fathi98-1}, \cite{Fathi98-2}, \cite{Fathi-book} and its analogues  recently developed by many reserachers, so that readers can get a clear view of the current paper as one of such analogous theories.   
Briefly speaking, weak KAM theory investigates global properties of the action maximizing/minimizing curves for 
\begin{eqnarray*}
&&\mathcal{L}^{t}_+(\gamma;v^+):= -\int_0^t L^{(c)}(\gamma(s),\gamma'(s))ds  +v^+(\gamma(t)),\,\,\,\gamma\in AC([0,t];\T^d),\,\,\,\gamma(0)=x,\\
&&\mathcal{L}^{t}_-(\gamma;v^-):= \int^0_{-t} L^{(c)}(\gamma(s),\gamma'(s))ds +v^-(\gamma(-t)),\,\,\,\gamma\in AC([-t,0];\T^d),\,\,\,\gamma(0)=x
\end{eqnarray*} 
with $v^\pm$ being weak solutions of the stationary Hamilton-Jacobi equation
\begin{eqnarray}\label{HJ1}
H(x,c+v_x(x))=h\mbox{\quad in $\T^d$},
\end{eqnarray}      
where $AC(I;\T^d)$ stands for the family of all absolutely continuous curves $\gamma:I\to\T^d$ and  $h\in\R$ is some constant. This problem is closely related to one of the central issues of Hamiltonian dynamics, i.e., the issue to investigate structures of the phase space and to understand time global properties of motions. 
Kolmogorov-Arnold-Moser (KAM) theory provides an excellent framework to show existence of invariant manifolds called KAM tori for nearly integrable Hamiltonian systems, i.e., Hamiltonian systems  
\begin{eqnarray}\label{HS}
x'(s)=H_p(x(s),p(s)),\qquad p'(s)=-H_x(x(s),p(s))
\end{eqnarray}
generated by Hamiltonians $H(x,p)=H_0(p)+H_1(x,p):\T^d\times G\to\R$ with $G\subset\R^d$ and $\norm H_1\norm_{C^0} \ll1$ (KAM theory is available also for  nearly integrable twist maps). After KAM theory was established, lots of efforts have been made to understand the situation where $H_1$ gets larger or Hamiltonian systems far from being integrable.  A rigorous answer to this kind of questions is given by Aubry-Le Daeron \cite{Aubry}  and Mather \cite{Mather1} for twist maps on the annulus  (Aubry-Mather theory), where they find invariant sets called Aubry-Mather sets which can be seen as a generalization of smooth invariant circles obtained in KAM theory. Then, Moser \cite{Moser1} shows that for each smooth twist map there exists a certain time-dependent Hamiltonian system generated by  $H(x,t,p):\T\times\T\times\R\to\R$, strictly convex in the $p$-variable such that its time-$1$ map coincides with the twist map. Due to this result, Aubry-Mather theory can be extended to continuum Hamiltonian dynamics. Mather \cite{Mather2} generalizes his previous theory to the Lagrangian dynamics with an arbitrary space dimension,  generated by $C^2$-Lagrangians
$$L(x,\zeta):\T^d\times\R^d\to\R,\quad\mbox{strictly convex and superlinear with respect to  $\zeta$},$$
which is nowadays called Tonelli Lagrangians. Here, the Euler-Lagrange system is given as    
\begin{eqnarray}\label{EL}
\frac{d}{ds}L_\zeta(x(s),x'(s))=L_x(x(s),x'(s)), 
\end{eqnarray}
which is equivalent to the Hamiltonian system (\ref{HS}) generated by  the Legendre transform $H$ of $L$. 
The main tool of Mather's investigation is the family of action minimizing invariant measures for 
\medskip 

\noindent {\bf Mather's minimizing problem}: {\it For each constant $c\in\R^d$, consider 
\begin{eqnarray}\label{p-mather}
\inf_\mu \int_{\T^d\times\R^d} L^{(c)}(x,\zeta)\,d\mu, 
\end{eqnarray}
where the infimum is taken over all probability measures on $\T^d\times\R^d$ that are invariant under the Euler-Lagrange flow generated by $L$.}
\medskip

\noindent A minimizing measure of \eqref{p-mather} is called a Mather measure, and the union of the supports of all Mather measures for each fixed $c$ is called the {\it Mather set}.  Each Mather set provides an invariant set analogous to  a KAM torus in the corresponding Hamiltonian system. 

Fathi relates KAM theory and Aubry-Mather theory to weak solutions of Hamilton-Jacobi equations of \eqref{HJ1}. 
This seems to be quite natural, since existence of a KAM torus implies the existence of $c\in\R^n$, $h\in\R$ and $v\in C^2(\T^d;\R)$ that satisfy (\ref{HJ1}) (each KAM torus is a Lagrangian submanifold and it is given as the graph of the  derivative of $v$). More generally with an arbitrary $C^2$-Hamiltonian $H$,  if $c\in\R^d$ and $v\in C^2(\T^d;\R)$ have the property that 
$$\graph(c+v_x):=\{(x,c+v_x(x))\,|\,x\in\T^d\}$$
is invariant under the Hamiltonian flow $\phi^s_H$ of \eqref{HS}, then, $v$ satisfies \eqref{HJ1}  with some constant $h\in\R$; conversely, if there exists a $C^2$-solution $v$ of \eqref{HJ1}, then, $\graph(c+v_x)$ is invariant under $\phi^s_H$. Unfortunately, one cannot always expect a smooth solution of \eqref{HJ1}. Lions-Papanicolaou-Varadhan \cite{LPV}  analyze \eqref{HJ1} in the class of viscosity solutions, where they prove that for each $c\in\R^d$ there exists a unique number $\bar{H}(c)\in\R$ such that \eqref{HJ1} admits at least one viscosity solution only when $h=\bar{H}(c)$. This result is based on methods of viscosity solutions without any motivation from  dynamical systems. Fathi investigates \eqref{HJ1} by means of the Lax-Oleinik semigroup $\{\mathcal{T}^t_\pm\}_{t\ge0}$ defined in $C^0(\T^n;\R)$ to itself as  
\begin{eqnarray}\label{LO+}
&&\mathcal{T}^t_+u(x):=\sup_{\gamma\in AC([0,t];\T^d),\gamma(0)=x}\mathcal{L}^{t}_+(\gamma;u),\\\label{LO-}
&&\mathcal{T}^t_-u(x):=\inf_{\gamma\in AC([-t,0];\T^d),\gamma(0)=x}\mathcal{L}^{t}_-(\gamma;u).
\end{eqnarray}       
 Tonelli's calculus of variations guarantees that maximizing/minimizing  curves for the right hand side of \eqref{LO+} and   \eqref{LO-} exist, and both of which are  $C^2$-solutions of the Euler-Lagrange system \eqref{EL} for any $c\in\R^d$. Hence, the results from   $\mathcal{T}^t_\pm$ are closely related to each other. The key fact of weak KAM theory is the existence of a constant $\bar{H}(c)\in\R$ and $u_\pm\in C^0(\T^n;\R)$ such that 
 $$\mbox{$\mathcal{T}^t_+ u_+ -\bar{H}(c)t=u_+$,\,\,\, $\mathcal{T}^t_- u_- + \bar{H}(c)t=u_-$\,\,\,\,  for all $t\ge0$,}$$ where $u_\pm$ are called weak KAM solutions and $\bar{H}(c)$ the effective Hamiltonian (there are other names of $\bar{H}(c)$). 
Here are some important  consequences  \cite{Fathi-book}: 
\begin{itemize}
\item $u_+$ is semiconvex, and an a.e. solution of \eqref{HJ1} with $h=\bar{H}(c)$. 
\item $u_-$ is semiconcave, and a viscosity solution of \eqref{HJ1} with $h=\bar{H}(c)$. 
\item $\graph(c+\frac{\partial }{\partial x}u_+)$ is invariant under $\phi^s_H$ for $s\ge0$. 
\item $\graph(c+\frac{\partial }{\partial x}u_-)$ is invariant under $\phi^s_H$ for $s\le0$.
\item There exists $\mathcal{M}^\ast(u_\pm)\subset\graph(c+\frac{\partial }{\partial x}u_\pm)$ that is invariant under $\phi^s_H$ for $s\in\R$ and contains the Mather set for $c$.
\item $-\bar{H}(c)$ coincides with the infimum \eqref{p-mather}. 
\item  The ``recurrence rate'' of a minimizing curve $\gamma^\ast$ of $\mathcal{T}^t_-u_-$ extended to $-(\infty,0]$ yields a Mather measure, i.e., roughly speaking, the probability measure of $\T^d\times\R^d$ defined as 
$${\rm meas}[A]:=\lim_{t\to\infty} \frac{\mbox{lenth} \{s\in[-t,0]\,|\,(\gamma^\ast(s),\gamma^\ast{}'(s))\in A\}}{t},\quad A\subset\T^d\times\R^d$$
is a Mather measure (the same holds for an extended maximizing curve of $\mathcal{T}^t_+u_+$). 
\end{itemize} 
Weak KAM theory for twist maps is given by E \cite{WE}. We refer also to the works by Evans-Gomes \cite{Evans-Gomes1}, \cite{Evans-Gomes2} and Fathi-Siconolfi \cite{Fathi-Siconolfi} for further  ideas and methods in weak KAM theory. 
      
As an analogous problem, consider the action maximizing/minimizing curves for
\begin{eqnarray*}
&&\mathcal{L}^{t}_+(\gamma;v^+):= -\int_0^t e^{-\ep s}L^{(c)}(\gamma(s),\gamma'(s))ds  +e^{-\ep t}v^+(\gamma(t)),\\
&&\qquad\qquad\qquad\qquad\qquad\qquad\qquad \gamma\in AC([0,t];\T^d),\,\,\,\gamma(0)=x,\\
&&\mathcal{L}^{t}_-(\gamma;v^-):= \int^0_{-t} e^{\ep s}L^{(c)}(\gamma(s),\gamma'(s))ds +e^{-\ep t}v^-(\gamma(-t)),\\
&&\qquad\qquad\qquad\qquad\qquad\qquad\qquad  \gamma\in AC([-t,0];\T^d),\,\,\,\gamma(0)=x,
\end{eqnarray*}  
where $\ep>0$, the discount factor,  gives a dissipation effect. 
Then, we have an analogue of weak KAM theory with $v^\pm$ being weak solutions of the discounted Hamilton-Jacobi equation  
\begin{eqnarray}\label{dHJ}   
\varepsilon v^\ep(x)+H(x,c+v_x^\ep(x))=0\mbox{\qquad in $\T^d$}. 
\end{eqnarray} 
The corresponding dynamical system is the discounted Euler-Lagrange system
\begin{eqnarray} \label{dEL}
\frac{d}{ds}\left( L_\zeta(x(s),x'(s)) \right)=L_x(x(s),x'(s))-\ep L_\zeta(x(s),x'(s))+\ep c, 
\end{eqnarray}
which is equivalent to the discounted Hamiltonian system
\begin{eqnarray}\label{dHS}
\left\{
\begin{array}{l}
x'(s)=H_p(x(s),p(s)),\\
p'(s)=-H_x(x(s),p(s))+\ep c- \ep p(s),
\end{array}
\right. 
\end{eqnarray}
where both of maximizing/minimizing curves for $\mathcal{L}^{t}_\pm$ are $C^2$-solutions of \eqref{dEL}.  
 We refer to Mar\`o-Sorrentino \cite{MS} and Mitake-Soga \cite{Mitake-Soga} for  Aubry-Mather theory and weak KAM theory for the discounted problems; Gomes \cite{Gomes2}, Iturriaga-Sanchez Morgado \cite{ISM}, Davini-Fathi-Iturriaga-Zavidovique \cite{Davini}, \cite{Davini 2} for analysis of the selection problem in the vanishing discount process of \eqref{dHJ} based on weak KAM theory, i.e., the problem whether or not the whole sequence $\{v^\ep\}_{\ep>0}$ converges to some weak KAM solution as $\ep\to0+$ (convergence up to subsequence is well-known  and is already used in \cite{LPV});  Mitake-Tran \cite{MT} for a  result similar to \cite{Davini} on the selection problem for degenerate viscous Hamilton-Jacobi equations based on a PDE approach called the nonlinear adjoint method introduced by Evans \cite{Evans0}.  Recently, Wang-Wang-Yan \cite{WWY1}, \cite{WWY2}, \cite{WWY3} develop weak KAM theory for contact Lagrangian/Hamiltonian dynamics and contact Hamilton-Jacobi equations, which include discounted problems as a particular case. The contact problem arises from the action maximizing/minimizing curves for some implicitly given action functionals. The selection problem in the vanishing contact process is studied in Chen-Cheng-Ishii-Zhao \cite{CCI}.  
 
Related to smooth approximation, an analogue of weak KAM theory has been developed for viscous Hamilton-Jacobi equations.  The first result in such a direction is provided by Moser \cite{Moser-5} (though it is not about viscous Hamilton-Jacobi equations), where he shows smooth approximation of Aubry-Mather sets by a regularization technique. 
After  weak KAM theory is announced, Jauslin-Kreiss-Moser \cite{JKM} demonstrate smooth approximation of $\graph(c+\frac{\partial}{\partial x}u_-)$ in the context of weak KAM theory for twist maps through the vanishing viscosity method for the forced Burgers equations (they are equivalent to Hamilton-Jacobi equations in $1$-dimensional space), via a PDE approach. Furthermore, they make the first attempt to solve the  selection problem in the vanishing viscosity process. 
The regularized problems with artificial viscosities can be treated also in terms of stochastic optimal control, based on the pioneering works by Fleming \cite{Fleming}: 
consider the action maximizing/minimizing controls for
\begin{eqnarray*}
&&\mathcal{L}^{t}_+(\xi;v^+):= E\Big[-\int^t_{0} L^{(c)}(\gamma(s),s,\xi(\gamma(s),s))ds+ v^+(\gamma(t),t)\Big],\\
&&d\gamma(s)=\xi(\gamma(s),s)ds+\sqrt{2\nu}dW(s),\quad \gamma(0)=x,\quad 0\le s\le t
\end{eqnarray*}
and 
\begin{eqnarray*}
&&\mathcal{L}^{t}_-(\xi;v^-):= E\Big[\int^0_{-t} L^{(c)}(\gamma(s),s,\xi(\gamma(s),s))ds +v^-(\gamma(-t),-t)\Big],\\
&&d\gamma(s)=\xi(\gamma(s),s)ds+\sqrt{2\nu}dW(|s|),\quad \gamma(0)=x,\quad -t\le s\le 0\\
&&(\mbox{backward sense}),
\end{eqnarray*}   
where $\xi\in C^1(\T^d\times\T;\R^d)$ are controls, $W$ is the standard Brownian motion and $E[\cdot]$ stands for the expectation with respect to the Wiener measure; 
$\mathcal{L}^{t}_+ $, $\mathcal{L}^{t}_-$ involve the viscous Hamilton-Jacobi equations
\begin{eqnarray*}
&&v_t+H(x,t,c+v_x)=-\nu \Delta v,\,\,\,x\in\T^d,\,\,\,t<0,\\
&&v_t+H(x,t,c+v_x)=\nu \Delta v,\,\,\,x\in\T^d,\,\,\,t>0,
\end{eqnarray*}
respectively, with $\nu>0$. There are many results based on the above $\mathcal{L}^{t}_\pm$ and ideas of weak KAM theory:  we refer to  Gomes \cite{Gomes1} and Iturriaga-Sanchez Morgado \cite{ISM0} for analysis of $H(x,t,c+v_x)=h+\nu \Delta v$ in $\T^d$ and a stochastic analogue of Mather measures; Bessi \cite{B} and Anantharaman-Iturriaga-Padilla-S\'anchez Morgado \cite{A} for partial answers to the selection problem in the vanishing viscosity process as a generalization of Jauslin-Kreiss-Moser \cite{JKM}.

The problem of action maximizing/minimizing random walks for \eqref{00001} and \eqref{00002} with $d=1$ is partially studied in Soga \cite{Soga2}, \cite{Soga3}. Then,  Soga \cite{Soga4} applies the results in  \cite{Soga2} and  \cite{Soga3} to  an investigation of the selection problem in the limit process of finite difference approximation, where the result is apparently similar to that of Bessi \cite{B}, but there is a crucial difference due to the finite propagation speed of random walks under the hyperbolic scaling.      

In addition to the above mentioned analogues of weak KAM theory,  we refer to Bernard-Buffoni \cite{Bernard} and Zavidovique \cite{Zavidovique} for weak KAM like formulation of  abstract functional equations; Evans \cite{Evans} for a quantum analogue of weak KAM theory; Bessi \cite{B2} for  an Aubry-Mather theory approach to the Vlasov equation.   

The results of this paper can be seen also as numerical methods of weak KAM theory and Hamilton-Jacobi equations. In fact, we construct analogous objects of exact weak KAM solutions, calibrated curves, effective Hamiltonians, Mather measures, etc., which tend to the exact ones at the hyperbolic scaling limit. 
Let us relate such results to the literature of numerical analysis of Hamilton-Jacobi equations and weak KAM theory.  Crandall-Lions \cite{Crandall-Lions} show an abstract result that monotone finite difference schemes for Hamilton-Jacobi equations yield viscosity solutions of initial value problems, which is generalized by Souganidis \cite{Souganidis}.  Verification of the monotonicity of a scheme under consideration is highly non-trivial, also for Tonelli Hamiltonians.  Soga \cite{Soga2} develops mathematical analysis of the Lax-Friedrichs finite difference scheme applied to hyperbolic scalar conservation laws and the corresponding  Hamilton-Jacobi equations in the $1$-dimensional setting through techniques of optimal control theory, and extends the classical results by Oleinik \cite{Oleinik} to obtain monotonicity and convergence within an arbitrary time interval. Then, Soga \cite{Soga3} shows existence of time-$1$-periodic discrete solutions, which corresponds to weak KAM solutions $u_-$. Soga \cite{Soga5} generalizes  the work \cite{Soga2}  to problems with a multi space dimension. The current paper discusses existence of time-$1$-periodic discrete solutions  based on \cite{Soga5}. 
We mention that there is a big literature on numerical analysis or computational observation of weak KAM theory based on other techniques to approximate a specific object such as viscosity solutions, effective Hamiltonians, e.g., Gomes-Oberman \cite{GO}, Rorro \cite{Rorro}, Nishida-Soga \cite{Nishida-Soga}, Bouillard-Faou-Zavidovique \cite{BFZ}.  In contrast to these works,  the main  feature of our current investigation is that  {\it we give a framework that produces analogues of viscosity solutions, their derivatives,  their characteristic curves, etc., all at once; there are explicit equations for these objects, for which one can find a  structure similar to exact weak KAM theory; these objects and the structure are rigorously convergent.}

\setcounter{section}{1}
\setcounter{equation}{0}
\section{Random walk and Hamilton-Jacobi equation on grid}

We set up a class of controlled random walks and  Hamilton-Jacobi equations on a grid.   
 Then, based on the stochastic and variational approach \cite{Soga5}, we analyze the initial value problems of the equations within the time interval $[0,1]$ to obtain the solution maps, i.e., the time-$1$ maps with good a priori estimates and convergence properties.

The function $L$ is assumed to satisfy the following (L1)--(L4):
\begin{enumerate}
\item[(L1)] $L(x,t,\zeta):\T^d\times\T\times\R^d\to\R$, $C^2$,
\item[(L2)] $L_{\zeta\zeta}(x,t,\zeta)$ is positive definite in $\T^d\times\T\times\R^d$,
\item[(L3)] $L$ is  superlinear with respect to $\zeta$, i.e., for each $a\ge0$ there exists $b_1(a)\in\R$ such that $L(x,t,\zeta)\ge a  |\zeta|+b_1(a)$  in $\T^d\times\T\times\R^d$,
\item[(L4)] There exists $\alpha>0$ such that $| L_{x^j}|\le \alpha (1+|L|)$ in $\T^d\times\T\times\R^d$ for $j=1,\ldots,d$.
\end{enumerate} 
Here,  $x\cdot y:=\sum_{1\le j\le q} x^jy^j$, $| x|:=\sqrt{\sum_{1\le j\le q}(x^j)^2}$, $|x|_\infty:=\max_{1\le j\le q}|x^j|$ for $x,y\in\R^q$. Note that, due to (L1)--(L3), the function the Legendre transform $H(x,t,p):\R^d\times\R\times\R^d\to\R$ of $L$ with respect to $\zeta$ is well-defined and given by
$$H(x,t,p)=\sup_{\zeta\in\R^d}\{p\cdot\zeta-L(x,t,\zeta)\}.$$ 
with the properties: 
\begin{enumerate}
\item[(H1)] $H(x,t,p):\T^d\times\T\times\R^d\to\R$, $C^2$,
\item[(H2)] $H_{pp}(x,t,p)$ is positive definite in $\T^d\times\T\times\R^d$,
\item[(H3)] $H$ is uniformly superlinear with respect to $p$, i.e.,  for each $a\ge0$ there exists $b_2(a)\in\R$ such that $H(x,t,p)\ge a  |p|+b_2(a)$  in $\T^d\times\T\times\R^d$,
\end{enumerate} 
(L4) is a sufficient condition for the Euler-Lagrange flow to be complete (global in time). For each $c\in\R^d$, let $L^{(c)}$ be the Legendre transform of $H(x,t,c+p)$ with respect to $p$, i.e.,  $L^{(c)}(x,t,\zeta)=L(x,t,\zeta)-c\cdot\zeta$.   
Throughout this paper, the dependence on a variable in $\T$ is regarded as the dependence on a variable in $\R$ with $1$-periodicity.  
\subsection{Controlled  random walk}
 
Let $\delta=(h,\tau)$ be a pair of a unit small length of space and a unit small time. We choose $h=(2N)^{-1}$, $\tau=(2K)^{-1}$ with $N,K\in\N$, so that our grid can always contain $1$.   Introduce the following notation:
\begin{eqnarray*}
&&x_m:=(hm^1,\ldots,hm^d), \,\,\,t_k:=\tau k\mbox{\quad for $m=(m^1,\ldots,m^d)\in\Z^d$, $k\in\Z$},\\
&&G_{\rm even}:=\{x_m \,|\,m\in\Z^d,\,\,\,m^1+\cdots+ m^d={\rm even}\},\\
&& G_{\rm odd}:=\{x_m \,|\,m\in\Z^d\,\,\,m^1+\cdots+ m^d={\rm odd}\},\\
&&\mathcal{G}_\delta:=\bigcup_{k\in\Z}\Big\{(G_{\rm even}\times\{t_{2k}\})\cup(G_{\rm odd}\times\{t_{2k+1}\})\Big\}\\
&&(\mbox{summation of the indexes of each point is even}),\\
&&\tG_\delta:=\bigcup_{k\in\Z}\Big\{(G_{\rm odd}\times\{t_{2k}\})\cup(G_{\rm even}\times\{t_{2k+1}\})\Big\}\\
&&(\mbox{summation of the indexes of each point is odd}),\\
&&\mbox{$\{e_1,\ldots,e_d\}$ is the standard basis of $\R^d$},\\
&&B:=\{\pm e_1,\ldots,\pm e_d\}.
\end{eqnarray*}
Note that $x_{m\pm2Ne_i}=x_m\pm e_i,\quad t_{k\pm2K}=t_k\pm1$ and 
\begin{eqnarray*}
&&\pr G_{\rm even}:=\{x_m\in G_{\rm even}\,|\,0\le m_i\le2N-1  \},\\
&&\pr G_{\rm odd}:=\{x_m\in G_{\rm odd}\,|\,0\le m_i\le2N-1 \},\\
&&\pr \G_\delta:=\{(x_m,t_k)\in\G_\delta\,|\,0\le m_i\le2N-1,\,\,\,0\le k\le 2K-1    \},\\
&&\pr \tG_\delta:=\{(x_m,t_k)\in\tG_\delta\,|\,0\le m_i\le2N-1,\,\,\,0\le k\le 2K-1    \},
\end{eqnarray*}
where $\pr \G_\delta$, $\pr \tG_\delta$ are seen as discretization of  $\T^d\times\T$ with the mesh size $\delta=(h,\tau)$. 

\vspace{3mm}
\begin{center}
\includegraphics[bb=0 0 561 446, width=9cm]{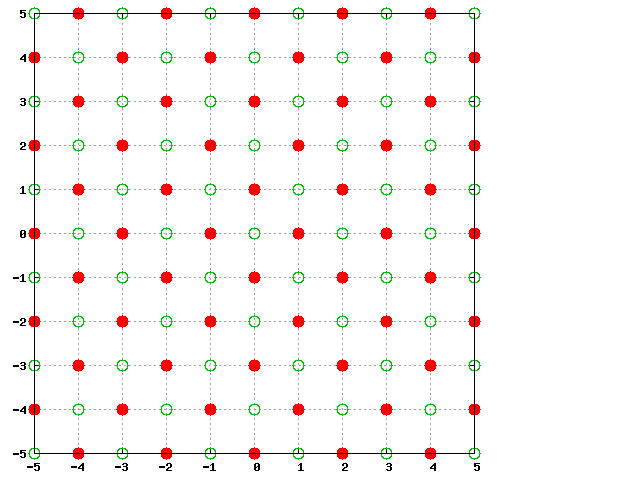}
\center{Figure 1.}
\end{center}
Figure 1 shows the two-dimensional $G_{\rm even}$ by the symbol {\color{green}$\circ$} and  $G_{\rm odd}$ by {\color{red}\textbullet}.  
We sometimes use the notation $(x_{m},t_k),(x_{m+\1},t_{k+1})$ to indicate points of $\G_\delta$ and $(x_{m+\1},t_k),(x_{m},t_{k+1})$ to indicate points of $\tG_\delta$ with $\1:=(1,0,\ldots,0)\in\Z^d$.  For $(x,t)\in \G_\delta\cup \tG_\delta$, the notation $m(x)$, $k(t)$ denotes the index of $x$, $t$, respectively.

For each point  $(x_n,t_{l+1})\in\tG_\delta$, we consider the backward random walks $\gamma$ within $[t_{l'},t_{l+1}]$ which start from $x_n$ at $t_{l+1}$ and move by $\omega\dx$, $\omega\in B$ in each backward time step $\dt$: 
$$\gamma=\{\gamma^k\}_{k=l',\cdots,l+1},\quad\gamma^{l+1}=x_{n},\quad \gamma^{k}=\gamma^{k+1}+\omega\dx.$$
We use the following notation to describe such random walks:  
\begin{eqnarray*}
&&X_n^{l+1,k}:=\Big\{ x_{m+\1} \,\Big|\, \mbox{ $(x_{m+\1},t_{k})\in\tG_\delta$, $\dis \max_{1\le j\le d}|x^j_{m+\1}-x^j_{n}|\le(l+1-k)\dx$}\Big\}\\
 &&\mbox{(the set of all reachable points of the random walk at time $k$)},\\
&&G_n^{l+1,l'}:=\bigcup_{l'\le k\le l+1}\big(X_n^{l+1,k}\times\{t_{k}\}\big)\subset\tG_\delta \\
&&\mbox{(the set of all reachable space-time points of the random walk within $l'\le k\le l+1$)},\\
&&\xi:G_n^{l+1,l'+1}\ni(x_m,t_{k+1})\mapsto\xi^{k+1}_m\in[-(d\lambda)^{-1},(d\lambda)^{-1}]^d,\quad \lambda:=\dt/\dx, \\
&&\rho: G_n^{l+1,l'+1}\times B \ni(x_m,t_{k+1},\omega)\mapsto\rho^{k+1}_m(\omega):=\frac{1}{2d}-\frac{\lambda}{2}(\omega\cdot\xi^{k+1}_m)\in[0,1],\\
&&\gamma:\{ l',l'+1,\ldots,l+1\}\ni k\mapsto \gamma^k\in X_n^{l+1,k},\mbox{ $\gamma^{l+1}=x_n$, $\gamma^{k}=\gamma^{k+1}+\omega\dx$, $\omega\in B$},\\
&&\Omega_n^{l+1,l'}:\mbox{ the family of the above $\gamma$}, 
\end{eqnarray*}
where $\xi$ and $\rho$ are not defined at $l'$. We see that $\rho^{k+1}_m(\omega)$, $\omega\in B$ are the transition probability from $(x_m,t_{k+1})$ to $(x_m+\omega\dx,t_k)$, because  
$$\sum_{\omega\in B}\rho^{k+1}_m(\omega)=\sum_{i=1}^d(\rho^{k+1}_m(e_i)+\rho^{k+1}_m(-e_i))=1.$$
Transition of random walks is controlled  by $\xi$.     
We define the probability density of each path $\gamma\in\Omega_n^{l+1,l'}$ as 
$$\mu_n^{l+1,l'}(\gamma):=\prod_{l'\le k\le l}\rho^{k+1}_{m(\gamma^{k+1})}(\omega^{k+1}),\quad \omega^{k+1}:=\frac{\gamma^k-\gamma^{k+1}}{\dx}.$$    
For each control $\xi$, the probability density $\mu_n^{l+1,l'}(\cdot)=\mu_n^{l+1,l'}(\cdot;\xi)$ yields a probability measure of $\Omega_n^{l+1,l'}$, i.e., 
$${\rm Prob}(A)=\sum_{\gamma\in A}\mu_n^{l+1,l'}(\gamma;\xi)\mbox{\quad for $A\subset\Omega_n^{l+1,l'}$}. $$
The expectation with respect to this probability measure is denoted by $E_{\mu_n^{l+1,l'}(\cdot;\xi)}[\cdot]$, i.e., for a function $f:\Omega_n^{l+1,l'}\to\R$,
$$E_{\mu_n^{l+1,l'}(\cdot;\xi)}[f(\gamma)]:=\sum_{\gamma\in\Omega_n^{l+1,l'}}\mu_n^{l+1,l'}(\gamma;\xi)f(\gamma).$$
In particular,  the average of sample paths $\gamma\in \Omega_n^{l+1,l'}$ with a control $\xi$ is denoted by $\bar{\gamma}^k$, i.e.,
$$\bar{\gamma}^k:=\sum_{\gamma\in\Omega^{l+1,l'}_n}\mu^{l+1,l'}_n(\gamma;\xi)\gamma^k,\quad l'\le k\le l+1.$$ 
As shown in \cite{Soga5}, we have  
\begin{eqnarray}\label{averaged-path}
\bar{\gamma}^{l+1}=x_n,\,\,\,\bar{\gamma}^k=\bar{\gamma}^{k+1}-\bar{\xi}^{k+1}\dt\,\,\,
{\rm \quad with \quad}\bar{\xi}^k:=\sum_{\gamma\in\Omega^{l+1,l'}_n}\mu^{l+1,l'}_n(\gamma;\xi)\xi^k_{m(\gamma^k)},
\end{eqnarray}
\indent  There is  another formulation of the probability measure of the random walk in terms of the configuration space, not the path space $\Omega^{l+1,l'}_n$, i.e., the distribution on  $X^{l+1,k}_n$ for each $l'\le k\le l+1$.  Define  $p(\xi):G^{l+1,l'}_n\ni(x_{m+\1},t_k)\mapsto p^{k}_{m+\1}(\xi)\in[0,1]$ as 
\begin{eqnarray}\label{prob-on-X} 
 p^{k}_{m+\1}(\xi):={\rm Prob}(\{\gamma\in \Omega^{l+1,l'}_{n}\,|\,\gamma^{k}=x_{m+\1}\}).
\end{eqnarray}
It follows from the definition of random walks that $p^{k}_{m+\1}$ is independent from the choice of $l'$  and 
\begin{eqnarray*}
 \sum_{\{m\,|\,x_{m+\1}\in X^{l+1,k}_{n} \}} p^k_{m+\1} (\xi) =1\quad \mbox{ for each $k$}.
\end{eqnarray*}
Furthermore, it holds that 
\begin{eqnarray}\label{p-evolution}
p^{k}_{m+\1}(\xi)=\sum_{\omega\in B}p^{k+1}_{m+\1+\omega}(\xi)\rho^{k+1}_{m+\1+\omega}(-\omega),
\end{eqnarray} 
where $p^{k+1}_{m+\1+\omega}(\xi)=\rho^{k+1}_{m+\1+\omega}(-\omega)=0$ if $x_{m+\1+\omega}\not\in X^{l+1,k+1}_n$.  We will see  in Section 3 that  $p(\xi)$ plays an important role to derive an analogue of Mather's minimizing problem and the construction of Mather measures.   

Since our transition probabilities are space-time inhomogeneous, the well-known law of large numbers does not always hold in hyperbolic scaling limit, i.e., $\delta=(h,\tau)\to0$ under $0<\lambda_0\le \lambda:=\tau/h$ with a constant $\lambda_0$. The author investigated the asymptotics  of the probability measure of $\Omega_n^{l+1,l'}$ as $\delta\to0$ in  \cite{Soga1},  \cite{Soga5} as follows:   Let $\eta(\gamma):\{l',l'+1,\ldots,l+1\}\to\R^d$ be a function defined for each $\gamma\in\Omega^{l+1,l'}_n$ as
$$\eta^k(\gamma)=\eta^{k+1}(\gamma)-\xi^{k+1}_{m(\gamma^{k+1})}\dt,\,\,\,\eta^{l+1}(\gamma)=x_n.$$
Define  $\tilde{\sigma}^{l+1,k}_i$ and $\hat{\sigma}^{l+1,k}_i$ for $i=1,\ldots,d$ as 
$$\tilde{\sigma}^{l+1,k}_i:=E_{\mu^{l+1,l'}_n(\cdot;\xi)}[|(\eta^k(\gamma)-\gamma^k)^i|^2],\quad \hat{\sigma}^{l+1,k}_i:=E_{\mu^{l+1,l'}_n(\cdot;\xi)}[|(\eta^k(\gamma)-\gamma^k)^i|],$$
where $(\eta^k(\gamma)-\gamma^k)^i$ denotes the $i$-th component of $\eta^k(\gamma)-\gamma^k$. 
\begin{Lemma}[ \cite{Soga1}, \cite{Soga5}]\label{limit-theorem} 
For any control $\xi$, we have 
$$(\hat{\sigma}^{l+1,k}_i)^2\le \tilde{\sigma}^{l+1,k}_i\le (t_{l+1}-t_k)\frac{\dx}{\lambda}\quad{\rm for}\quad  0\le k\le l+1.$$
\end{Lemma}
\noindent Note that $\tilde{\sigma}^{l+1,k}_i$ can be seen as a generalization of the standard variance; the standard variance is of $O(1)$ as $\delta\to0$ under hyperbolic scaling in general for space-time  inhomogeneous random walks; however, $\tilde{\sigma}^{l+1,k}_i$ and $\hat{\sigma}^{l+1,k}_i$ always tend to $0$ for any control $\xi$; in the space-homogeneous case, i.e., $\xi$ is constant for each $k$, $\tilde{\sigma}^{l+1,k}_i$ is equal to the standard variance. 
The above hyperbolic scaling limit of the random walks plays an important role to investigate convergence of our theory. 

Forward random walks are defined in the same manner with $\gamma^{l+1}=x_n$, $\gamma^{k+1}=\gamma^k+\omega h$ for $k=l+1,\ldots,l'-1$ and $\rho^{k}_{m+\1}(\omega):=\frac{1}{2d}+\frac{\lambda}{2}(\omega\cdot\xi^{k}_{m+\1})$. 
\subsection{Hamilton-Jacobi equation on grid}

Let $v$ denote a function: $\tG_\delta\ni(x_{m+\1},t_k)\mapsto v(x_{m+\1},t_k)=v^k_{m+\1}\in\R$. Introduce the spatial discrete derivatives of $v$ that are defined at each point $(x_m,t_k)\in\G_\delta$ as 
\begin{eqnarray*}
&&(D_{x^j}v)(x_m,t_k)=(D_{x^j}v)^k_{m}:=\frac{v^k_{m+e_j}-v^k_{m-e_j}}{2h},\\
&&(D_{x}v)(x_m,t_k)=(D_{x}v)^k_{m}:=\big((D_{x^1}v)^k_{m},\ldots,(D_{x^d}v)^k_{m}\big).
\end{eqnarray*}
Introduce the temporal discrete  derivative of $v$ that is defined at each point $(x_m,t_{k+1})\in\tG_\delta$ as 
$$(D_t v)(x_m,t_{k+1})=(D_t v)^{k+1}_m:=\left( v^{k+1}_{m}-\frac{1}{2d}\sum_{\omega\in B}v^k_{m+\omega} \right)\frac{1}{\tau}.$$
\indent Let $P\subset\R^d$ be an arbitrary convex and compact set. Fix any $r>0$. For each fixed $c\in P$, consider the initial value problems of the Hamilton-Jacobi equation on the grid  
\begin{eqnarray}\label{HJ-delta}
\left\{
\begin{array}{lll}
&v:\tG_\delta|_{0\le k\le 2K}\ni(x_{m+\1},t_k)\mapsto v^k_{m+\1}\in\R,\medskip \\
&v^k_{m+\1\pm2Ne_i}=v^k_{m+\1}\quad (i=1,\ldots,d),\medskip\\
&v(\cdot,0)=v^0:G_{\rm odd}\ni x_{m+\1}\mapsto v^0_{m+\1}\in\R\mbox{ is given so that}\medskip \\
&v^0_{m+\1\pm2Ne_i}=v^0_{m+\1}\mbox{\, and\, } |D_{x^i}v^0|\le r\,\,\,\,\,\, (i=1,\ldots,d),\medskip\\
&(D_tv)^{k+1}_m+H(x_m,t_k,c+(D_xv)^k_m)=0,
\end{array}
\right.
\end{eqnarray}
which corresponds to the initial value problems of the exact Hamilton-Jacobi equations \eqref{-} with initial data from Lip$_r(\T^d;\R)$. In (\ref{HJ-delta}), the quantity $v^{k+1}_{m}$ is unknown to be determined by $\{v^k_{m+\omega}\}_{\omega\in B}$ as a recursion. As explained in \cite{Soga5}, evolution of (\ref{HJ-delta}) can be intuitively seen as Figure 2: The value $v^{k+1}_{m}$ is determined by the values of the grid points {\color{red}\textbullet} contained in the pyramid  in the figure, where the pyramid grows up to $k=0$ keeping the aspect ratio determined by $\lambda:=\dt/\dx$ ({\it a finite speed of propagation}). 
We call the pyramid ``a pyramid of dependence''.  
The key point is that the contribution of the value at each grid point {\color{red}\textbullet} within the pyramid of dependence to the value $v^{k+1}_{m}$ can be characterized by the probability measure of a controlled backward random walk starting at $(x_m,t_{k+1})$. 
\vspace{-20mm}
\begin{center}
\includegraphics[bb=0 0 670 600,width=9cm]{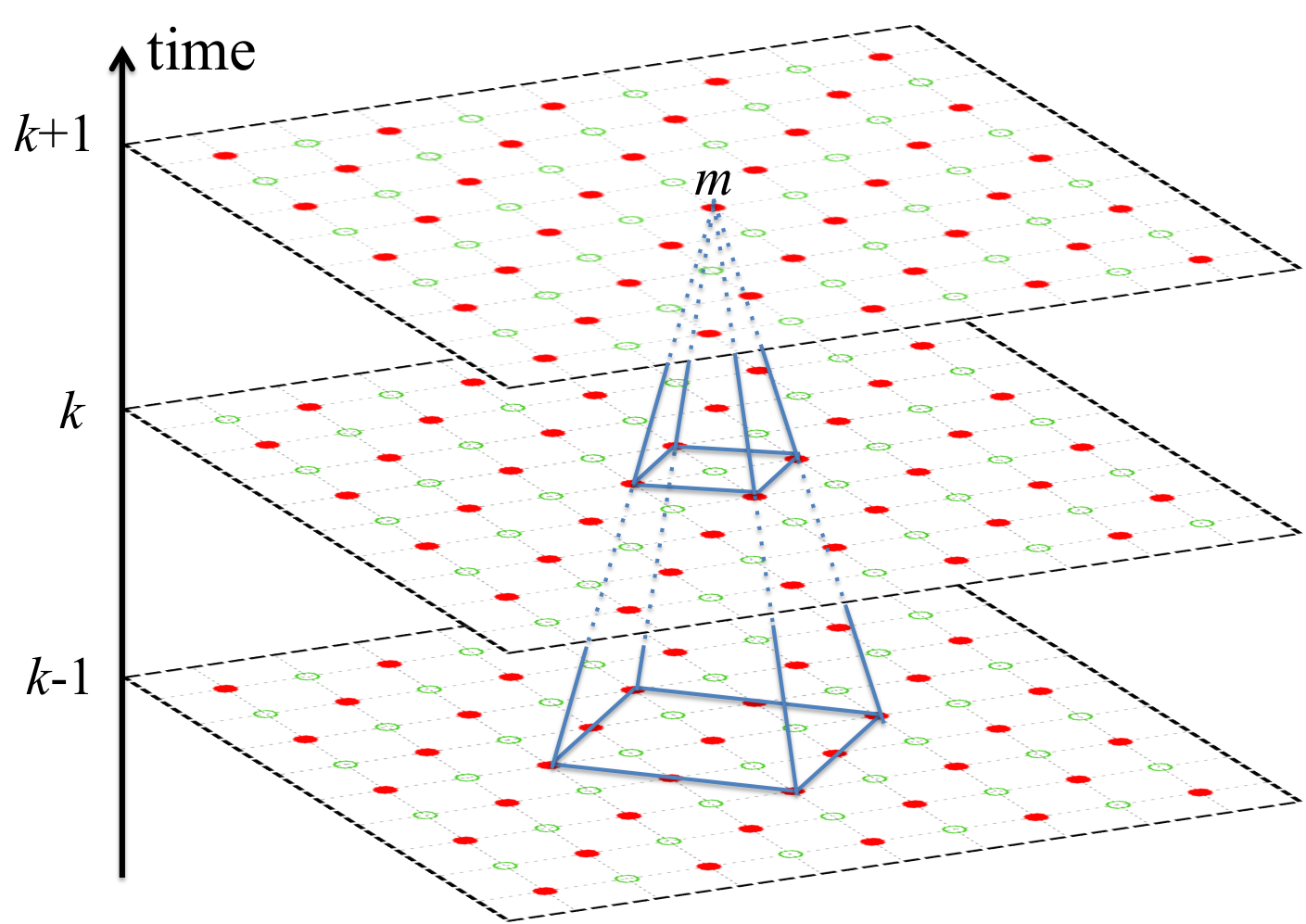}
\center{Figure 2.}
\end{center}
The Hamilton-Jacobi equation on the grid corresponding to \eqref{+} is given as 
\begin{eqnarray}\label{HJ-delta+}
\left\{
\begin{array}{lll}
&v:\tG_\delta|_{-2K\le k\le 0}\ni(x_{m+\1},t_k)\mapsto v^k_{m+\1}\in\R,\medskip \\
&v^k_{m+\1\pm2Ne_i}=v^k_{m+\1}\quad (i=1,\ldots,d),\medskip\\
&v(\cdot,0)=v^0:G_{\rm odd}\ni x_{m+\1}\mapsto v^0_{m+\1}\in\R\mbox{ is given so that}\medskip \\
&v^0_{m+\1\pm2Ne_i}=v^0_{m+\1},\,\,\, |D_{x^i}v^0|\le r\,\,\,\,\,\, (i=1,\ldots,d),\medskip\\
&(\tilde{D}_tv)^{k-1}_m+H(x_m,t_k,c+(D_xv)^k_m)=0
\end{array}
\right.
\end{eqnarray}
with 
$$(\tilde{D}_tv)^{k-1}_m:=   \left( v^{k-1}_{m}-\frac{1}{2d}\sum_{\omega\in B}v^k_{m+\omega} \right)\frac{1}{-\tau}.$$
\subsection{Lax-Oleinik type solution map}
 
We recall the results \cite{Soga5} on solvability of (\ref{HJ-delta}) in terms of the Lax-Oleinik type representation formula with the action functional \eqref{00002}. To be more precise than $\mathcal{L}^l_-$, we use the following notation: define the action functional for each $v^0:G_{\rm odd}\to\R$ as 
$$E^{l+1}_n(\xi;v^0,c):=E_{\mu_n^{l+1,0}(\cdot;\xi)}\Big[ \sum_{0<k\le l+1}L^{(c)}(\gamma^k,t_{k-1},\xi^k_{m(\gamma^k)})\dt+v_{m(\gamma^0)}^0 \Big].$$
\begin{Thm}[\cite{Soga5}]\label{main1}
For each  $r>0$ and the set $P$ (the set of $c$), there exists $\lambda_1>0$ for which the following statements hold for any small $\delta=(\dx,\dt)$ with $\lambda:=\dt/\dx<\lambda_1$, any $c\in P$ and any  initial data $v(\cdot,0)=v^0$ of \eqref{HJ-delta}:
\begin{enumerate} 
\item For each $n$ and $l$ with $0<l+1\le 2K$ such that $(x_n,t_{l+1})\in\tG_\delta$, the action functional $E^{l+1}_n(\xi;v^0,c)$ has the infimum within all controls $\xi:G_n^{l+1,1}\to[-(d\lambda)^{-1},(d\lambda^{-1})]^d$. There exists the unique minimizing control $\xi^\ast$ to attain the infimum, which satisfies 
$$\mbox{$| \xi^\ast{}^j|_\infty\le (d\lambda_1)^{-1}<(d\lambda)^{-1}$ on $G_n^{l+1,1}$ for all $1\le j\le d$.}$$
\item Define the function $v:\tG_\delta|_{0\le k\le 2K}\to\R$ as 
\begin{eqnarray}\label{nikoinko}
v(x_m,t_{k+1}):=\inf_\xi E^{k+1}_m(\xi;v^0,c),\quad v(x_{m+\1},0):=v_{m+\1}^0.
\end{eqnarray}
 Then, the minimizing control $\xi^\ast$ for $\inf_\xi E^{l+1}_n(\xi;v^0,c)$ satisfies
$$\xi^\ast{}^{k+1}_m=H_p(x_m,t_k,c+(D_xv)^k_m)\,\,\,\,(\Leftrightarrow\,\,(D_xv)^k_m=L_\zeta(x_m,t_k,\xi^\ast{}^{k+1}_m)-c).$$
In particular, $(D_xv)^k_m$ is uniformly bounded on $\tG_\delta|_{0\le k\le 2K}$ independently from $\delta$ (this is a CFL-type condition).
\item The function $v$ defined in the claim 2 is the unique solution of (\ref{HJ-delta}).
\end{enumerate}
\end{Thm}
\noindent Throughout this paper, $\lambda_1$ stands for the constant mentioned in Theorem \ref{main1}. 

The following families of the maps are well-defined as the Lax-Oleinik type solution maps for \eqref{HJ-delta}: for any $v^0$ given in \eqref{HJ-delta},  
$$\{\varphi^k_\delta(\cdot;c)\}_{k\in\N\cup\{0\}},\quad \varphi^k_\delta(\cdot;c): v^0\mapsto v(\cdot,t_k) \quad (\mbox{$v$ is given as \eqref{nikoinko}}).$$
In addition,  we set 
$$\{\psi^k_\delta(\cdot;c)\}_{k\in\N\cup\{0\}},\quad  \psi^k_\delta(\cdot;c): u^0=D_xv^0\mapsto u^k=D_xv(\cdot,t_k)\quad (\mbox{$v$ is given as \eqref{nikoinko}}),$$
where $\psi^k_\delta$ is indeed the solution map of the system of discrete conservation laws with restricted initial data derived from  \eqref{HJ-delta}: 
$u^k_m=(u^1{}^k_m,u^2{}^k_m,\ldots,u^d{}^k_m):=D_xv(\cdot,t_k):\G\to\R^d$ satisfies for  $i=1,2,\ldots,d$,
$$\frac{1}{\tau}\Big(u^i{}^{k+1}_{m+e_i} -\frac{1}{2d}\sum_{\omega\in B} u^i{}^{k}_{m+e_i+\omega} \Big)+\frac{1}{2h}\{H(x_{m+2e_i},t_k,c+u^k_{m+2e_i})-H(x_{m},t_k,c+u^k_m)\}=0.$$
\indent In Section 3, we will seek for a pair of constant $\bar{H}_\delta(c)$ and a function $\bar{v}^0$ such that $\varphi^{2K}_\delta(\bar{v}^0;c)+\bar{H}_\delta(c)=\bar{v}^0$. The pair yields a time-$1$-periodic solution $\bar{v}$ of 
$$(D_tv)^{k+1}_m+H(x_m,t_k,c+(D_xv)^k_m)=\bar{H}_\delta(c).$$
 Then, with $1$-periodic extension of $\bar{v}$ to the whole $\tG_\delta$, we complete the set up of the action minimizing problem for $\mathcal{L}^{l}_-(\xi;\bar{v})$. For this purpose,  we need more preliminary investigations. 

As for \eqref{HJ-delta+}, we have 
\begin{eqnarray*}
E^{-l-1}_n(\xi;c)&:=&E_{\mu_n^{-l-1,0}(\cdot;\xi)}\Big[ \sum_{-l-1\le k<0}-L^{(c)}(\gamma^k,t_{k+1},\xi^k_{m(\gamma^k)})\dt+v_{m(\gamma^0)}^0 \Big],\\
v(x_m,t_{-k-1})&=&\sup_{\xi}E^{-k-1}_m(\xi;c),
\end{eqnarray*}
where the solution map is denoted by $\tilde{\varphi}^k_\delta(\cdot;c)$ and $\tilde{\psi}^k_\delta(\cdot;c)$.
\subsection{Semiconcavity of Lax-Oleinik type solution map}

Due to the variational structure of the solution $v^k_{m+\1}=\varphi^k_{\delta}(v^0;c)(x_{m+\1})$ to (\ref{HJ-delta}), we have a kind of semiconcavity property, i.e.,  
$$(D^2_jv)^{k}_{m+\1}:=(v^k_{m+\1+2e_j}+v^k_{m+\1-2e_j}-2v^k_{m+\1})\frac{1}{4h^2}=\frac{(D_{x^j}v)^k_{m+\1+e_j}-(D_{x^j}v)^k_{m+\1-e_j}}{2h}$$
is bounded from the above.  An instant observation shows that the upper bound is given by ``$\sup_m (D^2_iv^0)_{m+\1}+$ [some positive constant independent from $\delta$ and $v^0$]''. Our aim is to obtain a sharper estimate independent from the initial data $v^0$. Note that our ``semiconcavity'' estimate is restricted to the directions of $e_1,\ldots,e_d$, not every direction of $\R^d$.  The sharper semiconcavity estimate implies a lot in regards to the behaviors of the derivative of discrete solutions.        

 Introduce the following notation: 
\begin{eqnarray*}
&&M^k_\delta:=\sup_{m,j} (D^2_jv)^{k}_{m+\1},\\
&&u^\ast:=\max_{x\in\T^d,t\in\T, |\zeta|_\infty\le(d\lambda_1)^{-1},c\in P}| L^{(c)}_\zeta(x,t,\zeta)|_\infty \\&&\mbox{(note that $| (D_xv)^k_{m+1}|_\infty\le u^\ast$ for any solution of \eqref{HJ-delta})},\\
&&H^\ast_p:=\max_{x\in\T^d,t\in\T, |u|_\infty\le u^\ast,c\in P}|H_{p}(x,t,c+u)|_\infty,\\
&& H_{xx}^\ast:=\max_{x\in\T^d,t\in\T, |u|_\infty\le u^\ast,c\in P,i,j}|H_{x^ix^j}(x,t,c+u)|,\\
&& H_{xp}^\ast:=\max_{x\in\T^d,t\in\T,|u|_\infty\le u^\ast,c\in P,i,j}|H_{x^ip^j}(x,t,c+u)|,\\
&& H_{pp}^\ast:=\inf_{x\in\T^d,t\in\T, |u|_\infty\le u^\ast,c\in P,y\in\R^d}\frac{H_{pp}(x,t,c+u)y\cdot y}{|y|^2},\quad\mbox{where $H_{pp}^\ast>0$ due to (H2)},\\
&&M^\ast_\pm:=\frac{H^\ast_{xp}\pm\sqrt{(1+d)(H^\ast_{xp})^2+H^\ast_{pp}H^\ast_{xx}}}{H^\ast_{pp}},\\
&&\eta^\ast:=M^\ast_+-M^\ast_-,\\
&&M(t):=M^\ast_++\frac{\eta^\ast e^{-\eta^\ast H^\ast_{pp}t} }{1-e^{-\eta^\ast H^\ast_{pp}t}},\quad t>0,\quad\mbox{where } M(t)\to M^\ast_+\mbox{ as $t\to\infty$}.
\end{eqnarray*}
\begin{Thm}\label{semi}
Suppose that $\delta=(h,\tau)$ with $\lambda:=\tau/h<\lambda_1$ is such that 
\begin{eqnarray*}
&&\lambda\le \min\Big\{
\frac{1-2d H^\ast_{xp}\tau}{2drH^\ast_{pp}+dH^\ast_p},\,\,\,\frac{1}{10rH^\ast_{pp}} \Big\},\\
&&\tau<\min\Big\{ 
\frac{1}{2dH^\ast_{xp}},\,\,\,  \frac{1-d\lambda H^\ast_p}{2d (H^\ast_{pp}M^\ast_++H^\ast_{xp})}, \,\,\,
\frac{1}{H^\ast_{pp}(M^\ast_+-M^\ast_-)},\,\,\,
\frac{\log2}{\eta^\ast H^\ast_{pp}},\,\,\,\\
&&\qquad\qquad\qquad\qquad\qquad\qquad\qquad\qquad\qquad\qquad \frac{1}{4\sqrt{(1+d)(H_{xp}^\ast)^2+H_{pp}^\ast H_{xx}^\ast}}
\Big\}.
\end{eqnarray*}
Then, we have $M^k_\delta\le M(t_k)$ for all $k=1,\ldots,2K$. In particular,  if $M^0_\delta\le M^\ast_+$, we have $M^k_\delta\le M^\ast_+$ for all $0\le k\le 2K$;  if $v^k$ is extended to $k\to\infty$ keeping the boundedness $|(D_xv)^k_{m}|_\infty\le u^\ast$, we have $M^k_\delta\le M(t_k)$ for all $k>0$. 
\end{Thm}
\begin{proof}
We insert 
\begin{eqnarray*}
&&v^{k+1}_m=\frac{1}{2d}\sum_{\omega\in B} v^k_{m+\omega}-H(x_m,t_k,c+(Dv)^k_m)\tau
\end{eqnarray*}
into $(D_j^2v)^{k+1}_m$ and apply Taylor's formula with short notation $H_{pp}$, $H_{xx}$, $H_{xp}$, etc., for the remainder terms, to get   
\begin{eqnarray*}
&&(D_j^2v)^{k+1}_m\cdot4h^2=\frac{1}{2d}\sum_{\omega\in B}\Big(v^k_{m+2e_j+\omega}+v^k_{m-2e_j+\omega}-2v^k_{m+\omega}\Big)\\
&&\quad- \Big\{H(x_{m+2e_j},t_k,c+(D_xv)^k_{m+2e_j})+H(x_{m-2e_j},t_k,c+(D_xv)^k_{m-2e_j})\\
&&\quad-2H(x_{m},t_k,c+(D_xv)^k_{m}) \Big\}\tau\\
&&=\frac{1}{2d}\sum_{\omega\in B} (D_j^2v)^k_{m+\omega}\cdot4h^2 \\
&&\quad -\Big\{   
H(x_{m+2e_j},t_k,c+(D_xv)^k_{m+2e_j})-H(x_{m},t_k,c+(D_xv)^k_{m+2e_j})\\
&&\quad +H(x_{m},t_k,c+(D_xv)^k_{m+2e_j})-H(x_{m},t_k,c+(D_xv)^k_{m})\\
&&\quad+H(x_{m-2e_j},t_k,c+(D_xv)^k_{m-2e_j})-H(x_{m},t_k,c+(D_xv)^k_{m-2e_j})\\
&&\quad +H(x_{m},t_k,c+(D_xv)^k_{m-2e_j})-H(x_{m},t_k,c+(D_xv)^k_{m})  \Big\}\tau\\
&&=\frac{1}{2d}\sum_{i=1}^d \Big( (D_{x^j}^2v)^k_{m+e_i}+(D_{x^j}^2v)^k_{m-e_i} \Big)\cdot4h^2 \\
&&\quad -\Big\{   
H_x(x_{m},t_k,c+(D_xv)^k_{m+2e_j})\cdot(2he_j)+\frac{1}{2}H_{xx}\times(2he_j)\cdot(2he_j)\\
&&\quad +H_p(x_{m},t_k,c+(D_xv)^k_{m})\cdot((D_xv)^k_{m+2e_j}-(D_xv)^k_{m})\\
&&\quad+\frac{1}{2}H_{pp}\times((D_xv)^k_{m+2e_j}-(D_xv)^k_{m})\cdot((D_xv)^k_{m+2e_j}-(D_xv)^k_{m}) \\
&&\quad +H_x(x_{m},t_k,c+(D_xv)^k_{m-2e_j})\cdot(-2he_j)+\frac{1}{2}H_{xx}\times(-2he_j)\cdot(-2he_j)\\
&&\quad +H_p(x_{m},t_k,c+(D_xv)^k_{m})\cdot((D_xv)^k_{m-2e_j}-(D_xv)^k_{m})\\
&&\quad+\frac{1}{2}H_{pp}\times((D_xv)^k_{m-2e_j}-(D_xv)^k_{m})\cdot((D_xv)^k_{m-2e_j}-(D_xv)^k_{m}) \Big\}\tau\\
&&=\frac{2h^2}{d}\sum_{i=1}^d \Big( (D_{x^j}^2v)^k_{m+e_i}+(D_{x^j}^2v)^k_{m-e_i} \Big) \\
&&\quad -\Big\{  
\sum_{i=1}^d H_{p^i}(x_{m},t_k,c+(D_xv)^k_{m}) 
\Big( \frac{v^k_{m+2e_j+e_i}-v^k_{m+2e_j-e_i}}{2h}
+\frac{v^k_{m-2e_j+e_i}-v^k_{m-2e_j-e_i}}{2h}\\
&&\quad -2\frac{v^k_{m+e_i}-v^k_{m-e_i}}{2h}\Big)
+ H_{x^jp}\cdot\Big( 
(D_xv)^k_{m+2e_j}-(D_xv)^k_{m}+(D_xv)^k_{m}-(D_xv)^k_{m-2e_j}
\Big)2h \\
&&\quad+\frac{1}{2}H_{pp}\times((D_xv)^k_{m+2e_j}-(D_xv)^k_{m})\cdot((D_xv)^k_{m+2e_j}-(D_xv)^k_{m}) \\
&&\quad+\frac{1}{2}H_{pp}\times((D_xv)^k_{m-2e_j}-(D_xv)^k_{m})\cdot((D_xv)^k_{m-2e_j}-(D_xv)^k_{m})\\
&&\quad +\frac{1}{2}H_{x^jx^j}\times4h^2+\frac{1}{2}H_{x^jx^j}\times4h^2\Big\}\tau.
\end{eqnarray*}
Since
\begin{eqnarray*}
&&\frac{v^k_{m+2e_j+e_i}-v^k_{m+2e_j-e_i}}{2h}
+\frac{v^k_{m-2e_j+e_i}-v^k_{m-2e_j-e_i}}{2h} -2\frac{v^k_{m+e_i}-v^k_{m-e_i}}{2h}\\
&&\qquad=\frac{2h^2}{d}\Big((D^2_jv)^k_{m+e_i}\cdot\frac{d}{h}-(D^2_jv)^k_{m-e_i}\cdot\frac{d}{h} \Big),
\end{eqnarray*}
we have with $\lambda=\tau/h$, 
\begin{eqnarray*}
&&(D_{x^j}^2v)^{k+1}_m\cdot4h^2=\frac{2h^2}{d}\sum_{i=1}^d\Big\{ 
(D^2_jv)^k_{m+e_i}\Big( 1-d\lambda H_{p^j}(x_m,t_k,c+(D_xv)^k_m)  \Big) \\
&&\quad + (D^2_jv)^k_{m-e_i}\Big( 1+d\lambda H_{p^j}(x_m,t_k,c+(D_xv)^k_m)  \Big)
\Big\}\\
&&\quad -\Big\{ H_{x^jp}\cdot\Big( 
(D_xv)^k_{m+2e_j}-(D_xv)^k_{m}+(D_xv)^k_{m}-(D_xv)^k_{m-2e_j}
\Big)2h \\
&&\quad+\frac{1}{2}H_{pp}\times((D_xv)^k_{m+2e_j}-(D_xv)^k_{m})\cdot((D_xv)^k_{m+2e_j}-(D_xv)^k_{m}) \\
&&\quad+\frac{1}{2}H_{pp}\times((D_xv)^k_{m-2e_j}-(D_xv)^k_{m})\cdot((D_xv)^k_{m-2e_j}-(D_xv)^k_{m})\\
&&\quad +\frac{1}{2}H_{x^jx^j}\times4h^2+\frac{1}{2}H_{x^jx^j}\times4h^2\Big\}\tau\\
&&\le \frac{2h^2}{d}\sum_{i=1}^d\Big\{ 
(D^2_jv)^k_{m+e_i}\Big( 1-d\lambda H_{p^j}(x_m,t_k,c+(D_xv)^k_m)  \Big) \\
&&\quad + (D^2_jv)^k_{m-e_i}\Big( 1+d\lambda H_{p^j}(x_m,t_k,c+(D_xv)^k_m)  \Big)
\Big\}+ 4h^2\tau H^\ast_{xx} \\
&&\quad -\Big\{ 
H_{x^jp}\cdot\Big( 
(D_xv)^k_{m+2e_j}-(D_xv)^k_{m}+(D_xv)^k_{m}-(D_xv)^k_{m-2e_j}
\Big)2h \\
&&\quad+\frac{H^\ast_{pp}}{2}((D_xv)^k_{m+2e_j}-(D_xv)^k_{m})\cdot((D_xv)^k_{m+2e_j}-(D_xv)^k_{m}) \\
&&\quad+\frac{H^\ast_{pp}}{2}((D_xv)^k_{m-2e_j}-(D_xv)^k_{m})\cdot((D_xv)^k_{m-2e_j}-(D_xv)^k_{m})
\Big\}\tau\\
&&=\frac{2h^2}{d}\sum_{i=1}^d\Big\{ 
(D^2_jv)^k_{m+e_i}\Big( 1-d\lambda H_{p^j}(x_m,t_k,c+(D_xv)^k_m)  \Big) \\
&&\quad + (D^2_jv)^k_{m-e_i}\Big( 1+d\lambda H_{p^j}(x_m,t_k,c+(D_xv)^k_m)  \Big)
\Big\}+ 4h^2\tau H^\ast_{xx} \\
&&\quad -\Big\{ 
\frac{H^\ast_{pp}}{2}\Big|\Big((D_xv)^k_{m+2e_j}-(D_xv)^k_{m}\Big)+\frac{2h}{H^\ast_{pp}}H_{x^jp}\Big|^2-\frac{2h^2}{H^\ast_{pp}}|H_{x^jp}|^2 \\
&&\quad+\frac{H^\ast_{pp}}{2}\Big|\Big((D_xv)^k_{m}-(D_xv)^k_{m-2e_j}\Big)+\frac{2h}{H^\ast_{pp}}H_{x^jp}\Big|^2-\frac{2h^2}{H^\ast_{pp}}|H_{x^jp}|^2
\Big\}\tau\\
&&\!\!\!\!\!\!\!\mbox{(since $|((D_xv)^k_{m+2e_j}-(D_xv)^k_{m})+\frac{2h}{H^\ast_{pp}}H_{x^jp}|^2\ge 
\{((D_{x^j}v)^k_{m+2e_j}-(D_{x^j}v)^k_{m})+\frac{2h}{H^\ast_{pp}}H_{x^jp^j}\}^2$)}\\
&&\le \frac{2h^2}{d}\sum_{i=1}^d\Big\{ 
(D^2_jv)^k_{m+e_i}\Big( 1-d\lambda H_{p^j}(x_m,t_k,c+(D_xv)^k_m)  \Big) \\
&&\quad + (D^2_jv)^k_{m-e_i}\Big( 1+d\lambda H_{p^j}(x_m,t_k,c+(D_xv)^k_m)  \Big)
\Big\}+ 4h^2\tau H^\ast_{xx} +\frac{4h^2\tau dH^\ast_{xp}{}^2}{H_{pp}^\ast} \\
&&\quad - \frac{\tau H^\ast_{pp}}{2}\Big\{ 
\Big((D_{x^j}v)^k_{m+2e_j}-(D_{x^j}v)^k_{m}+\frac{2h}{H^\ast_{pp}}H_{x^jp^j}\Big)^2\\
&&\quad + \Big((D_{x^j}v)^k_{m}-(D_{x^j}v)^k_{m-2e_j}+\frac{2h}{H^\ast_{pp}}H_{x^jp^j}\Big)^2
\Big\}\\
&&=\frac{4h^2}{2d}\sum_{i=1}^d\Big\{ 
(D^2_jv)^k_{m+e_i}\Big( 1-d\lambda H_{p^j}(x_m,t_k,c+(D_xv)^k_m)  \Big) \\
&&\quad + (D^2_jv)^k_{m-e_i}\Big( 1+d\lambda H_{p^j}(x_m,t_k,c+(D_xv)^k_m)  \Big)
\Big\}+ 4h^2\tau H^\ast_{xx} +\frac{4h^2\tau dH^\ast_{xp}{}^2}{H_{pp}^\ast} \\
&&\quad - \frac{4h^2\tau}{2} H^\ast_{pp}\Big\{ 
\Big((D_{x^j}^2v)^k_{m+e_j}+\frac{H_{x^jp^j}}{H^\ast_{pp}}\Big)^2 + \Big((D_{x^j}^2v)^k_{m-e_j}+\frac{H_{x^jp^j}}{H^\ast_{pp}}\Big)^2
\Big\}.
\end{eqnarray*}
Set $g_\pm(y):\R\to\R$ as  
\begin{eqnarray*}
g_\pm(y):=\frac{1}{2d}\Big(1\pm d\lambda H_{p^j}(x_m,t_k,c+(D_xv)^k_m)\Big)y-\frac{\tau H^\ast_{pp}}{2}\Big(y+\frac{H_{x^jp^j}}{H^\ast_{pp}}\Big)^2.
\end{eqnarray*}
We see that $g'_\pm(y)\ge0$, if 
\begin{eqnarray*}
y\le  \frac{1-d\lambda H^\ast_p}{2d\tau H^\ast_{pp}}-\frac{H^\ast_{xp}}{H^\ast_{pp}}\quad\Big(  \le \frac{1\pm d\lambda H_{p^j}(x_m,t_k,c+(D_xv)^k_m)}{2d\tau H^\ast_{pp}}-\frac{H_{x^jp^j}}{H^\ast_{pp}} \Big).
\end{eqnarray*}
Since $\lambda\le(1-2d H^\ast_{xp}\tau)/(2drH^\ast_{pp}+dH^\ast_p)$, we have for all initial data $v^0$ in \eqref{HJ-delta},
\begin{eqnarray*}
M^0_\delta =\sup_{m,j}(D^2_jv^0)_{m+\1}\le \frac{r}{h}=\frac{r\lambda}{ \tau}
\le \frac{1-3d H^\ast_{xp}\tau}{2d\tau H^\ast_{pp}}    
= \frac{1-d\lambda H^\ast_p}{2d\tau H^\ast_{pp}}-\frac{H^\ast_{xp}}{H^\ast_{pp}}.
\end{eqnarray*}
Suppose that for some $k\ge0$,
\begin{eqnarray}\label{semi3}
M^k_\delta=\sup_{m,j} (D^2_jv)^k_{m+\1}\le \frac{1-d\lambda H^\ast_p}{2d\tau H^\ast_{pp}}-\frac{H^\ast_{xp}}{H^\ast_{pp}}.
\end{eqnarray}
 Then, we have 
\begin{eqnarray*}
&&(D_j^2v)^{k+1}_m\le \frac{1}{2d}\sum_{i\in\{1,\ldots,d\}\setminus\{j\}}
\Big\{ 
(D^2_jv)^k_{m+e_i}\Big( 1-d\lambda H_{p^j}(x_m,t_k,c+(D_xv)^k_m)  \Big) \\
&&\quad + (D^2_jv)^k_{m-e_i}\Big( 1+d\lambda H_{p^j}(x_m,t_k,c+(D_xv)^k_m)  \Big)
\Big\}+ \tau H^\ast_{xx} +\frac{\tau dH^\ast_{xp}{}^2}{H_{pp}^\ast} \\
&&\quad +g_-(M^k_\delta)+g_+(M^k_\delta). 
\end{eqnarray*}
Since $1\pm d\lambda H_{p^j}(x_m,t_k,c+(D_xv)^k_m>0$ due to the CFL-type condition given in Theorem \ref{main1}, we have 
\begin{eqnarray*}
&&(D_j^2v)^{k+1}_m\le M^k_\delta+\tau H^\ast_{xx} +\frac{\tau dH^\ast_{xp}{}^2}{H_{pp}^\ast}
-\tau H^\ast_{pp} (M^k_\delta)^2+2\tau H^\ast_{xp} M^k_\delta, 
\end{eqnarray*}
and hence,
\begin{eqnarray}\label{semi4}
&&M^{k+1}_\delta\le M^k_\delta+ \tau G(M^k_\delta),\quad G(y):=-\Big(H^\ast_{pp}y^2-2H^\ast_{xp}y-H^\ast_{xx}-\frac{d H^\ast_{xp}{}^2}{H^\ast_{pp}}\Big), 
\end{eqnarray}
where $G(y)=-H^\ast_{pp}(y-M^\ast_+)(y-M^\ast_-)$, $G(y)>0$ for $0\le y<M^\ast_+$ and $G(y)<0$ for $y>M^\ast_+$. Since $\tau\le  (1-d\lambda H^\ast_p)/\{2d (H^\ast_{pp}M^\ast_++H^\ast_{xp})\}$ and $\tau\le1/\{H_{pp}^\ast(M_+^\ast - M^\ast_-)\}$, we have 
\begin{eqnarray*}
M^\ast_+< \frac{1-d\lambda H^\ast_p}{2d\tau H^\ast_{pp}}-\frac{H^\ast_{xp}}{H^\ast_{pp}},
\quad \tau G(y)\le M^\ast_+-y \mbox{\quad for all $0\le y\le M^\ast_+$.}
\end{eqnarray*}
Therefore, the following two cases happen:
\begin{itemize}
\item[(i)] If $M^k_\delta\le M^\ast$, we may have $M^{k+1}_\delta\ge M^k_\delta$, but we certainly have $M^{k+1}_\delta\le M^\ast_+$. 
\item[(ii)] If $M^k_\delta>M^\ast_+$, we have $M^{k+1}_\delta<M^k_\delta$.   
\end{itemize}
In  both cases, we have $M^{k+1}_\delta\le  \frac{1-d\lambda H^\ast_p}{2d\tau H^\ast_{pp}}-\frac{H^\ast_{xp}}{H^\ast_{pp}}$. By induction, we see that \eqref{semi3} holds for all $0\le k\le 2K$, and thus, \eqref{semi4} holds for all $0\le k<2K$. Now, it is clear that, if $M^0_\delta\le M^\ast_+$, we have $M^k_\delta\le M^\ast_+$ for all $0\le k\le 2K$. Note that these statements are true beyond $k=2K$ as long as $|(D_xv)^{k}_{m}|_\infty\le u^\ast$ holds.

We estimate the decay in the case (ii). Consider the initial value problem 
\begin{eqnarray*}
 w'(t)&=&G(w(\tau)),\quad w(0)=M^\ast_++\alpha, \\\alpha&:=&\frac{1-2\tau\{(1+d)(H_{xp}^\ast)^2+H_{pp}^\ast H_{xx}^\ast\}^{\frac{1}{2}}}{2H^\ast_{pp}\tau}=\frac{1+2H_{xp}^\ast\tau}{2H^\ast_{pp}\tau}-M_+^\ast. 
 \end{eqnarray*}
 The solution satisfies 
\begin{eqnarray*}
w(t)=M^\ast_+ +\frac{\eta^\ast e^{-\eta^\ast H^\ast_{pp}t} }{1-e^{-\eta^\ast H^\ast_{pp}t}+\frac{\eta^\ast}{\alpha}}=M^\ast_+ +\frac{\eta^\ast  }{e^{\eta^\ast H^\ast_{pp}t}-1+\frac{\eta^\ast}{\alpha}e^{\eta^\ast H^\ast_{pp}t}}\le M(t). 
\end{eqnarray*}
Since $\tau\le \log2/(H^\ast_{pp}\eta^\ast)$ ($\Rightarrow$ $e^{\eta^\ast H^\ast_{pp}\tau} -1\le 2\eta^\ast H^\ast_{pp}\tau$), $\tau\le [4\{(1+d)(H_{xp}^\ast)^2+H_{pp}^\ast H_{xx}^\ast\}^{\frac{1}{2}}]^{-1}$ ($\Rightarrow$ $\alpha^{-1}\le 4H_{pp}^\ast\tau$) and $\lambda\le 1/(10rH^\ast_{pp})$, we have 
\begin{eqnarray*}
&&w(\tau)\ge M^\ast_+ +\frac{\eta^\ast}{2\eta^\ast H^\ast_{pp}\tau+2\frac{\eta^\ast}{\alpha}}\ge M^\ast_++\frac{1}{10H^\ast_{pp}\tau},\\
&&M^1_\delta< M^0_\delta\le\frac{r}{h}=\frac{r\lambda}{\tau}\le \frac{1}{10H^\ast_{pp}\tau}\le M^\ast_++\frac{1}{10H^\ast_{pp}\tau}\le w(t_1)=w(\tau).
\end{eqnarray*}
Suppose that $M^k_\delta \le w(t_k)$ for some $k\ge1$. Note that $y+\tau G(y)$ is increasing for $y\le(1+2 H^\ast_{xp}\tau)/(2 H^\ast_{pp}\tau)$ and that 
\begin{eqnarray*}
w(t_k)\le w(0)= M_+^\ast+\alpha=\frac{1+2 H^\ast_{xp}\tau}{2H^\ast_{pp}\tau}.
\end{eqnarray*}  
Therefore, we see that 
\begin{eqnarray*}
M^{k+1}_\delta &\le& M^k_\delta+\tau G(M^k_\delta)\le w(t_k)+\tau G(w(t_k))=w(t_k)+\tau w'(t_k) \\
&=&w(t_{k+1})-\frac{\tau^2}{2}w''(t_k+\theta\tau)\qquad (\exists\,\theta\in(0,1))\\
&\le&w(t_{k+1}),
\end{eqnarray*}
where we note that $w''(t)>0$. By induction, we obtain our assertion. 
\end{proof}
Throughout this paper, we take $\tau,h>0$ small enough to satisfy the condition of Theorem \ref{semi}. 

As for \eqref{HJ-delta+}, we have the semiconvex estimate 
$$(D^2_jv)^k_{m+\1}=(v^k_{m+\1+2e_j}+v^k_{m+\1-2e_j}-2v^k_{m+\1})\frac{1}{4h^2}\ge -M(|t_k|),\quad k<0.$$

\subsection{Hyperbolic scaling limit of Lax-Oleinik type solution map}
We state convergence of the  solution map $\varphi^k_\delta$ of \eqref{HJ-delta} as $\delta\to0$.
We set  the Lax-Oleinik type operator $\varphi^t(\cdot;c):{\rm Lip}(\T^d;\R)\to {\rm Lip}(\T^d;\R)$, $t\ge0$  as  
$$\varphi^0(w;c)=w,\quad \varphi^t(w;c)(x)=\inf_{\gamma\in AC([0,t];\T^d),\,\,\gamma(t)=x}\left\{ \int^t_0 L^{(c)}(\gamma(s),s,\gamma'(s))ds+w(\gamma(0)) \right\},
$$
where we sometimes treat  $\gamma:[0,t]\to\T^d$ as $\gamma:[0,t]\to\R^d$. 
The viscosity solution $v$ of
\begin{eqnarray}\label{HJ}
\left\{
\begin{array}{lll}
&v_t(x,t)+H(x,t,c+v_x(x,t))=0\mbox{\quad in $\T^d\times(0,1]$},
\medskip\\
&v(x,0)=w(x)\mbox{\quad on $\T^d$}
\end{array}
\right.
\end{eqnarray}
is given as $v(\cdot,t)=\varphi^t(w;c)$ (see, e.g., \cite{Cannarsa}). By Tonelli's theory,  we have a minimizing curve $\gamma^\ast$ for each value $v(x,t)=\varphi^t(v^0;c)(x)$. 

Before discussing the hyperbolic scaling limit of $\varphi^k_\delta$, we  state Lipschitz interpolation of a function on $G_{\rm odd}$ or $G_{\rm even}$. 
\begin{Lemma}\label{interpolation}
For each function $u:G_{\rm odd}\ni x_{m+\1}\mapsto u_{m+\1}\in\R$ with $|D_xu|_\infty\le r$ on $G_{\rm odd}$, we have a Lipschitz continuous function $w:\R^d\to\R$ such that 
$$w|_{G_{\rm odd}}=u,\quad |w_x(x)|_\infty\le \beta r\mbox{\quad a.e. $x\in\R^d$\quad as $h\to0+$},$$
 where $\beta>0$ is a constant depending only on $d$. 
\end{Lemma}
\begin{proof}
Consider the $d$-dimensional cube 
$$C^d_{m+\1}:=\{x_{m+\1}+a_1e_1+\cdots+a_de_d\,|\,a_1,\ldots,a_d\in[0,1]\}.$$
We inductively construct $w$: For $d=1$, it is clear that 
$$f^1_0(x^1):=v_{m+\1}+\frac{v_{m+\1+2e_1}-v_{m+\1}}{2h}(x^1-x^1_{m+\1}),\quad 0\le x^1-x^1_{m+\1}\le 1$$   
is defined on $C^d_{m+\1}$ and can be connected with respect to $m$ to be a desired Lipschitz function. For $d=2$, in addition to the above $f^1_0(x^1)$, we 
set 
$$f^1_{0,2}(x^1):=v_{m+\1+2e_2}+\frac{v_{m+\1+2e_2+2e_1}-v_{m+\1+2e_2}}{2h}(x^1-x^1_{m+\1}),\quad 0\le x^1-x^1_{m+\1}\le 1$$   
and define 
$$f^{1,2}_0(x^1,x^2):=f^1_0(x^1)+\frac{f^1_{0,2}(x^1)-f^1_0(x^1)}{2h}(x^2-x^2_{m+\1}),\quad 0\le x^i-x^i_{m+\1}\le 1\,\,\,(i=1,2).$$
We see that $f^{1,2}_0(x^1,x^2)$ is defined on $C^d_{m+\1}$ and can be connected with respect to $m$ to be a desired Lipschitz function. For $d=3$, in addition to $f^{1,2}_0(x^1,x^2)$, we set $f^{1,2}_{3}(x^1,x^2)$ as 
\begin{eqnarray*}
&&f^1_{3}(x^1):=v_{m+\1+2e_3}+\frac{v_{m+\1+2e_3+2e_1}-v_{m+\1+2e_3}}{2h}(x^1-x^1_{m+\1}),\\
&&f^1_{3,2}(x^1):=v_{m+\1+2e_3+2e_2}+\frac{v_{m+\1+2e_3+2e_2+2e_1}-v_{m+\1+2e_3+2e_2}}{2h}(x^1-x^1_{m+\1}),\\
&&f^{1,2}_{3}(x^1,x^2):=f^1_{3}(x^1)+\frac{f^1_{3,2}(x^1)-f^1_{3}(x^1)}{2h}(x^2-x^2_{m+\1})
\end{eqnarray*}
and define 
\begin{eqnarray*}
&&f^{1,2,3}_0(x^1,x^2,x^3):=f^{1,2}_0(x^1,x^2)+\frac{f^{1,2}_{3}(x^1,x^2)-f^{1,2}_0(x^1,x^2)}{2h}(x^3-x^3_{m+\1}), \\
&&\mbox{ $0\le x^i-x^i_{m+\1}\le 1\,\,\,(i=1,2,3)$.}
\end{eqnarray*}
We see that $f^{1,2,3}_0(x^1,x^2,x^3)$ is defined on $C^d_{m+\1}$ and can be connected with respect to $m$ to be a desired Lipschitz function. 
For $d=4$, in addition to $f^{1,2,3}_0(x^1,x^2,x^3)$, we set $f^{1,2,3}_{4}(x^1,x^2,x^3)$ as 
\begin{eqnarray*}
&&f^1_{4}(x^1):=v_{m+\1+2e^4}+\frac{v_{m+\1+2e^4+2e_1}-v_{m+\1+2e^4}}{2h}(x^1-x^1_{m+\1}),\\
&&f^1_{4,2}(x^1):=v_{m+\1+2e_2+2e^4}+\frac{v_{m+\1+2e^4+2e_2+2e_1}-v_{m+\1+2e^4+2e_2}}{2h}(x^1-x^1_{m+\1}),\\
&&f^{1,2}_{4}(x^1,x^2):=f^1_{4}(x^1)+\frac{f^1_{4,2}(x^1)-f^1_{4}(x^1)}{2h}(x^2-x^2_{m+\1}),\\
&&f^1_{4,3}(x^1):=v_{m+\1+2e_3+2e^4}+\frac{v_{m+\1+2e_3+2e^4+2e_1}-v_{m+\1+2e_3+2e^4}}{2h}(x^1-x^1_{m+\1}),\\
&&f^1_{4,3,2}(x^1):=v_{m+\1+2e_2+2e_3+2e^4}+\frac{v_{m+\1+2e_2+2e_3+2e^4+2e_1}-v_{m+\1+2e_2+2e_3+2e^4}}{2h}(x^1-x^1_{m+\1}),\\
&&f^{1,2}_{4,3}(x^1,x^2):=f^1_{4,3}(x^1)+\frac{f^1_{4,3,2}(x^1)-f^1_{4,3}(x^1)}{2h}(x^2-x^2_{m+\1}),\\
&&f^{1,2,3}_{4}(x^1,x^2,x^3):=f^{1,2}_{4,3}(x^1,x^2)+\frac{f^{1,2}_{4,3}(x^1,x^2)-f^{1,2}_{4}(x^1,x^2)}{2h}(x^3-x^3_{m+\1})
\end{eqnarray*}
and define 
\begin{eqnarray*}
&&f^{1,2,3,4}_0(x^1,x^2,x^3,x^4):=f^{1,2,3}_0(x^1,x^2,x^3)\\
&& \qquad\qquad\qquad\qquad\qquad\qquad\qquad +\frac{f^{1,2,3}_{4}(x^1,x^2,x^3)-f^{1,2,3}_0(x^1,x^2,x^3)}{2h}(x^4-x^4_{m+\1}), \\
&&\mbox{ $0\le x^i-x^i_{m+\1}\le 1\,\,\,(i=1,2,3,4)$.}
\end{eqnarray*}
We see that $f^{1,2,3,4}_0(x^1,x^2,x^3,x^4)$ is defined on $C^d_{m+\1}$ and can be connected with respect to $m$ to be a desired Lipschitz function. We may repeat the same argument for $d=5,6,\ldots$.   
\end{proof}
\begin{Thm}\label{convergence-Lax-Oleinik}
Suppose that a sequence $\delta=(h,\tau)\to0$ is such that $0<\lambda_0\le \lambda=\tau/h< \lambda_1$ with any constant $\lambda_0\in(0,\lambda_1)$.  Let $\varphi_\delta^k(v^0_\delta;c)$ be  the solution of \eqref{HJ-delta} with  initial data $v^0=v^0_\delta$ for each element $\delta$ of the sequence, where $v^0_\delta$ may depend on $\delta$. 
\begin{enumerate}
\item If $|v^0_\delta|\le R$ for all $\delta$ with some constant $R\ge0$, there exist a subsequence of  $\{\varphi_\delta^k(v^0_\delta;c)\}_\delta$,  denoted by $\{\varphi_{\delta'}^k(v^0_{\delta'};c)\}_{\delta'}$, and a function $w\in {\rm Lip}(\T^d;\R)$ such that 
$$\sup_{0\le k\le 2K}\sup_{m}|\varphi^k_{\delta'}(v^0_{\delta'};c)(x_{m+\1})-\varphi^{t_k}(w;c)(x_{m+\1})|\to0\mbox{\quad as $\delta'\to0$},$$
where $\sup_m$ stands for the supremum with respect to $m$ such that $(x_{m+\1},t_k)\in \tG_\delta$.  
\item If $v^0_\delta$ is such that its Lipschitz interpolation  converges uniformly to a function $w\in {\rm Lip}(\T^d;\R)$ as $\delta\to0$, the whole sequence $\{\varphi_\delta^k(v^0_\delta;c)\}_\delta$ satisfies 
$$\sup_{0\le k\le 2K}\sup_{m}|\varphi^k_\delta(v^0_\delta;c)(x_{m+\1})-\varphi^{t_k}(w;c)(x_{m+\1})|\to0\mbox{\quad as $\delta\to0$}.$$
\item  If $v^0_\delta$ is such that $v^0_\delta(x_{m+\1})=w(x_{m+\1})$ with a fixed $w\in {\rm Lip}_r(\T^d)$, it holds that 
$$\sup_{0\le k\le 2K}\sup_{m}|\varphi^k_\delta(v^0_\delta;c)(x_{m+\1})-\varphi^{t_k}(w;c)(x_{m+\1})|\le\beta_0 \sqrt{h},$$
where $\beta_0>0$ is a constant independent of  $\delta$, $w$ and $c$. 
\item Let $w_\delta$ be the Lipschitz interpolation of $v^0_\delta$. Then, it holds that 
$$\sup_{0\le k\le 2K}\sup_{m}|\varphi^k_\delta(v^0_\delta;c)(x_{m+\1})-\varphi^{t_k}(w_\delta;c)(x_{m+\1})|\le\tilde{\beta}_0 \sqrt{h},$$
where $\tilde{\beta}_0>0$ is a constant independent of $\delta$, $v^0_\delta$ and $c$. 
\end{enumerate}
\end{Thm}
\begin{proof} By Lemma \ref{interpolation}, the sequence of the Lipschitz interpolation of $v^0_\delta$ is uniformly bounded and equi-Lipschitz ($r$ is fixed in \eqref{HJ-delta} for all $\delta$). Hence, we find a convergent subsequence with the limit $w\in {\rm Lip}(\T^d;\R)$. Then, the claim 1 is reduced to the claim 2.      
We can prove the claims 2, 3 and 4 in the same way as the proof of Theorem 2.2 in \cite{Soga5}, by means of Lemma \ref{limit-theorem}.   \end{proof}
As for \eqref{HJ-delta+}, $\tilde{\varphi}^k_\delta(\cdot;c)$ tends to $\tilde{\varphi}^t(\cdot;c)$, where   
$$ \tilde{\varphi}^t(w;c)(x)=\sup_{\gamma\in AC([t,0];\T^d),\,\,\gamma(t)=x}\left\{ -\int_t^0 L^{(c)}(\gamma(s),s,\gamma'(s))ds+w(\gamma(0)) \right\},\quad t\in[-1,0].
$$
Note that $\tilde{v}(\cdot,t)=\tilde{\varphi}^t(w;c)$ is the semiconvex a.e. solution of
\begin{eqnarray}\label{HJ+}
\left\{
\begin{array}{lll}
&v_t(x,t)+H(x,t,c+v_x(x,t))=0\mbox{\quad in $\T^d\times[-1,0)$},
\medskip\\
&v(x,0)=w(x)\mbox{\quad on $\T^d$}.
\end{array}
\right.
\end{eqnarray}
\subsection{Hyperbolic scaling limit of derivative of solution} 
We prove that, when $v^k=\varphi^k_\delta(v^0;c)$ converges to a viscosity solution $u(\cdot,t)=\varphi^t(w;c)$ of \eqref{HJ} (we use the symbol $u$ instead of $v$) for some sequence $\delta\to0$ in the sense of the claim 2 of Theorem \ref{convergence-Lax-Oleinik}, each partial derivative $(D_{x^i}v)^k$ also converges to $u_{x^i}$ pointwise a.e. A similar statement is proved in \cite{Soga5} under the assumption that $v^0_{m+\1}=w(x_{m+\1})$ with a fixed semiconcave function $w\in {\rm Lip}_r(\T^d;\R)$. Here, we remove this assumption by means of a different approach based on Theorem \ref{semi}.  
\begin{Thm}\label{convergence-derivative}
Consider the situation of the claim 2 of Theorem \ref{convergence-Lax-Oleinik}. 
Let $(x,t)\in \T^d\times(0,1]$ be such that $u_{x^i}(x,t)$ exists, where $u(\cdot,t):=\varphi^t(w;c)$. Note that a.e. points of $\T^d\times[0,1]$ have such a property. Let $(x_n,t_{l})\in{\mathcal{G}}_\delta$ be such that $(x_n,t_{l})\to(x,t)$ as $\delta\to0$.   Then, $v^k_{m+\1}=\varphi^k_\delta(v^0;c)(x_{m+\1})$ satisfies     
$$|(D_{x^i}v)^l_n-u_{x^i}(x,t)|\to0\mbox{\quad as $\delta\to0$}.$$
\end{Thm}
\begin{proof}
The strategy is the following: if the convergence fails, $(D_{x^i}v)^l_m$  must keep away from $u_{x^i}(\cdot,t)$ within a certain interval in the $e_i$-direction or $(-e_i)$-direction because of semiconcavity, which violates Theorem \ref{convergence-Lax-Oleinik}. 

We proceed by contradiction. Suppose that we have an $\ep_0>0$ and a subsequence of $\delta\to0$, still denoted by the same symbol, such that $|(D_{x^i}v)^l_n-u_{x^i}(x,t)|\ge \ep_0$ for $\delta\to0$. Then, there are two cases
$$\mbox {(i) $(D_{x^i}v)^l_n\ge u_{x^i}(x,t)+ \ep_0$,\qquad (ii) $(D_{x^i}v)^l_n\le u_{x^i}(x,t)- \ep_0$. }$$
Since $s(y):=v(x^1,\ldots,x^{i-1},y,x^{i+1},\ldots,x^d,t)$ is a semiconcave function of $y$ due to the semiconcave feature of $\varphi^t(w;c)$, we have 
$${\rm ess}\!\!\!\!\!\!\!\sup_{\{y\,|\,|y-x^i|\le \nu\}} |s'(y)-u_{x^i}(x,t)|\to 0 \mbox{\quad as $\nu\to0$}.$$
 Hence, there exists $\nu_0>0$ such that $|s'(y)-u_{x^i}(x,t)|\le \ep_0/3$ for a.e. $y$ with $|y-x^i|\le \nu_0$ ($y$ must be a point of differentiability of $s(\cdot)$).  Let $r_0\in\N$ be a natural number satisfying $\frac{1}{2}\min\{\nu_0, \frac{\ep_0}{3M(t_l)}\}\le2hr_0+2h\le \min\{\nu_0, \frac{\ep_0}{3M(t_l)}\} $, where $M(t)$ is mentioned in Theorem \ref{semi}. 

Case (i): Due to Theorem \ref{semi}, we have for any $0<r\le r_0$,
\begin{eqnarray*}
(D_{x^i}v)^l_n-(D_{x^i}v)^l_{n-2re_i}
&=&\sum_{r'=0}^{r-1}\Big(  (D_{x^i}v)^l_{n-2re_i+2(r'+1)e_i} -   (D_{x^i}v)^l_{n-2re_i+2r'e_i} \Big)\\ 
&\le& 2hrM(t_l)\le 2hr_0M(t_l)\le\frac{\ep_0}{3},\\
(D_{x^i}v)^l_{n-2re_i}-s'(y)&\ge&\Big( (D_{x^i}v)^l_n- \frac{\ep_0}{3}\Big)-\Big(u_{x^i}(x,t)+\frac{\ep_0}{3}\Big)\\
&\ge&\frac{\ep_0}{3}\mbox{ \quad for a.e. $y\in[x^i-\nu_0, x^i+\nu_0]$}.
\end{eqnarray*}
Therefore, we see that for $\delta\to0$,
\begin{eqnarray*}
&&(v^l_{n+e^i}- u(x+he_i,t))-(v^l_{n-2r_0e_i-e_i}-u(x-2hr_0e_i-he_i,t)) \\
&&\qquad = \sum_{r=0}^{r_0} \int_{x^i-2hr-h}^{x^i-2hr+h}\Big\{(D_{x^i}v)^l_{n-2re_i}-s'(y)\Big\}dy\\
&&\qquad \ge \frac{\ep_0}{3}(2hr_0+2h)\ge\frac{\ep_0}{3}\cdot \frac{1}{2}\min\{\nu_0, \frac{\ep_0}{3M(t_l)}\}. 
\end{eqnarray*}
Since Theorem \ref{convergence-Lax-Oleinik} implies that the first line tends to $0$ as $\delta\to0$, we reach a contradiction.

 Case (ii): Similarly,   we have for any $0<r\le r_0$,
\begin{eqnarray*}
(D_{x^i}v)^l_{n+2re_i}-(D_{x^i}v)^l_n&\le& 2hrM(t_l)\le 2hr_0M(t_l)\le \frac{\ep_0}{3},\\
s'(y)-(D_{x^i}v)^l_{n+2re_i}&\ge&\Big(u_{x^i}(x,t)-\frac{\ep_0}{3}\Big)-\Big( (D_{x^i}v)^l_n+ \frac{\ep_0}{3}\Big)\\
&\ge&\frac{\ep_0}{3}\mbox{ \quad for a.e. $y\in[x^i-\nu_0, x^i+\nu_0]$}.
\end{eqnarray*}
Therefore, we see that  for $\delta\to0$,
\begin{eqnarray*}
&&(u(x+2hr_0e_i+he_i,t)-v^l_{n+2r_0e_i+e_i})-( u(x-he_i,t)-v^l_{n-e^i}) \\
&&\qquad = \sum_{r=0}^{r_0} \int_{x^i+2hr-h}^{x^i+2hr+h}\Big\{s'(y)-(D_{x^i}v)^l_{n+2re_i}\Big\}dy\\
&&\qquad \ge \frac{\ep_0}{3}(2hr_0+2h)\ge \frac{\ep_0}{3}\cdot \frac{1}{2}\min\{\nu_0, \frac{\ep_0}{3M(t_l)}\},  
\end{eqnarray*}
which is a contradiction.  
\end{proof}
As for \eqref{HJ-delta+}, we obtain a similar result for $\tilde{\psi}^k_\delta(\cdot;c)$.
\subsection{Hyperbolic scaling limit of minimizing random walk} 
 
It is proved in \cite{Soga5} that a minimizing random walk of $\varphi^k_\delta$ converges to a minimizing curve with an end point $(x,t)$ of $\varphi^t$  as $\delta\to0$ under the hyperbolic scaling, provided a minimizing curve with the end point $(x,t)$ is unique (a.e. points $(x,t)$ have such a property). In this subsection, we extend the result  without any assumption.  Note that there is a case where we have  uncountably many minimizing curves with a common end point. 

Let  $\gamma_\delta(\cdot):[0,t]\to\R^d$, $\eta_\delta(\gamma)(\cdot):[0,t]\to\R^d$, $t\le 1$ be the linear interpolations of each sample path $\gamma\in \Omega_n^{l+1,0}$, $\eta(\gamma)$, respectively, where $t \in[t_{l+1},t_{l+2})$:
\begin{eqnarray*}
\gamma_\delta(s)&:=&
\left\{
\begin{array}{lll}
&x_n \mbox{\quad for  $s\in[t_{l+1},t]$},
\medskip\\
&\dis \gamma^k+\frac{\gamma^{k+1}-\gamma^k}{\dt}(s-t_k) \mbox{\quad for  $s\in[t_{k},t_{k+1}]$}.
\end{array}
\right.\\
\eta_\delta(\gamma)(s)&:=&
\left\{
\begin{array}{lll}
&x_n \mbox{\quad for  $s\in[t_{l+1},t]$},
\medskip\\
&\dis \eta^k(\gamma)+\frac{\eta^{k+1}(\gamma)-\eta^k(\gamma)}{\dt}(s-t_k) \mbox{\quad for  $s\in[t_{k},t_{k+1}]$}.
\end{array}
\right.
\end{eqnarray*} 
\begin{Lemma}\label{convert}
Let $f:[0,t]\to\R$ be a Lipschitz function with a Lipschitz constant $\theta$ satisfying $f(t)=0$. Then, it holds that $\norm f\norm_{C^0([0,t])}\le\theta\norm f \norm_{L^2([0,t])}+\sqrt{\norm f\norm_{L^2([0,t])}}$.
\end{Lemma}
\noindent See the proof of Lemma 3.5 in \cite{Soga2}. 
\begin{Thm}\label{in-probability1}
Consider the situation of the claim 2 of Theorem \ref{convergence-Lax-Oleinik}. 
Let $(x,t)\in \R^d\times(0,1]$ be an arbitrary point. Let $\Gamma^\ast(x,t)$ be the set of all minimizing curves $\gamma^\ast:[0,t]\to\R^d$ for $u(x,t):=\varphi^{t}(w;c)(x)$. Let $(x_n,t_{l+1})\in\tilde{\mathcal{G}}_\delta$ be such that $(x_n,t_{l+1})\to(x,t)$ as $\delta\to0$. For each $\delta$, let $\gamma\in \Omega^{l+1,0}_n$ be the random walk generated by the minimizing control $\xi^\ast$ for $v^{l+1}_n$.  
\begin{enumerate}
\item Fix  an arbitrary $\ep_1>0$ and define the set
\begin{eqnarray*}
\check{\Omega}^{\ep_1}_\delta&:=&\{ \gamma\in  \Omega^{l+1,0}_n\,|\,\mbox{there exists $\gamma^\ast=\gamma^\ast(\gamma)\in\Gamma^\ast(x,t)$} \\
&&\qquad\qquad\qquad\qquad\qquad\qquad\qquad\mbox{ such that $\norm \eta_\delta(\gamma)'-\gamma^\ast{}'\norm_{L^2([0,t])}\le \ep_1$}  \}.
\end{eqnarray*}
Then, we have {\rm Prob}$(\check{\Omega}_\delta^{\ep_1})\to1$ as $\delta\to0$.
\item Fix an arbitrary $\ep_2>0$ and define the set
$$
\tilde{\Omega}^{\ep_2}_\delta:=\{ \gamma\in  \Omega^{l+1,0}_n\,|\,\mbox{there exists $\gamma^\ast=\gamma^\ast(\gamma)\in\Gamma^\ast$}
\mbox{ such that $\norm\eta_\delta(\gamma)-\gamma^\ast\norm_{C^0([0,t])}\le \ep_2$}\}.$$
Then, we have $\mbox{\rm Prob}(\tilde{\Omega}_\delta^{\ep_2})\to1$ as $\delta\to0$. 
\item Fix  an arbitrary $\ep_3>0$ and  define the set
\begin{eqnarray*}
\Omega^{\ep_3}_\delta&:=&\{ \gamma\in  \Omega^{l+1,0}_n\,|\,\mbox{there exists $\gamma^\ast=\gamma^\ast(\gamma)\in\Gamma^\ast(x,t)$} \\
&&\qquad\qquad\qquad\qquad\qquad\qquad\qquad\mbox{ such that $\norm\gamma_\delta-\gamma^\ast\norm_{C^0([0,t])}\le \ep_3$}  \}. 
\end{eqnarray*}
Then, we have {\rm Prob}$(\Omega_\delta^{\ep_3})\to1$ as $\delta\to0$.
\end{enumerate}
\end{Thm}
\begin{proof}
For each $\gamma\in\check{\Omega}^{\ep_1}_\delta$, we have $\gamma^\ast\in\Gamma^\ast(x,t)$ such that 
\begin{eqnarray*}
\norm \eta_\delta(\gamma)'-\gamma^\ast{}'\norm_{L^1([0,t])}\le \sqrt{t}\norm\eta_\delta(\gamma)'-\gamma^\ast{}'\norm_{L^2([0,t])}\le \ep_1 \sqrt{t}\le\ep_1, 
\end{eqnarray*}
which implies 
\begin{eqnarray*}
|\eta_\delta(\gamma)(s)+x-x_n-\gamma^\ast(s)|=\Big| \int_s^t  \Big(\eta_\delta(\gamma)'(\tilde{s})-\gamma^\ast{}' (\tilde{s})\Big) d\tilde{s} \Big|\le \ep_1\mbox{\quad for all $s\in[0,t]$},
\end{eqnarray*}
where $\eta_\delta(\gamma)(s)=x_n$ for $s\in[t_{l+1},t]$. 
Hence, for sufficiently small $\delta$ such that $|x-x_n|\le \ep_1/2$, we have $\norm \eta_\delta(\gamma)-\gamma^\ast\norm_{C^0([0,t])}\le 3\ep_1/2$. Therefore,   
$\check{\Omega}^{\ep_1}_\delta\subset \tilde{\Omega}^{\frac{3\ep_1}{2}}_\delta$ holds and the claim 2 follows from the claim 1. 

For any $\ep>0$, define the set 
\begin{eqnarray*}
&&\Gamma^\ep:=\{r:[0,t]\to\R^d\,|\, r\in {\rm Lip},\,\, |r'|_\infty\le (d\lambda_1)^{-1},\,\,  \\
&&\qquad\qquad\qquad\qquad\qquad   r(t)=x,\,\,\norm r'-\gamma^\ast{}'\norm_{L^2([0,t])}\ge\ep \mbox{\,\,\,for all $\gamma^\ast\in\Gamma^\ast(x,t)$}\}.
\end{eqnarray*}
We prove that there exists  $\nu(\ep)>0$ such that 
\begin{eqnarray}\label{kiks}
 \inf_{r\in \Gamma^{\ep}} 
 \Big(\int_0^tL^{(c)}(r(s),s,r'(s))ds +w(r(0))- u(x,t)\Big)  \ge\nu(\ep).
 \end{eqnarray}
Suppose that there is no such $\nu(\ep)>0$. Then, we have  a sequence $\{r_i\}_{i\in\N}\subset \Gamma^\ep$ such that  
\begin{eqnarray*}
 &&\int_0^tL^{(c)}(r_i(s),s,r_i'(s))ds +w(r_i(0))\\
 &&\qquad\qquad \to u(x,t)=\inf_{\gamma\in AC,\,\,\gamma(t)=x}\left\{ \int^t_0 L^{(c)}(\gamma(s),s,\gamma'(s))ds+w(\gamma(0))\right\}.
 \end{eqnarray*}
 By Tonelli's theory (see \cite{Fathi-book}), we find a subsequence of $\{r_i\}$, still denoted by the same symbol, which converges uniformly  to a curve $\gamma^\ast\in\Gamma^\ast(x,t)$. Observe that 
 \begin{eqnarray*}
&&\kappa_i:= \int_0^tL^{(c)}(r_i(s),s,r_i'(s))ds +w(r_i(0))- \Big\{
  \int^t_0 L^{(c)}(\gamma^\ast(s),s,\gamma^\ast{}'(s))ds+w(\gamma^\ast(0))\Big\}\\
&&\,\,\quad =\int_0^t\Big\{L^{(c)}(r_i(s),s,r_i'(s))-L^{(c)}(\gamma^\ast(s),s,r_i'(s))\Big\}ds\\
&&\qquad+ \int_0^t \Big\{L^{(c)}(\gamma^\ast(s),s,r_i'(s)) -L^{(c)}(\gamma^\ast(s),s,\gamma^\ast{}'(s)) \Big\}ds+(w(r_i(0))-w(\gamma^\ast(0)))\\
&&\,\,\quad=\int_0^tL^{(c)}_x(\gamma^\ast(s)+\theta_i(s),s,r_i'(s))\cdot (r_i(s)-\gamma^\ast(s))ds \\
&&\qquad+ \int_0^t L^{(c)}_\zeta(\gamma^\ast(s),s,\gamma^\ast{}'(s))\cdot  (r_i'(s)-\gamma^\ast{}'(s))ds\\
&&\qquad+\int_0^t\frac{1}{2}L^{(c)}_{\zeta\zeta}(\gamma^\ast(s),s,\gamma^\ast{}'(s)+\tilde{\theta}_i(s)) (r_i'(s)-\gamma^\ast{}'(s))\cdot (r_i'(s)-\gamma^\ast{}'(s))ds\\
&&\qquad+ (w(r_i(0))-w(\gamma^\ast(0)))\to0\mbox{\quad as $i\to\infty$},
 \end{eqnarray*}
 where $\theta_i, \tilde{\theta}_i$ come from the Taylor's formula. In the right hand side, the first term and the last term are bounded by $O(\norm r_i-\gamma^\ast\norm_{C^0([0,t])})$; since every element of $\Gamma^\ast(x,t)$ is $C^2$-solution of the Euler-Lagrange equation, the second term is treated as 
 \begin{eqnarray*}
&& \int_0^t L^{(c)}_\zeta(\gamma^\ast(s),s,\gamma^\ast{}'(s))\cdot  (r_i'(s)-\gamma^\ast{}'(s))ds\\
 &&\quad =L^{(c)}_\zeta(\gamma^\ast(s),s,\gamma^\ast{}'(s))\cdot (r_i(s)-\gamma^\ast{}(s))|_{s=0}^{s=t} \\
&&\qquad  - \int_0^t \frac{d}{ds}L^{(c)}_{\zeta}(\gamma^\ast(s),s,\gamma^\ast{}'(s))\cdot  (r_i(s)-\gamma^\ast{}(s))ds\\
&&\quad =L^{(c)}_\zeta(\gamma^\ast(s),s,\gamma^\ast{}'(s))\cdot (r_i(s)-\gamma^\ast{}(s))|_{s=0}^{s=t} \\
&&\qquad  - \int_0^t L^{(c)}_x(\gamma^\ast(s),s,\gamma^\ast{}'(s))\cdot  (r_i(s)-\gamma^\ast{}(s))ds\\
&&\quad\le O(\norm r_i-\gamma^\ast\norm_{C^0([0,t])});
 \end{eqnarray*}
 by convexity of $L^{(c)}$, the third term is bounded from the below by $\alpha' \norm r_i'-\gamma^\ast{}'\norm_{L^2([0,t])}^2$ with a positive constant $\alpha'$ independent from $i$. Hence, we obtain 
 $$0<\alpha'\ep^2\le\alpha' \norm r_i'-\gamma^\ast{}'\norm_{L^2([0,t])}^2\le \kappa_i+O(\norm r_i-\gamma^\ast\norm_{C^0([0,t])})\to0\quad {\rm as}\,\,\,i\to\infty,$$
which is a contradiction. 
 
If the claim 1 does not hold, we have $b>0$ and a sequence $\delta_i=(h_i,\tau_i)\to0$ as $i\to\infty$ for which Prob$(\Omega^{l+1,0}_n\setminus\check{\Omega}_{\delta_i}^{\ep_1})=1-{\rm Prob}(\check{\Omega}_{\delta_i}^{\ep_1})\ge b$ for all $i$. 
For any $\gamma\in \Omega^{l+1,0}_n\setminus\check{\Omega}_{\delta_i}^{\ep_1}$, we have $\eta_{\delta_i}(\gamma)(\cdot)+x-x_n\in\Gamma^{\ep_1}$. Hence, by Lemma \ref{limit-theorem} and \eqref{kiks}, we have  
\begin{eqnarray*}
v^{l+1}_n&=&E_{\mu^{l+1,0}_n(\cdot;\xi^\ast)}\Big[ \int_0^tL^{(c)}(\eta_{\delta_i}(\gamma)(s),s,\eta_{\delta_i}(\gamma)'(s))ds +w(\eta_{\delta_i}(\gamma)(0))\Big]\\
&&+O(\sqrt{h_i})+O(|t-t_{l+1}|)+O(\max|v^0_{\delta_i}-w|)\\
&=&\sum_{\gamma\in\Omega^{l+1,0}_n\setminus\check{\Omega}_{\delta_i}^{\ep_1}} \mu_n^{l+1,0}(\gamma;\xi^\ast)\Big\{       \int_0^tL^{(c)}(\eta_{\delta_i}(\gamma)(s)+x-x_n,s,\eta_{\delta_i}(\gamma)'(s))ds \\
&&+w(\eta_{\delta_i}(\gamma)(0)+x-x_n) \Big\} \\
&&+\sum_{\gamma\in\check{\Omega}_{\delta_i}^{\ep_1}} \mu_n^{l+1,0}(\gamma;\xi^\ast)\Big\{       \int_0^tL^{(c)}(\eta_{\delta_i}(\gamma)(s)+x-x_n,s,\eta_{\delta_i}(\gamma)'(s))ds\\
&& +w(\eta_{\delta_i}(\gamma)(0)+x-x_n) \Big\}\\
&&+O(|x-x_n|_\infty)+O(\sqrt{h_i})+O(|t-t_{l+1}|)+O(\max|v^0_{\delta_i}-w|)\\
&\ge&\sum_{\gamma\in\Omega^{l+1,0}_n\setminus\check{\Omega}_{\delta_i}^{\ep_1}} \mu_n^{l+1,0}(\gamma) \Big(\nu(\ep)+u(x,t) \Big)
+\sum_{\gamma\in\check{\Omega}_{\delta_i}^{\ep_1}} \mu_n^{l+1,0}(\gamma) u(x,t)\\
&&+O(|x-x_n|_\infty)+O(\sqrt{h_i})+O(|t-t_{l+1}|)+O(\max|v^0_{\delta_i}-w|)\\
&\ge&b\nu(\ep)+ u(x,t) \\
&&+O(|x-x_n|_\infty)+O(\sqrt{h_i})+O(|t-t_{l+1}|)+O(\max|v^0_{\delta_i}-w|)\\
&>&u(x,t)+\frac{1}{2}b\nu(\ep) \mbox{\quad as $i\to\infty$}.
\end{eqnarray*}
Since we have  the convergence $v^{l+1}_n\to u(x,t)$ as $\delta\to0$, we reach a contradiction.

We prove the claim 3.  For any $\ep>0$, define the set   
\begin{eqnarray*}
B^{\ep}_\delta:=\{ \gamma\in  \Omega^{l+1,0}_n\,|\,\norm\eta_\delta(\gamma)-\gamma_\delta\norm_{L^2([0,t])}\le \ep  \}.
\end{eqnarray*}
Then, by Lemma \ref{limit-theorem}, we have 
\begin{eqnarray*}
\sum_{\gamma\in\Omega^{l+1,0}_n\setminus B^{\ep}_\delta}\mu^{l+1,0}_n(\gamma;\xi^\ast)\ep^2&\le& \sum_{\gamma\in\Omega^{l+1,0}_n}\mu^{l+1,0}_n(\gamma;\xi^\ast)\int_0^{t}|\eta_\delta(\gamma)(s)-\gamma_\delta(s)|^2ds\\
&&\!\!\!\!\!\!\!\!\!\!\!\!\!\!\!\!\!\!\!\!\!\!\!\!\!\!\!\!\!\!
=\,\,\sum_{\gamma\in\Omega^{l+1,0}_n}\mu^{l+1,0}_n(\gamma)\Big\{\sum_{0\le k\le l} |\eta^k(\gamma)-\gamma^k|^2\tau \Big\}+O(h)
=O(h).
\end{eqnarray*}
Hence, for any fixed $\ep_3\gg\ep>0$, we obtain 
\begin{eqnarray}\label{hhhii}
\mbox{Prob}(\Omega^{l+1,0}_n\setminus B^{\ep}_\delta)=1-\mbox{Prob}( B^{\ep}_\delta)\le\frac{O(h)}{\ep{}^2} \to0\mbox{\quad as $\delta\to0$}.
\end{eqnarray}
Below, $\beta_1,\beta_2,\ldots$ are some positive constants independent of $\delta$. By Lemma \ref{convert}, we have a constant $\beta_1>0$ such that for any $\gamma\in B^{\ep}_\delta$,
\begin{eqnarray*}
\norm\eta_\delta(\gamma)-\gamma_\delta\norm_{C^0([0,t])}\le \beta_1 \sqrt{\ep}.
\end{eqnarray*}
With $\sum_{\gamma\in\Omega^{l+1,0}_n}=\sum_{\gamma\in B^{\ep}_\delta}+\sum_{\gamma\in\Omega^{l+1,0}_n\setminus B^{\ep}_\delta}$, we have 
\begin{eqnarray}\label{isa}
\sum_{\gamma\in\Omega^{l+1,0}_n}\mu^{l+1,0}_n(\gamma)\norm \eta_\delta(\gamma)-\gamma_\delta\norm_{C^0([0,t])}\le \beta_1\sqrt{\ep}+\beta_2(1-\mbox{Prob}(B^{\ep}_\delta)).
\end{eqnarray}
For each $\gamma\in B^{\ep}_\delta\setminus \Omega_\delta^{\ep_3}$, it holds that 
$$\norm \eta_\delta(\gamma)-\gamma_\delta\norm_{C^0([0,t])}\le\beta_1\sqrt{\ep},\quad \norm \gamma_\delta-\gamma^\ast\norm_{C^0([0,t])}>\ep_3\mbox{\quad for all $\gamma^\ast\in\Gamma^\ast$},$$
which implies 
\begin{eqnarray*}
\norm \eta_\delta(\gamma) -\gamma^\ast\norm_{C^0([0,t])}&\ge &\norm \gamma_\delta -\gamma^\ast\norm_{C^0([0,t])}-\norm \eta_\delta(\gamma)-\gamma_\delta\norm_{C^0([0,t])}\\
&>&\ep_3- \beta_1\sqrt{\ep}\mbox{\quad for all $\gamma^\ast\in\Gamma^\ast$}.
\end{eqnarray*}
Therefore, we see that for $0<\ep\ll\ep_3$,
\begin{eqnarray}\label{SOSOS}
\mbox{$\gamma\in B^{\ep}_\delta\setminus \Omega_\delta^{\ep_3}$ \,\,$\Rightarrow$\,\, $\gamma\in \Omega^{l+1,0}_n\setminus\tilde{\Omega}_\delta^{\ep_3-\beta_1\sqrt{\ep}}$}.
\end{eqnarray}
For each $\gamma\in\Omega^{l+1,0}_n$, let $\gamma^\ast=\gamma^\ast(\gamma)$ denote the element of  $\Gamma^\ast$ nearest to $\gamma_\delta$ in $\norm \cdot\norm_{C^0}$ (if there are several, it is one of them). Then, by \eqref{isa} and \eqref{SOSOS}, we have 
\begin{eqnarray*}
\beta_1\sqrt{\ep}&\!\!\!\!+\!\!\!\!&\beta_2(1-\mbox{Prob}(B^{\ep}_\delta))\ge\sum_{\gamma\in\Omega^{l+1,0}_n}\mu^{l+1,0}_n(\gamma)\norm \gamma_\delta-\eta_\delta(\gamma)\norm_{C^0([0,t])}\\
&\ge& \sum_{\gamma\in B^{\ep}_\delta\setminus \Omega_\delta^{\ep_3}}\mu^{l+1,0}_n(\gamma)\norm \gamma_\delta-\eta_\delta(\gamma)\norm_{C^0([0,t])}\\
&\ge& \sum_{\gamma\in B^{\ep}_\delta\setminus \Omega_\delta^{\ep_3}}\mu^{l+1,0}_n(\gamma)\norm \gamma_\delta-\gamma^\ast(\gamma)\norm_{C^0([0,t])} \\
&&- \sum_{\gamma\in B^{\ep}_\delta\setminus \Omega_\delta^{\ep_3}}\mu^{l+1,0}_n(\gamma)\norm \gamma^\ast(\gamma)-\eta_\delta(\gamma)\norm_{C^0([0,t])}\\
&\ge&\mbox{Prob}(B^{\ep}_\delta\setminus \Omega_\delta^{\ep_3})\ep_3 -  \beta_3\mbox{Prob}(\Omega^{l+1,0}_n\setminus\tilde{\Omega}_\delta^{\ep-\beta_1\sqrt{\ep}}),
\end{eqnarray*}
which implies with $\mbox{Prob}(\Omega^{l+1,0}_n\setminus\tilde{\Omega}_\delta^{\ep_3-\beta_1\sqrt{\ep}})=1-\mbox{Prob}(\tilde{\Omega}_\delta^{\ep_3-\beta_1\sqrt{\ep}})$,
\begin{eqnarray*}
\mbox{Prob}(B^{\ep}_\delta)-\mbox{Prob}( \Omega_\delta^{\ep_3})&\le&\mbox{Prob}(B^{\ep}_\delta\setminus \Omega_\delta^{\ep_3})\\
&\le&\frac{\beta_1\sqrt{\ep}+\beta_2(1-\mbox{Prob}(B^{\ep}_\delta))  + \beta_3(1-\mbox{Prob}(\tilde{\Omega}_\delta^{\ep_3-\beta_1\sqrt{\ep}}))}{\ep_3}.
\end{eqnarray*}
Sending $\delta\to0$, we obtain with the claim 2 and \eqref{hhhii}, 
\begin{eqnarray*}
1-\frac{\beta_1\sqrt{\ep}}{\ep_3}\le\liminf_{\delta\to0}\mbox{Prob}( \Omega_\delta^{\ep_3})\le1.
\end{eqnarray*}
Since we may take $\ep>0$ arbitrarily small so that $\sqrt{\ep}/\ep_3>0$ can be arbitrarily small, we conclude 
$$\lim_{\delta\to0}\mbox{Prob}( \Omega_\delta^{\ep_3})=1.$$
\end{proof}
As for \eqref{HJ-delta+}, we obtain a similar result for maximizing random walks of $\tilde{\varphi}^k_\delta(\cdot;c)$.
\subsection{Time-global extension of Lax-Oleinik type solution map} 

We extend the solution map $\varphi^k_{\delta}(\cdot;c)$ of \eqref{HJ-delta} to $k\to\infty$ with a fixed $\delta=(h,\tau)$. This is a non-trivial issue, because the constant $\lambda_1$ found in Theorem \ref{main1} that guarantees the solvability of (\ref{HJ-delta}) within $[0,1]$ depends on the terminal time, i.e, an instant observation only shows that $v^0$ with  $|(D_xv^0)_m|_\infty\le r$ yields $v^{2K}$ such that $|(D_xv)^{2K}_m|_\infty\le u^\ast$ with $u^\ast>r$ in general.  Hence, we need an additional argument to prove $|(D_xv)^{2K}_m|_\infty\le r$ for some $r>0$ so that we can solve (\ref{HJ-delta}) within $[0,1]$ with initial data $v^{2K}$ with the same $\delta$. 

Our strategy is the following: we first observe a priori boundedness of minimizing curves of $u(\cdot,t)=\varphi^t(w:c)$ independently from $w\in{\rm Lip}(\T^d;\R)$, which implies a priori boundedness of $u_x(\cdot,1)$ independently from $w\in{\rm Lip}(\T^d;\R)$; Theorem \ref{semi} and \ref{convergence-Lax-Oleinik} implies a similar boundedness for  $v^{2K}=\varphi^{2K}_\delta(v^0;c)$.
\begin{Lemma}\label{bdd-of-minimizer}  
For each $t>0$, there exists a constant $\beta(t)>0$ independent of $c\in P$ and $w\in {\rm Lip}(\T^d;\R)$ such that every minimizing curve $\gamma^\ast$ for $u(\cdot,t)=\varphi^t(w;c)$ satisfies  $|\gamma^\ast{}'(s)|_\infty\le \beta(t)$ for all $s\in[0,t]$.
\end{Lemma}
\begin{proof}
This is proved in \cite{Fathi-book} for time-independent Lagrangians. For readers' convenience, we give a proof. 
Let $\gamma^\ast$ be a minimizing curve for $u(x,t)=\varphi^t(w;c)(x)$ regarded as $\gamma^\ast:[0,t]\to\R^d$ with $x\in[0,1)^d$. Set $y:=\gamma^\ast(0)+z$ with $z\in\Z^d$ such that $\gamma^\ast(0)+z\in[0,1)^d$.  We see that 
\begin{eqnarray*}
\int_0^t L^{(c)}(\gamma^\ast(s),s,\gamma^\ast{}'(s))ds\le\inf_{\gamma\in AC, \gamma(t)=x,\gamma(0)=y} \int_0^t L^{(c)}(\gamma(s),s,\gamma'(s))ds.
\end{eqnarray*} 
In fact, if not, there exists $\gamma\in AC([0,t];\R^d), \gamma(t)=x,\gamma(0)=y$ such that 
\begin{eqnarray*}
&&\int_0^t L^{(c)}(\gamma^\ast(s),s,\gamma^\ast{}'(s))ds>\int_0^t L^{(c)}(\gamma(s),s,\gamma'(s))ds,\\
&&\quad \Rightarrow 
\int_0^t L^{(c)}(\gamma^\ast(s),s,\gamma^\ast{}'(s))ds+w(\gamma^\ast(0))>\int_0^t L^{(c)}(\gamma(s),s,\gamma'(s))ds+w(\gamma(0)),
\end{eqnarray*} 
where we note that $w(\gamma^\ast(0))=w(y)=w(\gamma(0))$ due to the periodicity of $w$; this violates the assumption that $\gamma^\ast$ is a minimizing curve for $u(x,t)=\varphi^t(w;c)(x)$. 
Set $\gamma(s):=x+\frac{x-y}{t}(s-t)$. Since  $|x-y|_\infty\le1$, we have $|\gamma'(s)|_\infty\le t^{-1}$ for all $s\in[0,t]$.
Hence, we have 
\begin{eqnarray*}
\int_0^t L^{(c)}(\gamma^\ast(s),s,\gamma^\ast{}'(s))ds
&\le& \int_0^t L^{(c)}(\gamma(s),s,\gamma'(s))ds\\
&\le& \sup_{x\in\T^d,t\in\T,|\zeta|_\infty\le t^{-1},c\in P}|L^{(c)}(x,t,\zeta)|t=:\beta_1(t)t.
\end{eqnarray*}
Since $L^{(c)}(\gamma^\ast(s),s,\gamma^\ast{}'(s))$ is continuous with respect to $s$, we have $s^\ast\in[0,t]$ such that 
\begin{eqnarray*}
L^{(c)}(\gamma^\ast(s^\ast),s^\ast,\gamma^\ast{}'(s^\ast))t=\int_0^t L^{(c)}(\gamma^\ast(s),s,\gamma^\ast{}'(s))ds\le \beta_1(t)t.
\end{eqnarray*}
Due to (L3), there exists a constant $\beta_2(t)>0$ depending only on  $t$ and $\beta_1(t)$ such that $|\gamma^\ast{}'(s^\ast)|\le \beta_2(t)$.
Therefore, with the Euler-Lagrange flow (this is independent of $c$) $\phi_{L}^{s,s_0}=(\phi_{L1}^{s,s_0},\phi_{L2}^{s,s_0}):\R^d\times\R^d\to\R^d\times\R^d$ and the periodicity of $L$ in $(x,t)$, we see that 
\begin{eqnarray*}
\gamma^\ast{}'(s)\in \bigcup_{0\le s_0\le t}\bigcup_{0\le s\le t}\phi_{L2}^{s,s_0}([0,1]^d\times[-\beta_2(t),\beta_2(t)]^d)\mbox{\quad for all $s\in[0,t]$,}
\end{eqnarray*}
where  the right hand side is a compact set independent of $c\in P$ and  $w\in {\rm Lip}(\T^d;\R)$.
\end{proof}
\begin{Lemma}\label{bdd-v}
Let $\beta(1)$ be the number mentioned in Lemma \ref{bdd-of-minimizer} with $t=1$. Set 
$$\tilde{\beta}:=\sup_{x\in\T^d,t\in\T,|\zeta|_\infty\le\beta(1),j} |L_{x^j}(x,t,\zeta)|+\sup_{x\in\T^d,t\in\T, |\zeta|_\infty\le\beta(1),c\in P,j} |L^{(c)}_{\zeta^j}(x,t,\zeta)|.$$ 
Then,  $u(x,t):=\varphi^t(w;c)(x)$ satisfies for any $w\in {\rm Lip}(\T^d;\R)$ and any $x\in\T^d$, 
$$-\tilde{\beta}\le \liminf_{\ep\to0}\frac{u(x+\ep e_i,1)-u(x,1)}{\ep}\le \limsup_{\ep\to0}\frac{u(x+\ep e_i,1)-u(x,1)}{\ep}\le \tilde{\beta}
\quad (i=1,\ldots,d).$$
\end{Lemma}
\begin{proof}
Let $\gamma^\ast:[0,1]\to\T^d$ be a minimizing curve for $u(x,t)$. It is well-known that we have for any $0<\tau<1$,
$$u(x,1)=\int_\tau^1L^{(c)}(\gamma^\ast(s),s,\gamma^\ast{}'(s))ds+u(\gamma^\ast(\tau),\tau)$$
and $u_{x}(\gamma^\ast(\tau),\tau)=L^{(c)}_\zeta(\gamma^\ast(\tau),\tau,\gamma^\ast{}'(\tau))$. By the variational structure and Lemma \ref{bdd-of-minimizer}, we obtain 
\begin{eqnarray*}
&&\limsup_{\ep\to0}\frac{u(x+\ep e_i,1)-u(x,1)}{\ep}\\
&&\qquad \le\limsup_{\ep\to0}\frac{1}{\ep}\Big\{\int^1_\tau L^{(c)}(\gamma^\ast(s)+\ep e_i,s,\gamma^\ast{}'(s))-L^{(c)}(\gamma^\ast(s),s,\gamma^\ast{}'(s))ds\\
&&\qquad\quad +u(\gamma^\ast(\tau)+\ep e_i,\tau)-u(\gamma^\ast(\tau),\tau)\Big\}\\
&&\qquad =\int^1_\tau L_{x^i}(\gamma^\ast(s),s,\gamma^\ast{}'(s))ds+L^{(c)}_{\zeta^i}(\gamma^\ast(\tau),\tau,\gamma^\ast{}'(\tau)) \le \tilde{\beta}.
\end{eqnarray*}
Let $\gamma^\ast_\ep:[0,1]\to\T^d$ be a minimizing curve for $u(x+\ep e_i,t)$. By the variational structure and Lemma \ref{bdd-of-minimizer}, we obtain 
\begin{eqnarray*}
&&\liminf_{\ep\to0}\frac{u(x+\ep e_i,1)-u(x,1)}{\ep}\\
&&\qquad \ge \liminf_{\ep\to0}\frac{1}{\ep}\Big\{\int^1_\tau L^{(c)}(\gamma^\ast_\ep(s),s,\gamma^\ast_\ep{}'(s))-L^{(c)}(\gamma^\ast_\ep(s)-\ep e_i,s,\gamma^\ast_\ep{}'(s))ds\\
&&\qquad\quad +u(\gamma^\ast_\ep(\tau),\tau)-u(\gamma^\ast_\ep(\tau)-\ep e_i,\tau)\Big\}\\
&&\qquad =\liminf_{\ep\to0}\Big\{-\int^1_\tau L_{x^i}(\gamma^\ast_\ep(s),s,\gamma^\ast_\ep{}'(s))ds- 
L^{(c)}_{\zeta^i}(\gamma_\ep^\ast(\tau),\tau,\gamma_\ep^\ast{}'(\tau))\Big\}
\ge-\tilde{\beta}. 
\end{eqnarray*}
\end{proof}
\begin{Thm}\label{bdd}
Set $h_0=(12\tilde{\beta}_0M(1))^{-2}$, where $\tilde{\beta}_0$ is mentioned in the claim 4 of Theorem \ref{convergence-Lax-Oleinik} and $M(1)$ in Theorem \ref{semi}. If $h\in(0,h_0)$ with $0\le \lambda_0\le \lambda:=\tau/h<\lambda_1$, the solution $v^{k}_{m+\1}$ of (\ref{HJ-delta}) satisfies with the constant $\tilde{\beta}$ in Lemma  \ref{bdd-v},
\begin{eqnarray*}
|(D_{x^i}v)^{2K}_m| \le \tilde{\beta}+1\mbox{ \quad for all $m$ ($i=1,\ldots,d$)},
\end{eqnarray*} 
where $\tilde{\beta}_0$, $M(1)$ and $\lambda_1$ depend on $r$ but $\tilde{\beta}$ does not. 
\end{Thm}
\begin{proof}
We proceed by contradiction: Suppose that there exists $m$ such that $|(D_{x^i}v)^{2K}_m| > \tilde{\beta}+1$. We deal with the case $(D_{x^i}v)^{2K}_m > \tilde{\beta}+1$. Take $r_0\in\N$ such that $\frac{1}{3M(1)}\le2hr_0\le \frac{1}{2M(1)}$.  By Theorem \ref{semi}, we have for all $r=0,\ldots,r_0$,
\begin{eqnarray*}
&&(D_{x^i}v)^{2K}_m-(D_{x^i}v)^{2K}_{m-2re_i}\le M(1)\cdot2hr\le \frac{1}{2},\quad (D_{x^i}v)^{2K}_{m-2re_i}>\tilde{\beta}+\frac{1}{2}.\end{eqnarray*}
Let $w_\delta$ be the Lipschitz interpolation of $v^0$, whose Lipschitz constant is bounded by $\beta r$ due to Lemma \ref{interpolation}. Set $u(\cdot,t):=\varphi^t(w_\delta;c)$. Then, we have with the claim 4 of Theorem \ref{convergence-Lax-Oleinik} and Lemma \ref{bdd-v},
\begin{eqnarray*}
2\tilde{\beta}_0 \sqrt{h}&\ge&(v^{2K}_{m+e^i}- u(x_m+he_i,1))-(v^{2K}_{m-2r_0e_i-e_i}-u(x_m-2hr_0e_i-he_i,t)) \\
&=& \sum_{r=0}^{r_0} \int_{x_m^i-2hr-h}^{x_m^i-2hr+h}\Big\{(D_{x^i}v)^{2K}_{m-2re_i}-u_{x^i}(x_m^1,\ldots,x_m^{i-1},y,x_m^{i+1},\ldots,x_m^d)\Big\}dy\\
&\ge&(r_0+1)(\tilde{\beta}+\frac{1}{2}-\tilde{\beta})2h\ge\frac{1}{6M(1)}, 
\end{eqnarray*}
which is a contradiction for $h\in(0,h_0)$. The other case $(D_{x^i}v)^{2K}_m < -\tilde{\beta}-1$ is treated in the same way. 
\end{proof}
Theorem \ref{bdd} implies that the solution $v^k_{m+\1}$ of  (\ref{HJ-delta}) with $r\ge\tilde{\beta}+1$ satisfies
$$|(D_{x}v)^{2K}_m|_\infty \le \tilde{\beta}+1\le r.$$
Hence, we may solve (\ref{HJ-delta}) with initial data $v^{2K}$ up to $k=2K$, which means that we have $v^k_{m+\1}$  up to $k=4K$ with $|(D_{x}v)^{4K}_m|_\infty \le \tilde{\beta}+1\le r$ and $|H_p(x_m,t_k,c+(D_xv)^k_m)|_\infty\le (d\lambda_1)^{-1}$  for all $0\le k\le 4K$ and $m$. 
In this way, the solution $v^k_{m+\1}$ of (\ref{HJ-delta}) can be extended to $k\to\infty$ with 
\begin{eqnarray}\nonumber
&&|(D_{x}v)^{2K\cdot l}_m|_\infty \le \tilde{\beta}+1\le r\mbox{\quad for all  $l\in\N$ and $m$},\\ \label{llllii}
&&|H_p(x_m,t_k,c+(D_xv)^k_m)|_\infty\le (d\lambda_1)^{-1} \mbox{\quad for all $k\ge0$ and $m$}. 
\end{eqnarray}
Furthermore, the bound \eqref{llllii} and the proof of Theorem 2.1 of \cite{Soga5} imply that 
$$v^{l+1}_n=\inf_{\xi} E^{l+1}_n(\xi;v^0,c)\mbox{\quad for all $l\ge0$ and $n$}$$ 
with the unique minimizing control $\xi^\ast$ given as
$$\xi^\ast{}^{k+1}_m=H_p(x_m,t_k,c+(D_xv)^k_m),\,\,\,\,
|\xi^\ast{}^{k+1}_m|_\infty \le (d\lambda_1)^{-1}. $$ 
\begin{Thm}\label{extension}
Under the condition of Theorem \ref{bdd} with $r\ge\tilde{\beta}+1$, the solution map $\varphi^{k}_\delta(\cdot;c)$ of \eqref{HJ-delta}  is extended to all $k\in\N\cup\{0\}$, satisfying 
$|D_x\varphi^{2K\cdot l}_\delta(\cdot;c)|_\infty\le\tilde{\beta}+1$ for all $l\in\N$. 
Furthermore, every minimizing control of  $\varphi^{k}_\delta(\cdot;c)$ is bounded by $(d\lambda_1)^{-1}$.
\end{Thm}
As for \eqref{HJ-delta+}, we obtain the extension of $\tilde{\varphi}^k_\delta(\cdot;c)$ to $k\to-\infty$ .
\setcounter{section}{2}
\setcounter{equation}{0}
\section{Weak KAM theory - non-autonomous case}

We formulate an analogue of weak KAM theory and Aubry-Mather theory for the problem of action minimizing random walks. We construct analogues of the weak KAM objects for each $\delta$ and investigate their hyperbolic scaling limit.   

Suppose that  $\delta=(h,\tau)$ and $r$  are taken so that the assertions in Section 2 hold. Introduce the following notation: 
\begin{eqnarray*}
&&W_\delta:=\{w:G_{\rm odd}\to\R\,|\,w(x_{m+\1\pm2Ne_i})=w(x_{m+\1})\mbox{ for $i=1,\ldots,d$},\,\,\,|(D_xw)|_\infty\le r    \},\\
&&X_\delta:=\{ w|_{x\in\pr G_{\rm odd}} \,|\, w\in W_\delta \},\\
&&\tilde{X}_\delta:=\{u:\pr G_{\rm even}\to\R^d \,|\,u=(D_xw),\,\,w\in W_\delta\},\\
&&\hat{X}_\delta:=\{(u(y_1),u(y_2), \ldots, u(y_a))\in\R^{da} \,|\,u\in \tilde{X}_\delta \},\\
&&\mbox{where the grid points of $\pr G_{\rm even}$ are re-labeled as $y_1,\ldots,y_a$ with $a:=\sharp \pr G_{\rm even}$},\\
&&\mathcal{P}_\delta:\tilde{X}_\delta\ni u\mapsto(u(y_1), \ldots, u(y_a))\in\hat{X}_\delta\quad(\mbox{clearly one to one and onto}).
\end{eqnarray*}
We may reduce the time-1 maps  of (\ref{HJ-delta}) as 
 $$\varphi_\delta(\cdot;c):=\varphi^{2K}_\delta(\cdot;c) :X_\delta\to X_\delta,\quad \psi_\delta(\cdot;c):=\psi^{2K}_\delta(\cdot;c) :\tilde{X}_\delta\to\tilde{X}_\delta.$$
 We may also reduce $\hat{\psi}_\delta(\cdot;c):=\mathcal{P}_\delta\circ \psi_\delta(\cdot;c)\circ\mathcal{P}_\delta^{-1}$ to be 
$$\hat{\psi}_\delta(\cdot;c):\hat{X}_\delta\to\hat{X}_\delta.$$
As for \eqref{HJ-delta+}, we set $\tilde{\varphi}_\delta(\cdot;c):=\tilde{\varphi}^{-2K}_\delta(\cdot;c):X_\delta\to X_\delta$, $\tilde{\psi}_\delta(\cdot;c):=\tilde{\psi}^{-2K}_\delta(\cdot;c):\tilde{X}_\delta\to\tilde{X}_\delta$.
\subsection{Weak KAM solution on grid}
We look for a fixed point of $\psi_\delta(\cdot;c)$, which yields an analogue of a weak KAM solution. We see that  $\hat{X}_\delta$ is a convex compact set. Since the images of $\varphi_\delta$ and $\psi_\delta$ are obtained through a finite number of the four arithmetic operations, they are continuous in the topology of $\sup_{\pr G_{\rm odd}}|\cdot|$ or  $\sup_{\pr G_{\rm even}}|\cdot|_\infty$. Hence, we see that $\hat{\psi}_\delta(\cdot;c)$ is continuous in the topology of $|\cdot|_\infty$ of $\R^{da}$. Brouwer's fixed-point theorem yields a fixed point of $\hat{\psi}_\delta(\cdot;c)$ in $\hat{X}_\delta$, from which we obtain $\bar{v}^0=\bar{v}^0(c)\in X_\delta$  for each $c\in P$ such that 
$$D_x\bar{v}^0=\psi_\delta(D_x\bar{v}^0;c)=D_x\varphi_\delta(\bar{v}^0;c).$$
Therefore, we have a constant $\bar{H}_\delta(c)\in\R$ such that 
\begin{eqnarray}\label{fixed-point}
\varphi_\delta(\bar{v}^0;c)+\bar{H}_\delta(c)=\bar{v}^0.
\end{eqnarray}
This leads to a version of weak KAM theorem: 
\begin{Thm}\label{weak-KAM-delta}
For each $c\in P$, there exists the unique number $\bar{H}_\delta(c)\in\R$ for which we have $\bar{v}^0\in X_\delta$ such that $\varphi_\delta(\bar{v}^0;c)+\bar{H}_\delta(c)=\bar{v}^0$. 
\end{Thm}
 \begin{proof}
Existence of a pair of $\bar{H}_\delta(c)$ and  $\bar{v}^0$  is shown in  \eqref{fixed-point}. We prove the uniqueness of $\bar{H}_\delta(c)$.  Suppose that for some $c$, we have $A\in\R$ and $v^0\in X_\delta$ such that $\varphi_\delta(v^0;c)+A=v^0$. Since we have $(\varphi_\delta)^l(\bar{v}^0;c)=\bar{v}^0-l\bar{H}_\delta(c)$ and $(\varphi_\delta)^l(v^0;c)=v^0-l A$ for any $l\in\N$,  
where $(\varphi_\delta)^l$ stands for $l$-iteration of $\varphi_\delta(\cdot;c)$, it holds that 
\begin{eqnarray*}
(\varphi_\delta)^l(\bar{v}^0;c)(x_{n+\1})&=&E_{\mu_{n+\1}^{2lK,0}(\cdot;\xi^\ast)}\Big[  \sum_{1\le k\le2lK}L^{(c)}(\gamma^k,t_{k-1},\xi^\ast{}^k_{m(\gamma^k)})\tau+\bar{v}^0_{m(\gamma^0)} \Big]\\
&=&\bar{v}^0_{n+\1}-l\bar{H}_\delta(c),\\
(\varphi_\delta)^l(v^0;c)(x_{n+\1})&=&E_{\mu_{n+\1}^{2lK,0}(\cdot;\tilde{\xi}^\ast)}\Big[  \sum_{1\le k\le2lK}L^{(c)}(\gamma^k,t_{k-1},\tilde{\xi}^\ast{}^k_{m(\gamma^k)})\tau+v^0_{m(\gamma^0)} \Big]\\
&=&v^0_{n+\1}-lA,
\end{eqnarray*} 
where $\xi^\ast,\tilde{\xi}^\ast$ are minimizing controls. Then, we have 
 \begin{eqnarray*}
v^0_{n+\1}-lA&\le& E_{\mu_{n+\1}^{2lK,0}(\cdot;\xi^\ast)}\Big[  \sum_{1\le k\le2lK}L^{(c)}(\gamma^k,t_{k-1},\xi^\ast{}^k_{m(\gamma^k)})\tau+v^0_{m(\gamma^0)} \Big], \\
\bar{H}_\delta(c)-A&\le& \frac{1}{l}E_{\mu_{n+\1}^{2lK,0}(\cdot;\xi^\ast)}[v^0_{m(\gamma^0)}-\bar{v}^0_{m(\gamma^0)}]-\frac{1}{l}(v^0_{n+\1}-\bar{v}^0_{n+\1}).
\end{eqnarray*}
Letting $l\to\infty$, we have  $\bar{H}_\delta(c)-A\le0$. Similar reasoning yields  $\bar{H}_\delta(c)-A\ge0$.
\end{proof}
As for \eqref{HJ-delta+}, we have a pair of the unique constant $\tilde{\bar{H}}_\delta(c)$ and a function $\tilde{\bar{v}}^0$ such that  $\tilde{\varphi}_\delta(\tilde{\bar{v}}^0;c)-\tilde{\bar{H}}_\delta(c)=\tilde{\bar{v}}^0$ (we set ``$-$'' to have the statements following \eqref{HJ-delta2+}).

In Section 4, we treat the autonomous case and  prove that $(\varphi_\delta)^l(v^0;c)+\bar{H}_\delta(c)$ converges to some $\bar{v}^0$ as $l\to\infty$ for each  $v^0\in X_\delta$, where $\bar{v}^0$ depends on $v^0$.

We discuss (up to constants) uniqueness  of a solution  $\bar{v}^0$ of  \eqref{fixed-point}. When $d=1$, it is unique up to constants \cite{Nishida-Soga}, \cite{Soga3}. The proof is based on the strict contraction property of the derivative of \eqref{HJ-delta}, i.e., if $d=1$, the derivatives $u^k_m:=(D_xv)^k_{m+1}$, $\tilde{u}^k_m:=(D_x\tilde{v})^k_{m+1}$ of arbitrary two solutions $v$, $\tilde{v}$ satisfy for each $l$,
\begin{eqnarray}\label{contraction}   
\sum_{\{m\,|\,(x_{m+1},t_{l+1}) \in\G_\delta\}}|\tilde{u}^{l+1}_{m+1}-u^{l+1}_{m+1}|\le \sum_{\{m\,|\,(x_{m},t_{l})\in \G_\delta\}} |\tilde{u}^{l}_{m}-u^{l}_{m}|,  
\end{eqnarray}
which becomes a strict inequality for some $l$. This implies that  $(\varphi_\delta)^l(v^0;c)+\bar{H}_\delta(c)l$ converges to some $\bar{v}^0$ satisfying \eqref{fixed-point} as $l\to\infty$ for each  $v^0\in X_\delta$, and that $\bar{v}^0$ satisfying \eqref{fixed-point} is unique up to constants, even in the non-autonomous case. 
One could expect a similar inequality for $d\ge2$, e.g.,  the derivatives $u^j{}^k_m:=(D_{x^j}v)^k_{m}$, $\tilde{u}^j{}^k_m:=(D_{x^j}\tilde{v})^k_{m}$ of arbitrary two solutions $v$, $\tilde{v}$ might satisfy  for each $l$,
\begin{eqnarray}\label{contraction2}   
\sum_{\{m\,|\,(x_{m+\1},t_{l+1}) \in\G_\delta\}} \sum_{j=1}^d   |\tilde{u}^j{}^{l+1}_{m+\1}-u^j{}^{l+1}_{m+\1}|\le \sum_{\{m\,|\,(x_{m},t_{l})\in \G_\delta\}}  \sum_{j=1}^d   |\tilde{u}^j{}^{l}_{m}-u^j{}^{l}_{m}|.
\end{eqnarray}
Surprisingly, application of the same reasoning for (\ref{contraction}) to a possible proof of \eqref{contraction2} fails. Therefore, (up to constants) uniqueness of $\bar{v}^0$ is not clear for $d\ge2$. We remark that no such contraction property is known for $d\ge2$ in the theory of  hyperbolic systems of conservation laws (the derivative $u:=v_x$ of a viscosity solution of \eqref{HJ} satisfies a $d$-system of conservation laws). It would be worth demonstrating why reasoning for \eqref{contraction} fails to yield   \eqref{contraction2}.   By  (\ref{HJ-delta}), we have 
\begin{eqnarray*}  
 v^{l+1}_{m+\1+e_j}&=&\frac{1}{2d}\sum_{i=1}^d (v^l_{m+\1+e_j+e_i}+v^l_{m+\1+e_j-e_i})-H(x_{m+\1+e_j},t_l,u^l_{m+\1+e_j})\tau,\\
 v^{l+1}_{m+\1-e_j}&=&\frac{1}{2d}\sum_{i=1}^d (v^l_{m+\1-e_j+e_i}+v^l_{m+\1-e_j-e_i})-H(x_{m+\1-e_j},t_l,u^l_{m+\1-e_j})\tau, 
\end{eqnarray*} 
which yields 
\begin{eqnarray*}
u^j{}^{l+1}_{m+\1}&=&\frac{1}{2d}\sum_{i=1}^d (u^j{}^l_{m+\1+e_i}+u^j{}^l_{m+\1-e_i})\\
&&-\frac{\lambda}{2}\Big(H(x_{m+\1+e_j},t_l,u^l_{m+\1+e_j})-H(x_{m+\1-e_j},t_l,u^l_{m+\1-e_j})\Big).
\end{eqnarray*}
Hence, we obtain 
\begin{eqnarray*}
\tilde{u}^j{}^{l+1}_{m+\1}-u^j{}^{l+1}_{m+\1}&=&\frac{1}{2d}\sum_{i=1}^d \{ (\tilde{u}^j{}^l_{m+\1+e_i}-u^j{}^l_{m+\1+e_i})+(\tilde{u}^j{}^l_{m+\1-e_i}-u^j{}^l_{m+\1-e_i}) \}\\
&&-\frac{\lambda}{2}\Big(H(x_{m+\1+e_j},t_l,\tilde{u}^l_{m+\1+e_j})-H(x_{m+\1+e_j},t_l,u^l_{m+\1+e_j})\\
&&-H(x_{m+\1-e_j},t_l,\tilde{u}^l_{m+\1-e_j})+H(x_{m+\1-e_j},t_l,u^l_{m+\1-e_j})\Big)\\
&=&\frac{1}{2d}\sum_{i=1}^d \{ (\tilde{u}^j{}^l_{m+\1+e_i}-u^j{}^l_{m+\1+e_i})+(\tilde{u}^j{}^l_{m+\1-e_i}-u^j{}^l_{m+\1-e_i}) \}\\
&&-\frac{\lambda}{2}\sum_{i=1}^d\{ \zeta^i{}^l_{m+\1+e_j} (\tilde{u}^i{}^l_{m+\1+e_j}-u^i{}^l_{m+\1+e_j})\\
&&-\zeta^i{}^l_{m+\1-e_j}(\tilde{u}^{il}_{m+\1-e_j}-u^{il}_{m+\1-e_j})\},
\end{eqnarray*} 
where $ \zeta^i{}^l_{m+\1\pm e_j}:=H_{p^i}(x_{m+\1\pm e_j},t_l,u^l_{m+\1\pm e_j}+\theta^l_{m+\1\pm e_j} (\tilde{u}^l_{m+\1\pm e_j}-u^l_{m+\1\pm e_j}))$ with $\theta^l_{m+\1\pm e_j}\in(0,1)$  are given by Taylor's formula. Let $\sigma^j{}^{l+1}_{m+\1}=\pm1$ denote the sign of  $\tilde{u}^j{}^{l+1}_{m+\1}-u^j{}^{l+1}_{m+\1}$. Then, we have  
\begin{eqnarray*}
\sum_{j=1}^d|\tilde{u}^j{}^{l+1}_{m+\1}-u^j{}^{l+1}_{m+\1}|&=&\sum_{i,j=1}^d \Big( \frac{1}{2d}\sigma^i{}^{l+1}_{m+\1} -\frac{\lambda}{2} \zeta^i{}^{l}_{m+\1+e_j}\sigma^j{}^{l+1}_{m+\1}\Big)\big(\tilde{u}^i{}^l_{m+\1+e_j}-u^i{}^l_{m+\1+e_j}\big)\\
&+&\sum_{i,j=1}^d \Big( \frac{1}{2d}\sigma^i{}^{l+1}_{m+\1} +\frac{\lambda}{2} \zeta^i{}^{l}_{m+\1-e_j}\sigma^j{}^{l+1}_{m+\1}\Big)\big(\tilde{u}^i{}^l_{m+\1-e_j}-u^i{}^l_{m+\1-e_j}\big).
\end{eqnarray*} 
Due to the periodic boundary condition, we have  
\begin{eqnarray*}
&&\sum_{\{m\,|\,(x_{m+\1},t_{l+1}) \in\G_\delta\}}\sum_{j=1}^d|\tilde{u}^j{}^{l+1}_{m+\1}-u^j{}^{l+1}_{m+\1}|\\
&&\qquad=\sum_{\{m\,|\,(x_{m},t_{l}) \in\G_\delta\}}\sum_{i,j=1}^d \Big( \frac{1}{2d}\sigma^i{}^{l+1}_{m-e_j} -\frac{\lambda}{2} \zeta^i{}^{l}_{m}\sigma^j{}^{l+1}_{m-e_j}\Big)\big(\tilde{u}^i{}^l_{m}-u^i{}^l_{m}\big)\\
&&\qquad\quad+\sum_{\{m\,|\,(x_{m},t_{l}) \in\G_\delta\}}\sum_{i,j=1}^d \Big( \frac{1}{2d}\sigma^i{}^{l+1}_{m+e_j} +\frac{\lambda}{2} \zeta^i{}^{l}_{m}\sigma^j{}^{l+1}_{m+e_j}\Big)\big(\tilde{u}^i{}^l_{m}-u^i{}^l_{m}\big)\\
&&\qquad=\sum_{\{m\,|\,(x_{m},t_{l}) \in\G_\delta\}}\sum_{i=1}^d |\tilde{u}^i{}^l_{m}-u^i{}^l_{m}| +R^l,
\end{eqnarray*} 
where 
\begin{eqnarray*}
R^l&=&\sum_{\{m\,|\,(x_{m},t_{l}) \in\G_\delta\}}\sum_{i,j=1}^d \kappa^{ij}_m|\tilde{u}^i{}^l_{m}-u^i{}^l_{m}|,\\
\kappa^{ij}_m&=&-\frac{1}{d}+\Big\{\Big( \frac{1}{2d}\sigma^i{}^{l+1}_{m-e_j} -\frac{\lambda}{2} \zeta^i{}^{l}_{m}\sigma^j{}^{l+1}_{m-e_j}\Big)+\Big( \frac{1}{2d}\sigma^i{}^{l+1}_{m+e_j} +\frac{\lambda}{2} \zeta^i{}^{l}_{m}\sigma^j{}^{l+1}_{m+e_j}\Big)\Big\}\sigma^i{}^l_{m}.
\end{eqnarray*}
Note that 
\begin{enumerate}
\item[(i)] If $\sigma^j{}^{l+1}_{m-e_j}+\sigma^j{}^{l+1}_{m+e_j}=\pm2$, we have $\kappa^{ij}_m=-1/d\pm(\sigma^i{}^{l+1}_{m-e_j}+\sigma^i{}^{l+1}_{m+e_j})/(2d)=0$, $-1/d$ or $-2/d$.
\item[(ii)] If $\sigma^j{}^{l+1}_{m-e_j}+\sigma^j{}^{l+1}_{m+e_j}=0$, we have $\kappa^{ij}_m=-1/d \pm \{(\sigma^i{}^{l+1}_{m-e_j}+\sigma^i{}^{l+1}_{m+e_j})/(2d)\pm \lambda\zeta^{il}_m\}$, which is not necessarily non-positive. When $d=1$, we have $\sigma^j{}^{l+1}_{m\pm e_j}=\sigma^i{}^{l+1}_{m\pm e_j}$ to obtain $\kappa^{ij}_m=-1/d\pm\lambda\zeta^{il}_m<0$ due to the CFL condition mentioned in Theorem \ref{main1}.
\end{enumerate}
Hence, the above argument cannot yield $R^l\le0$ for $d\ge2$ in general. 
%
\subsection{Property of effective Hamiltonian}
We call $\bar{H}_\delta(\cdot)$ the effective Hamiltonian of (\ref{HJ-delta}). Here are the properties of $\bar{H}_\delta(\cdot)$:
\begin{Thm}\label{effective}
\begin{enumerate}
\item The problem with a constant $A$ 
 \begin{eqnarray}\label{HJ-delta2}
 \left\{
\begin{array}{lll}
&v:\tG_\delta|_{k\ge0}\to\R,\medskip \\
&v^k_{m+\1\pm2Ne_i}=v^k_{m+\1}\quad (i=1,\ldots,d),\medskip\\
&v^0\in X_\delta,\medskip\\
&(D_tv)^{k+1}_m+H(x_m,t_k,c+(D_xv)^k_m)=A
\end{array}
\right.
\end{eqnarray}
admits a time-$1$-periodic solution (i.e., $v^{k+2K}=v^k$ for all $k\ge0$), if and only if $A=\bar{H}_\delta(c)$ and $v^0$ satisfies $\varphi_\delta(v^0;c)+\bar{H}_\delta(c)=v^0$. 
\item  Let $\bar{v}$ be a time-$1$-periodic solution of \eqref{HJ-delta2} with $A=\bar{H}_\delta(c)$. Then, we have 
$$ \bar{H}_\delta(c)=\sum_{ \{(m,k)\,|\,(x_m,t_k)\in\pr\G_\delta\}}H(x_m,t_k,c+(D_x\bar{v})_m^k)(2h)^d\tau.$$
\item Let $v^k_{m+\1}$ be the solution of \eqref{HJ-delta} with  arbitrary initial date $v^0\in X_\delta$. Then, we have 
$$\lim_{l \to \infty}\frac{v^{l+1}_{n}}{t_{l+1}}=-\bar{H}_\delta(c)\mbox{\quad for any $n$}.$$
\item $\bar{H}_\delta(\cdot):P\to\R$ is convex (hence, Lipschitz continuous).      
\end{enumerate}
\end{Thm}
\begin{proof}
1. The time-1 map of \eqref{HJ-delta2},  denoted by $\check{\varphi}_\delta(\cdot;c)$, is given as $\check{\varphi}_\delta(\cdot;c)=\varphi_\delta(\cdot;c)+A$. Note that a solution of \eqref{HJ-delta2} is time-$1$-periodic, if and only if $v^0$ is a fixed point of $\check{\varphi}_\delta(\cdot;c)$. 
 If $A=\bar{H}_\delta(c)$ and $\varphi_\delta(v^0;c)+\bar{H}_\delta(c)=v^0$, we have $\check{\varphi}_\delta(v^0;c)=\varphi_\delta(v^0;c)+\bar{H}_\delta(c)=v^0$. 
If $\check{\varphi}_\delta(v^0;c)=v^0$, we have $\varphi(v^0;c)+A=v^0$;  Theorem \ref{weak-KAM-delta} implies that $A$ must be $\bar{H}_\delta(c)$.

2. Since $\bar{v}$ is time-$1$-periodic, we have 
\begin{eqnarray*}
 \bar{H}_\delta(c)&=&\sum_{\{(m,k)\,|\,(x_m,t_k)\in\pr\G_\delta\}}H(x_m,t_k,c+(D_x\bar{v})_m^k)(2h)^d\tau\\
&&+\sum_{\{(m,k)\,|\,(x_m,t_k)\in\pr\G_\delta\}}\Big(\bar{v}^{k+1}_m-\frac{1}{2d}\sum_{\omega\in B} \bar{v}^k_{m+\omega}\Big)(2h)^d,
\end{eqnarray*}
where  the second summation is $0$ due to periodicity of $\bar{v}$.

3. Let $\xi^\ast$ be the minimizing control of $v^n_{n+\1}$, i.e.,
$$v^{l+1}_n= E_{\mu_n^{l+1,0}(\cdot;\xi^\ast)}\Big[  \sum_{1\le k\le l+1}L^{(c)}(\gamma^k,t_{k-1},\xi^\ast{}^k_{m(\gamma^k)})\tau+v^0_{m(\gamma^0)} \Big].$$
Let $\bar{v}$ be a time-$1$-periodic solution of \eqref{HJ-delta2} with $A=\bar{H}_\delta(c)$. We have  the minimizing control $\hat{\xi}^\ast$ of $\bar{v}^{l+1}_n$, i.e., 
$$\bar{v}^{l+1}_n= E_{\mu_n^{l+1,0}(\cdot;\hat{\xi}^\ast)}\Big[  \sum_{1\le k\le l+1}L^{(c)}(\gamma^k,t_{k-1},\hat{\xi}^\ast{}^k_{m(\gamma^k)})\tau+\bar{v}^0_{m(\gamma^0)} \Big]+\bH_\delta(c)t_{l+1}.$$
Due to the variational property, we have 
\begin{eqnarray*}
v^{l+1}_n-\bar{v}^{l+1}_n&\le& E_{\mu_n^{l+1,0}(\cdot;\hat{\xi}^\ast)}\Big[  v^0_{m(\gamma^0)}-\bar{v}^0_{m(\gamma^0)} \Big]-\bH_\delta(c)t_{l+1},\\
v^{l+1}_n-\bar{v}^{l+1}_n&\ge& E_{\mu_n^{l+1,0}(\cdot;\xi^\ast)}\Big[  v^0_{m(\gamma^0)}-\bar{v}^0_{m(\gamma^0)} \Big]-\bH_\delta(c)t_{l+1}.
\end{eqnarray*}
We divide the inequalities by $t_{l+1}$ and sending $l\to\infty$. Due to boundedness of $v^0$, $\bar{v}^0$ and $\bar{v}$, we obtain the assertion. 

4. Consider the solution $v^k_{m+\1}=v^k_{m+\1}(c)$ of \eqref{HJ-delta} for each $c$ with common initial date $v^0\in X_\delta$. It is enough to show that the function $c\mapsto v^{l+1}_{n}(c)$ is concave for fixed $l,n$. In fact, if so, $c\mapsto v^{l+1}_{n}(c)/t_{l+1}$ is also concave for any $l$; due to the claim 3, $\lim_{l \to \infty}v^{l+1}_{n}(c)/t_{l+1}=-\bar{H}_\delta(c)$ is also concave.     
Let $\xi^\ast$ be the minimizing control for $v^{l+1}_n(c^\ast)$ with $c^\ast:=\theta c+(1-\theta)\tilde{c}$, $\theta\in[0,1]$, $c,\tilde{c}\in P$. Then, the variational property yields 
\begin{eqnarray*}
v^{l+1}_n(c^\ast)&-& \{ \theta v^{l+1}_n(c)+(1-\theta)v^{l+1}_n(\tilde{c})\}\\
 &\ge& \theta E_{\mu(\cdot;\xi^\ast)}\Big[    \sum_{1\le k\le l+1} -(c^\ast-c)\cdot \xi^\ast{}^k_{m(\gamma^k)}\tau\big]\\
 &&+(1-\theta) E_{\mu(\cdot;\xi^\ast)}\Big[    \sum_{1\le k\le l+1} -(c^\ast-\tilde{c})\cdot \xi^\ast{}^k_{m(\gamma^k)}\tau\Big]\\
 &=&0.
\end{eqnarray*}
\end{proof}
As for the pair of $\tilde{\bar{H}}_\delta(c)$ and $\tilde{\bar{v}}^0$, we have a similar result. In particular, it holds that the problem with a constant $\tilde{A}$ 
 \begin{eqnarray}\label{HJ-delta2+}
 \left\{
\begin{array}{lll}
&v:\tG_\delta|_{k\le0}\to\R,\medskip \\
&v^k_{m+\1\pm2Ne_i}=v^k_{m+\1}\quad (i=1,\ldots,d),\medskip\\
&v^0\in X_\delta,\medskip\\
&(\tilde{D}_tv)^{k-1}_m+H(x_m,t_{k},c+(D_xv)^k_m)=\tilde{A}
\end{array}
\right.
\end{eqnarray}
admits a time-$1$-periodic solution $\tilde{\bar{v}}_\delta(c)$, if and only if $\tilde{A}=\tilde{\bar{H}}_\delta(c)$ and $v^0$ satisfies $\tilde{\varphi}_\delta(\tilde{\bar{v}}^0;c)-\tilde{\bar{H}}_\delta(c)=\tilde{\bar{v}}^0$. We remark that $\tilde{\bar{H}}_\delta(c)=\bar{H}_\delta(c)$ would not be true in general. In fact, the reasoning to compare $\bar{H}_\delta(c)$ and $A$ demonstrated in the proof of Theorem \ref{weak-KAM-delta} does not work in the comparison of $\bar{H}_\delta(c)$ and $\tilde{\bar{H}}_\delta(c)$, because switching a minimizer and maximizer between the two variational problems for $\varphi^k_\delta$ and $\tilde{\varphi}^k_\delta$ does not make any sense. Nevertheless, as we will see later, $\bar{H}_\delta(c)$ and $\tilde{\bar{H}}_\delta(c)$ converge to the same quantity at the hyperbolic scaling limit. 
%
\subsection{Asymptotics of minimizing random walk and Mather measure}

From now on, we consider time-$1$-periodic solutions $\bar{v}$ of 
\begin{eqnarray}\label{HJ-delta-cell}
\quad \left\{
\begin{array}{lll}
\!\!\!\!&(D_tv)^{k+1}_m+H(x_m,t_k,c+(D_xv)^k_m)=\bar{H}_\delta(c),
\medskip\\
&v^k_{m+\1+2Ne_j}=v^k_{m+\1}\quad (j=1,\ldots,d), \medskip\\
\end{array}
\right.
\end{eqnarray}
where $\bar{v}$ is also denoted by $\bar{v}(c)$ when we specify the value of $c$. Note that the solution map of \eqref{HJ-delta-cell} is given as $\varphi^k_\delta(\cdot;c)+\bar{H}_\delta(c)t_{k}$.  If $\bar{v}(c)$ is a time-$1$-periodic solution of \eqref{HJ-delta-cell}, we may periodically extend $\bar{v}(c)$ to the whole $\tG_\delta$; controlled random walks and the representation formula stated in Theorem \ref{main1} and \ref{extension} are well-defined with an arbitrary (negative) terminal time, i.e., we have  for any $l,l'\in\Z$ with $l'\le l$ and $n$,
 $$\bar{v}(c)^{l+1}_n=\inf_\xi E_{\mu_n^{l+1,l'}(\cdot;\xi)}\Big[ \sum_{l'<k\le l+1}L^{(c)}(\gamma^k,t_{k-1},\xi^k_{m(\gamma^k)})\dt+\bar{v}(c)_{m(\gamma^{l'})}^{l'} \Big]+\bar{H}_\delta(c)(t_{l+1}-t_{l'}),$$ 
where its minimizing control is given as 
$$\mbox{$\xi^\ast{}^{k+1}_{m}=H_p(x_m,t_k,c+(D_x\bar{v})^k_m)$, $|\xi^\ast{}^{k+1}_{m}|_\infty\le (d\lambda_1)^{-1}$};$$ 
the minimizing control $\xi^\ast$ and minimizing random walk can be extended to $l'\to-\infty$ with   
$\mbox{$\xi^\ast{}^{k+1}_{m}=H_p(x_m,t_k,c+(D_x\bar{v})^k_m)$ on $\tG_\delta$}. $
Then, we are ready to consider the action minimizing problem 
$$\inf_{\xi}\mathcal{L}^l_-(\xi;\bar{v})$$
 stated in Introduction.  We will see that the asymptotic behavior of minimizing random walks is very  similar to that of calibrated curves in weak KAM theory. Furthermore, if we look at  the asymptotic behavior on $\pr \tG_\delta$, we find an object very similar to Mather measures.  Roughly speaking, a Mather measure shows ``recurrence rate'' of trajectories of Euler-Lagrange system  in standard autonomous weak KAM theory (see construction of a Mather measure by calibrated curves \cite{Fathi-book}). 
We will observe that, in our problem, the quantity corresponding to``recurrence rate''  is given by the long-time average of the (configuration-space-based) probability measure of a minimizing random walk projected on $\pr \tG_\delta$.     

The next theorem is reminiscent of  the backward rotation vector of calibrated curves. 
\begin{Thm}
Let $c\in P$ be such that $\bar{H}_\delta(c)$ is differentiable (a.e. points have such a property due to  convexity). Let $\bar{v}(c)$ be a time-$1$-periodic solution of \eqref{HJ-delta-cell}. Then, every minimizing random walk of $\bar{v}(c)$ has the backward rotation vector $\frac{\partial}{\partial c}\bar{H}_\delta(c)$, i.e., the average $\bar{\gamma}^k$, $l'\le k\le l+1$ of the minimizing random walk (see  \eqref{averaged-path}) for $\bar{v}^{l+1}_n$ satisfies  
$$\lim_{l'\to-\infty}\frac{\bar{\gamma}^{l'}}{t_{l'}}=\frac{\partial}{\partial c}\bar{H}_\delta(c),$$
where we note that the averaged path defined for $l'\le k\le l+1$ and that defined for $l''\le k\le l+1$ are identical for $\max\{l',l''\}\le k\le l+1$.  
\end{Thm}
\begin{proof}
Let $\xi^\ast$ be the minimizing control for  $\bar{v}^{l+1}_n(c)$ up to $l'\le l$.    
Recall \eqref{averaged-path}: $\bar{\gamma}^k=\bar{\gamma}^{k+1}-\bar{\xi}^\ast{}^{k+1}\tau$.    
For $\bar{v}^{l+1}_n(\tilde{c})$ with $\tilde{c}\neq c$, we have 
\begin{eqnarray*}
\bar{v}^{l+1}_n(\tilde{c})&\le&E_{\mu^{l+1,l'}_n(\cdot;\xi^\ast)}\Big[ \sum_{l'<k\le l+1}\big\{L^{(\tilde{c})}(\gamma^k,t_{k-1},\xi^\ast{}^k_{m(\gamma^k)})\big\}\tau+\bar{v}^{l'}_{m(\gamma^{l'})}(\tilde{c}) \Big]\\
&&+\bar{H}_\delta(\tilde{c})(t_{l+1}-t_{l'}),\\ 
\bar{v}^{l+1}_n(\tilde{c})-\bar{v}^{l+1}_n(c)&\le& E_{\mu^{l+1,l'}_n(\cdot;\xi^\ast)}\Big[\sum_{l'<k\le l+1}-(\tilde{c}-c)\cdot\xi^\ast{}^k_{m(\gamma^k)}\tau+\bar{v}^{l'}_{m(\gamma^{l'})}(\tilde{c})-\bar{v}^{l'}_{m(\gamma^{l'})}(c) \Big]\\
&&+(\bar{H}_\delta(\tilde{c})-\bar{H}_\delta(c))(t_{l+1}-t_{l'}).
\end{eqnarray*}
Hence, we have
\begin{eqnarray*}
&&\frac{\bar{v}^{l+1}_n(\tilde{c})-\bar{v}^{l+1}_n(c)-E_{\mu^{l+1,l'}_n(\cdot;\xi^\ast)}\Big[\bar{v}^{l'}_{m(\gamma^{l'})}(\tilde{c})-\bar{v}^{l'}_{m(\gamma^{l'})}(c) \Big]}{t_{l+1}-t_{l'}}\\
&&\qquad\qquad\qquad\le -\frac{\tilde{c}-c}{t_{l+1}-t_{l'}}\cdot \sum_{l'<k\le l+1} \bar{\xi}^\ast{}^k \tau+\bar{H}_\delta(\tilde{c})-\bar{H}_\delta(c)\\
&&\qquad\qquad\qquad= -(\tilde{c}-c)\cdot \frac{x_n-\bar{\gamma}^{l'}}{t_{l+1}-t_{l'}}+\bar{H}_\delta(\tilde{c})-\bar{H}_\delta(c).
\end{eqnarray*}
Since $\bar{v}^k_{m+1}(c)$, $\bar{v}^k_{m+1}(\tilde{c})$ are uniformly bounded, we obtain 
\begin{eqnarray*}
0&\le&-(\tilde{c}-c)\cdot\Big(\liminf_{l'\to-\infty}\frac{\bar{\gamma}^{l'}}{t_{l'}}\Big)+\bar{H}_\delta(\tilde{c})-\bar{H}_\delta(c)\\
&\le&-(\tilde{c}-c)\cdot\Big(\limsup_{l'\to-\infty}\frac{\bar{\gamma}^{l'}}{t_{l'}}\Big)+\bar{H}_\delta(\tilde{c})-\bar{H}_\delta(c).
\end{eqnarray*}
For $\tilde{c}= c+\ep e_i$ and $\ep\to0+$, we have
$$e_i\cdot\Big(\limsup_{l'\to-\infty}\frac{\gamma^{l'}}{t_{l'}}\Big)\le \frac{\bar{H}_\delta(\tilde{c})-\bar{H}_\delta(c)}{\ep}\to e_i\cdot\frac{\partial}{\partial c}\bar{H}_\delta(c).$$
For $\tilde{c}= c+\ep e_i$ and $\ep\to0-$, we have
$$e_i\cdot\Big(\liminf_{l'\to-\infty}\frac{\gamma^{l'}}{t_{l'}}\Big)\ge \frac{\bar{H}_\delta(\tilde{c})-\bar{H}_\delta(c)}{\ep}\to e_i\cdot\frac{\partial}{\partial c}\bar{H}_\delta(c).$$
\end{proof}

We observe asymptotics of the probability measures of controlled random walks  on $\pr \tG_\delta$ and derive an analogue of Mather's minimizing problem as well as Mather measures.  
We call $\xi:\pr\tG_\delta \to[-(d\lambda)^{-1},(d\lambda)^{-1}]^d$, $\lambda:=\tau/h<\lambda_1$ an admissible control, where we do not distinguish $\xi$ and its $1$-periodic extension to $\tG_\delta$. Let $\mu^{0,-l}_{n+\1}(\cdot; \xi|_{G^{0,-l}_{n+\1}})$, re-denoted simply by $\mu^{0,-l}_{n+\1}(\cdot; \xi)$,  be the probability density of the backward random walk controlled by $\xi$, where we may take any $l\in\N$. For each admissible control $\xi$, define the linear functional 
\begin{eqnarray*}
&&\mathcal{F}^l_\delta(\cdot;\xi): C_c(\T^d\times\T\times\R^d;\R)\to \R,\\
&&\mathcal{F}^l_\delta(f;\xi):=E_{\mu^{0,-l}_{n+\1}(\cdot; \xi)}\Big[\frac{1}{t_{l}}\sum_{-l<k\le0}f(\gamma^k,t_{k},\xi^k_{m(\gamma^k)})\tau\Big],
\end{eqnarray*}
where $C_c(\T^d\times\T\times\R^d;\R)$ is the family of compactly supported continuous functions defined in  $\T^d\times\T\times\R^d$. 
The Riesz representation theorem yields the unique  probability measure $\mu^l_\delta(\xi)$ of $\T^d\times\T\times\R^d$ such that 
\begin{eqnarray*}
\mathcal{F}^l_\delta(f;\xi)=\int_{\T^d\times\T\times\R^d}f \,\,d \mu^{l}_\delta(\xi)\mbox{\quad for any $f\in C_c(\T^d\times\T\times\R^d;\R)$}.
\end{eqnarray*}
We see that,  for all $l\in\N$,  supp\,($\mu^l(\xi)_\delta$) is contained in the compact set 
$$K_\delta:=\{(x_{m+\1},t_{k},\xi^{k}_{m+\1})\,|\, (x_{m+\1},t_{k})\in \pr \tG_\delta\}\subset \T^d\times\T\times[-(d\lambda)^{-1},(d\lambda)^{-1}]^d.$$ 
Hence, we have a sequence $l\to\infty$ for which $\mu^l_\delta(\xi)$ converges weakly to a probability measure $\mu_\delta(\xi)$ of $\T^d\times\T\times\R^d$. 
Let $\mathcal{P}_\delta$ be the family of all such probability measures for all admissible controls.  

We demonstrate more concrete construction of each $\mu_\delta(\xi)$ and $\mathcal{P}_\delta$ by means of the configuration-space-based  distribution $p^k_{m+\1}(\xi)$ (mentioned in Subsection 2.2) of the random walk starting at $(x_{m+\1},0)$ and controlled by $\xi$.  
Re-define $p^k_{m+\1}(\xi)$ as $p(\cdot;\xi):\tG_\delta\to[0,1]$; $p(\cdot;\xi)=0$ outside $G^{0,-\infty}_{n+\1} $; otherwise  
$$p(x_{n+\1},0)(\xi):=1,\quad p(x_{m+\1},t_k;\xi):=\sum_{\{\gamma\in \Omega^{0,-l}_{n+\1}\,|\,\gamma^{k}=x_{m+\1}\}}\mu^{0,-l}_{n+\1}(\gamma; \xi)\mbox{\quad with $-l\le k$},$$
where we note that $p(x_{m+\1},t_k;\xi)$ is independent of  the choice of $l$. We have for each $k\le0$, 
\begin{eqnarray}\label{313131}
 \sum_{\{x\,|\,(x,t_k)\in \tG_\delta\} } p(x,t_k; \xi) =1.
 \end{eqnarray}
 Observe that 
\begin{eqnarray*}
\mathcal{F}^l_\delta(f;\xi)&=&\frac{1}{t_{l}}  \sum_{\gamma\in\Omega^{0,-l}_{n+\1}}  \mu^{0,-l}_{n+\1}(\gamma;\xi) \Big(\sum_{-l<k\le0}f(\gamma^k,t_{k},\xi^k_{m(\gamma^k)})\tau\Big)\\
&=&\frac{1}{t_{l}}\sum_{-l<k\le0}    \Big( \sum_{\gamma\in\Omega^{0,-l}_{n+\1}}  \mu^{0,-l}_{n+\1}(\gamma;\xi)    f(\gamma^k,t_{k},\xi^k_{m(\gamma^k)})\Big)\tau\\
&=&\frac{1}{t_{l}}\sum_{-l<k\le0}    \Big( \sum_{\{m\,|\,(x_{m+\1},t_k)\in \tG_\delta \}} p(x_{m+\1},t_k; \xi) f(x_{m+\1},t_{k},\xi^k_{m+\1})\Big)\tau.
\end{eqnarray*}
Define $p^l_\delta(\cdot;\xi):\pr\tG\to[0,1]$ as 
$$p^l_\delta(x_{m+\1},t_k;\xi):=\frac{1}{t_l}\sum_{\{(m',k')\,|\, \pr(x_{m'+\1},t_{k'})=(x_{m+\1},t_k), \,\,-l<k'\le0  \}} p(x_{m'+\1},t_{k'}; \xi)\tau,$$
where,  for $y\in\R^q$, $\pr y:=y+z$ with $z\in\Z^q$ such that $y+z\in[0,1)^q$.    We note that \eqref{313131} implies  
$$\sum_{(x,t)\in\pr\tG} p^l_\delta(x,t;\xi)=1\mbox{\quad for each $l\in\N$}.$$
Due to the periodicity of $f\in C_c(\T^d\times\T\times\R^d;\R)$ with respect to $(x,t)$, we have
\begin{eqnarray*}
\mathcal{F}^l_\delta(f;\xi)&=&\sum_{(x,t)\in\pr\tG} p^l_\delta(x,t;\xi)f(x,t,\xi^{k(t)}_{m(x)})
\end{eqnarray*}
Now, we see that the above $\mu^l_\delta(\xi)$ is represented as   
$$\mu^l_\delta(\xi)=  \sum_{(x,t)\in\pr\tG} p^l_\delta(x,t;\xi) \mathfrak{d}_{x,t,\xi^{k(t)}_{m(x)}},$$
where $\mathfrak{d}_{x,t,\zeta}$ is the Dirac measure of $\T^d\times\T\times\R^d$ supported by the point $(x,t,\zeta)$.  Since the function $p^l_\delta(\cdot;\xi)$ is determined by a finite number of values in $[0,1]$,  the sequence $\{p^l_\delta(\cdot;\xi)\}_{l\in\N}$ can be regarded as a sequence of $[0,1]^{da}$ with $a=\sharp (\pr\tG)$. Hence, we find a convergent subsequence with the limit $p_\delta(\cdot;\xi)$. Let $\mathcal{Q}_\delta(\xi)$ be the set of all limits of convergent subsequences of  $\{p^l_\delta(\cdot;\xi)\}_{l\in\N}$. Then, we see that 
$$\mathcal{P}_\delta=\bigcup_{\xi}\Big\{   
 \mu_\delta(\xi)=\sum_{(x,t)\in\pr\tG} p_\delta(x,t;\xi) \mathfrak{d}_{x,t,\xi^{k(t)}_{m(x)}}\,\Big|\,
 p_\delta(\cdot;\xi)\in\mathcal{Q}_\delta(\xi)
 \Big\},$$ 
where the union is taken over all admissible controls.

We investigate an a priori  constraint of $\mathcal{P}_\delta$, which corresponds to the holonomic constraint in Mather's minimizing problem \cite{Mane}.  Fix an arbitrary admissible control $\xi$. For each function $g:\pr\tG\to\R$, consider a continuous function $f:\T^d\times\T\times\R^d\to\R$ whose restriction to $\pr\tG\times\R^d$ is  
\begin{eqnarray}\label{holoholo} 
f(x_{m},t_{k+1}, \zeta):=(D_tg)^{k+1}_{m}+(D_x g)^{k}_{m}\cdot \zeta,\quad (x_{m},t_{k+1}, \zeta)\in \pr\tG\times\R^d.
\end{eqnarray}
\begin{Prop}\label{holonomic-delta}
Each $\mu_\delta(\xi)\in\mathcal{P}_\delta$ satisfies the constraint 
$$\int_{\T^d\times\T\times\R^d}f\,\,d\mu_\delta(\xi)=0$$
for any continuous function $f:\T^d\times\T\times\R^d\to\R$ whose restriction to $\pr\tG\times\R^d$ is of the form \eqref{holoholo}. 
\end{Prop}
\begin{proof}
We cut off $f$ with respect to $\zeta$ without changing the values for $\zeta\in[-(d\lambda)^{-1},(d\lambda)^{-1}]^d$ so that $f$ belongs to $C_c(\T^d\times\T\times\R^d;\R)$ and  periodically extend it to $\R^d\times\R\times\R^d$. Then, we have 
\begin{eqnarray*}
&&\mathcal{F}_\delta^l(f;\xi)=\frac{1}{t_{l}}\sum_{-l<k\le0}    \Big( \sum_{\{m\,|\,(x_{m+\1},t_k)\in \tG_\delta \}} p(x_{m+\1},t_k; \xi) f(x_{m+\1},t_{k},\xi^k_{m+\1})\Big)\tau\\
&&\quad= \frac{1}{t_{l}}\sum_{-l<k\le0}    \Big( \sum_{\{m\,|\,(x_{m+\1},t_k)\in G^{0,-l}_{n+\1} \}} p^k_{m+\1} (\xi)  f(x_{m+\1},t_k,\xi^k_{m+\1})\Big)\tau\\
&&\quad=\frac{1}{t_{l}}\sum_{-l<k\le0}    \Big[ \sum_{\{m\,|\,(x_{m+\1},t_k)\in G^{0,-l}_{n+\1} \}} p^k_{m+\1} (\xi) 
\Big\{ g^k_{m+\1}-\frac{1}{2d}\sum_{i=1}^d \big(g^{k-1}_{m+\1+e_i} +g^{k-1}_{m+\1-e_i} \big) \\
&&\qquad + \sum_{i=1}^d\big(  g^{k-1}_{m+\1+e_i} -g^{k-1}_{m+\1-e_i}     \big)\xi^i{}^k_{m+\1} \frac{\lambda}{2}
\Big\}\Big]\\
&&\quad=\frac{1}{t_{l}}\sum_{-l<k\le0}    \Big\{ \sum_{\{m\,|\,(x_{m+\1},t_k)\in G^{0,-l}_{n+\1} \}} \Big(p^k_{m+\1} (\xi) g^k_{m+\1} -\sum_{\omega\in B} p^k_{m+\1} (\xi) \rho^k_{m+\1}(\omega)g^{k-1}_{m+\1+\omega} 
\Big)\Big\}.
\end{eqnarray*}
It follows from  \eqref{p-evolution} that 
$$\sum_{\{m\,|\,(x_{m+\1},t_k)\in G^{0,-l}_{n+\1} \}} \sum_{\omega\in B} p^k_{m+\1} (\xi) \rho^k_{m+\1}(\omega)g^{k-1}_{m+\1+\omega} 
=\sum_{\{m\,|\,(x_{m},t_{k-1})\in G^{0,-l}_{n+\1} \}} p^{k-1}_{m}(\xi)g^{k-1}_{m}. 
$$
Hence, we obtain 
\begin{eqnarray*}
&&\mathcal{F}_\delta^l(f;\xi)=\frac{1}{t_{l}}\sum_{-l<k\le0}    \Big( \sum_{\{m\,|\,(x_{m+\1},t_k)\in G^{0,-l}_{n+\1} \}} p^k_{m+\1} (\xi) g^k_{m+\1}
- \sum_{\{m\,|\,(x_{m},t_{k-1})\in G^{0,-l}_{n+\1} \}} p^{k-1}_{m} (\xi) g^{k-1}_{m}
\Big)\\
&&\quad =\frac{1}{t_{l}}\Big(  g^0_{n+\1}-      \sum_{\{m\,|\,(x_{m+\1},t_{-l})\in G^{0,-l}_{n+\1} \}} p^{-l}_{m+\1} (\xi) g^{-l}_{m+\1}    \Big)\\
&&\quad =\int_{\T^d\times\R^d\times\T}f \,\,d \mu^{l}_\delta(\xi).
\end{eqnarray*}
Since $\mu_\delta(\xi)$ is the weak limit of a  subsequence of $\{\mu^l_\delta\}_{l\in\N}$, the assertion is proved. 
\end{proof}
The next theorem is an analogue of Mather's minimizing problem. 
\begin{Thm}\label{Mather-delta}
For each $c\in P$, we have
\begin{eqnarray}\label{33minimizing}
\inf_{\mu_\delta\in \mathcal{P}_\delta} \int_{\T^d\times\T\times\R^d} L^{(c)}(x,t-\tau,\zeta)\,d\mu_\delta=-\bar{H}_\delta(c), 
\end{eqnarray}
where there exists at least one minimizing probability measure $\mu_\delta^\ast\in \mathcal{P}_\delta$ that attains the infimum. 
\end{Thm}
\begin{proof} 
Let $\bar{v}=\bar{v}(c)$ be a time-$1$-periodic solution of  \eqref{HJ-delta-cell} and consider the admissible control 
$$\xi^\ast:\pr \tG\ni(x_{m},t_{k+1})\mapsto H_p(x_{m},t_k,c+(D_x\bar{v})^k_{m})\in[-(d\lambda_1)^{-1},(d\lambda_1)^{-1}]^d$$
Then, for any $x_{n+\1}$, we have 
\begin{eqnarray*}
\int_{\T^d\times\T\times\R^d} L^{(c)}(x,t-\tau,\zeta)+\bar{H}_\delta(c) \,\,\,d\mu^{l}_\delta(\xi^\ast)&=&  \mathcal{F}_\delta^l(L^{(c)}(\cdot,\cdot-\tau,\cdot)+\bar{H}_\delta(c) ;\xi^\ast)\\
&=& \frac{1}{t_l}\Big(\bar{v}^0_{n+\1}-    
E_{\mu^{0,-l}_{n+\1}(\cdot; \xi^\ast)}\big[\bar{v}^{-l}_{m(\gamma^{-l})}\big]
\Big).
\end{eqnarray*} 
Since $\bar{v}$ is bounded, the weak limit $\mu^\ast_\delta(\xi^\ast)\in \mathcal{P}_\delta$ of a subsequence of $\{\mu^l_\delta(\xi^\ast)\}_{l\in\N}$ yields  
$$\int_{\T^d\times\T\times\R^d} L^{(c)}(x,t-\tau,\zeta)+\bar{H}_\delta(c) \,\,\,d\mu^\ast_\delta(\xi^\ast)=0.$$
\indent Let $\mu_\delta$ be an arbitrary element of $\mathcal{P}_\delta$, where $\mu_\delta$ is the weak limit of a subsequence of $\{\mu_\delta^{l}(\xi)\}_{l\in\N}$ with some admissible control $\xi$.  Due to variational structure of  \eqref{HJ-delta-cell}, we have  
\begin{eqnarray*}
\frac{1}{t_l}\bar{v}^0_{n+\1}&=&\mathcal{F}_\delta^l(L^{(c)}(\cdot,\cdot-\tau,\cdot)+\bar{H}_\delta(c);\xi^\ast )+ \frac{1}{t_l}E_{\mu^{0,-l}_{n+\1}(\cdot; \xi^\ast)}\big[\bar{v}^{-l}_{m(\gamma^{-l})}\big]\\
&\le& \mathcal{F}_\delta^l(L^{(c)}(\cdot,\cdot-\tau,\cdot)+\bar{H}_\delta(c) ;\xi)+ \frac{1}{t_l}E_{\mu^{0,-l}_{n+\1}(\cdot; \xi)}\big[\bar{v}^{-l}_{m(\gamma^{-l})}\big],
\end{eqnarray*} 
which leads to 
$$\int_{\T^d\times\T\times\R^d} L^{(c)}(x,t-\tau,\zeta)+\bar{H}_\delta(c) \,\,\,d\mu^\ast_\delta(\xi^\ast)\le \int_{\T^d\times\T\times\R^d} L^{(c)}(x,t-\tau,\zeta)+\bar{H}_\delta(c) \,\,\,d\mu_\delta(\xi).$$
\end{proof}
\subsection{Mather set and Aubry set}

We define an analogue of the Mather set $\mathcal{M}_\delta(c)$ for each $c$ as 
$$\mathcal{M}_\delta(c):=\bigcup_{\mu^\ast_\delta}\mbox{supp}(\mu^\ast_\delta)\subset \pr \tG_\delta\times[-(d\lambda)^{-1},(d\lambda)^{-1}]^d,$$
where the union is taken over all minimizing probability measures $\mu^\ast_\delta$ of \eqref{33minimizing}. We state properties of $\mathcal{M}_\delta(c)$. 
\begin{Thm}\label{Aubry-delta}
For each $c$, the support of any minimizing probability measure for \eqref{33minimizing} is contained in the set 
\begin{eqnarray*}
\mathcal{A}_\delta(c):=\bigcap_{\bar{v}(c)} \Big\{\big(x_{m},t_{k+1},H_p(x_{m},t_k,c+(D_x\bar{v})^{k}_{m})\big)\,|\, (x_{m},t_{k+1})\in \pr\tG_\delta    \Big\},
\end{eqnarray*}
where the intersection is taken over all time-$1$-periodic solutions of (\ref{HJ-delta-cell}); namely we have 
$$\mathcal{M}_\delta(c)\subset\mathcal{A}_\delta(c)\subset \pr \tG_\delta\times[-(d\lambda)^{-1},(d\lambda)^{-1}]^d.$$
\end{Thm}
\begin{proof} 
Suppose that there exists a minimizing probability measure $\mu_\delta^\ast$ whose support is not contained in  the set $\mathcal{A}_\delta(c)$. Then, we have a point $(x_{n},t_{l+1},\zeta^\ast)\in$\,supp$(\mu_\delta^\ast)$ such that $\zeta^\ast\neq H_p(x_{n},t_l,c+(D_x\bar{v})^{l}_{n})$ for some time-$1$-periodic solution $\bar{v}=\bar{v}(c)$. It holds that  
\begin{eqnarray*}
-\bar{H}_\delta(c)&=&-(D_t\bar{v})^{k+1}_{m}-H(x_{m},t_k,c+(D_x\bar{v})^{k}_{m})\\
&\le&-(D_t\bar{v})^{k+1}_{m}-(D_x\bar{v})^{k}_{m}\cdot\zeta+L^{(c)}(x_{m},t_k,\zeta)\mbox{\quad for any $\zeta\in\R^d$},
\end{eqnarray*}
which is a strict inequality for any $\zeta\neq H_p(x_{m},t_k,c+(D_x\bar{v})^{k}_{m})$ due to the Legendre transform. 
Consider a continuous function $f:\T^d\times\T\times\R^d\to\R$ whose restriction to $\pr\tG_\delta$ is  given as  $f(x_m,t_{k+1},\zeta)=-\{(D_t\bar{v})^{k+1}_{m}+(D_x\bar{v})^{k}_{m}\}\cdot\zeta+L^{(c)}(x_{m},t_{k+1}-\tau,\zeta)$. 
Then, Proposition \ref{holonomic-delta} implies that 
\begin{eqnarray*}
\int_{\T^d\times\T\times\R^d}f(x,t,\zeta)d\mu_\delta^\ast=\int_{\T^d\times\T\times\R^d}L^{(c)}(x,t-\tau,\zeta)\,\,d\mu_\delta^\ast>-\bar{H}_\delta(c).
\end{eqnarray*}
This is a contradiction.
\end{proof}
\noindent One can see the set $\mathcal{A}_\delta(c)$ as an analogue of the Aubry set. 
\begin{Cor}
Let $(x_{m+\1},t_k,\zeta)$ be an arbitrary point of $\mathcal{M}_\delta(c)$. Any time-$1$-periodic solution $\bar{v}=\bar{v}(c)$ of \eqref{HJ-delta-cell} satisfies 
$$\zeta=H_p(x_{m+\1},t_{k-1},c+(D_x\bar{v})^{k-1}_{m+\1})\Leftrightarrow c+(D_x\bar{v})^{k-1}_{m+\1}=L_\zeta(x_{m+\1},t_{k-1},\zeta).$$ 
There exists a bijective map: $\pi \mathcal{M}_\delta(c):=\{(x,t)\,|\, (x,t,\zeta)\in \mathcal{M}_\delta(c)\}\subset\pr\tG_\delta\to\mathcal{M}_\delta(c)$. If $\pi \mathcal{M}_\delta(c)=\pr\tG_\delta$, then a periodic solution $\bar{v}=\bar{v}(c)$ of \eqref{HJ-delta-cell} is unique up to constants. 
\end{Cor}
\noindent Although $p^l_\delta(\cdot;\xi)$ with $|\xi|_\infty<(d\lambda)^{-1}$ is supported by whole $\pr\tG_\delta$ for any large $l$ due to the diffusive effect, it is not clear if $p_\delta(\cdot;\xi)$ is also supported by whole $\pr\tG_\delta$. Since minimizing random walks  are inhomogeneous, direct calculation with Stirling's approximation does not seem to be effective. Hence,  to check $\pi \mathcal{M}_\delta(c)=\pr\tG_\delta$ or not  is an open problem.  
\begin{Thm}
Let $\bar{v}=\bar{v}(c)$ and $\hat{v}=\hat{v}(c)$ be time-$1$-periodic solutions of \eqref{HJ-delta-cell}.  If $\bar{v}=\hat{v}$ on $\pi \mathcal{M}_\delta(c)$, then $\bar{v}=\hat{v}$ on whole $\pr \tG_\delta$. 
\end{Thm}
\begin{proof}
It is enough to prove $\bar{v}^0_{n+\1}=\hat{v}^0_{n+\1}$ on whole  $\pr G_{\rm odd}$. Let $x_{n+\1}$ be any point of $\pr G_{\rm odd}$. Set $\xi^\ast:\pr \tG_\delta\ni(x_{m},t_{k+1})\mapsto H_p(x_{m},t_k,c+(D_x\bar{v})^k_{m})\in[-(d\lambda_1)^{-1},(d\lambda_1)^{-1}]^d$. Then, for any $k<0$, we have  
\begin{eqnarray*}
\bar{v}^0_{n+\1}&=&t_l\mathcal{F}_\delta^k(L^{(c)}(\cdot,\cdot-\tau,\cdot)+\bar{H}_\delta(c);\xi^\ast)+E_{\mu^{0,k}_{n+\1}(\cdot;\xi^\ast)}[\bar{v}^{k}_{m(\gamma^{k})}]\\
&=&t_l\mathcal{F}_\delta^k(L^{(c)}(\cdot,\cdot-\tau,\cdot)+\bar{H}_\delta(c);\xi^\ast)+\sum_{\{m\,|(x_{m+\1},t_k)\in\tG_\delta\}}p(x_{m+1},t_k;\xi^\ast) \bar{v}^{k}_{m(\gamma^{k})},\\
\hat{v}^0_{n+\1}&\le&t_l\mathcal{F}_\delta^k(L^{(c)}(\cdot,\cdot-\tau,\cdot)+\bar{H}_\delta(c);\xi^\ast)+E_{\mu^{0,-k}_{n+\1}(\cdot;\xi^\ast)}[\hat{v}^{k}_{m(\gamma^{k})}]\\
&=&t_l\mathcal{F}_\delta^k(L^{(c)}(\cdot,\cdot-\tau,\cdot)+\bar{H}_\delta(c);\xi^\ast)+\sum_{\{m\,|(x_{m+\1},t_k)\in\tG_\delta\}}p(x_{m+1},t_k;\xi^\ast) \hat{v}^{k}_{m(\gamma^{k})}.
\end{eqnarray*}
Hence, we see that 
\begin{eqnarray*}
&&\hat{v}^0_{n+\1}-\bar{v}^0_{n+\1}=\sum_{\{m\,|(x_{m+\1},t_0)\in\tG_\delta\}}p(x_{m+1},t_0;\xi^\ast) (\hat{v}^{0}_{m(\gamma^{0})}-\bar{v}^{0}_{m(\gamma^{0})}),\\
&&\hat{v}^0_{n+\1}-\bar{v}^0_{n+\1}\le \sum_{\{m\,|(x_{m+\1},t_k)\in\tG_\delta\}}p(x_{m+1},t_k;\xi^\ast) (\hat{v}^{k}_{m(\gamma^{k})}-\bar{v}^{k}_{m(\gamma^{k})})  \mbox{\quad for $-l<k\le -1$}.
\end{eqnarray*}
Summation of all the (in)equalities yields 
\begin{eqnarray*}
(\hat{v}^0_{n+\1}-\bar{v}^0_{n+\1})l
\le \sum_{-l<k\le0}\sum_{\{m\,|(x_{m+\1},t_k)\in\tG_\delta\}}p(x_{m+1},t_k;\xi^\ast) (\hat{v}^{k}_{m(\gamma^{k})}-\bar{v}^{k}_{m(\gamma^{k})}),  
\end{eqnarray*}
which leads to
\begin{eqnarray*}
\hat{v}^0_{n+\1}-\bar{v}^0_{n+\1}
\le \sum_{\{(z,t)\in\pr\tG_\delta\}}p^l_\delta(x,t;\xi^\ast) (\hat{v}^{k(t)}_{m(x)}-\bar{v}^{k(t)}_{m(x)}).  
\end{eqnarray*}
Since $p^l_\delta(\cdot;\xi^\ast)$ converges to  $p^\ast_\delta(\cdot;\xi^\ast)$ as $l\to\infty$ (up to a subsequence), which yields a minimizing measure $\mu^\ast_\delta(\xi^\ast)\in \mathcal{P}_\delta$, we obtain $\hat{v}^0_{n+\1}-\bar{v}^0_{n+\1}\le0$. The same argument with 
$\hat{\xi}^\ast:\pr \tG_\delta\ni(x_{m},t_{k+1})\mapsto H_p(x_{m},t_k,c+(D_x\hat{v})^k_{m})\in[-(d\lambda_1)^{-1},(d\lambda_1)^{-1}]^d$ proves $\hat{v}^0_{n+\1}-\bar{v}^0_{n+\1} \ge0$.
\end{proof}
As for \eqref{HJ-delta+} and $\mathcal{L}^l_+$, we have  similar results. 
\subsection{Hyperbolic scaling limit}

We observe the hyperbolic scaling limit of what we obtained in Subsection 3.1-3.4, where we mean by ``$\delta=(h,\tau)\to0$'' that $\delta\to0$ with the condition $0<\lambda_0\le \lambda=\tau/h\le\lambda_1$  with the constant $\lambda_1$ found in Section 2. 

First, we state convergence of the effective Hamiltonian $\bar{H}_\delta$ to that of the exact problem 
\begin{eqnarray}\label{4cell}
v_t+H(x,t,c+v_x)=\bar{H}(c)\quad{\rm in}\,\,\,\T^d\times\T.
\end{eqnarray}
Hereinafter $b_1,b_2,\ldots$ are positive constants independent of $\delta$ and $c$.
\begin{Thm}\label{effective222}
There exists a constant $b_1>0$ for which we have 
$$\sup_{c\in P}|\bar{H}_\delta(c)-\bar{H}(c)|\le b_1\sqrt{h}\quad \mbox{for $\delta\to0$}.$$
\end{Thm}
\begin{proof}
Let $\bar{v}^k_{m+\1}$ be a time-$1$-periodic solution of \eqref{HJ-delta-cell} and let $\bar{w}$ be a viscosity solution of \eqref{4cell}.  
Let  $u_\delta(\cdot):\T^d\to\R$ be the Lipschitz interpolation of $\bar{v}^0_{m+\1}=\bar{v}^{2K}_{m+\1}$, where one can find a Lipschitz constant of $u_\delta$  independently from $\delta$. Let $x\in\T^d$ be the point at which $\max_{x\in\T^d} (u_\delta(x)-\bar{w}(x,1))$ is attained. 
Let $x_{n+\1}\in G_{\rm odd}$ be such that $|x_{n+\1}-x|_\infty\le h$. We have $|u_\delta(x)-u_\delta(x_{n+\1})|\le b_2 h$.
For a minimizing curve $\gamma^\ast:[0,1]\to\R^d$ for $\bar{w}(x_{n+\1},1)$, define the control $\xi^{k+1}_{m}:=\gamma^\ast{}'(t_{k+1})$ (a space-homogeneous control).  Then, the random walk generated by $\xi$ satisfies 
$$|\eta^k(\gamma)-\gamma^\ast{}(t_k)|_\infty\le b_3 h\mbox{ for all $\gamma\in\Omega^{2K,0}_{n+\1}$}.$$  
Hence, we have with Lemma \ref{limit-theorem},
\begin{eqnarray*}
&&u_\delta(x_{n+\1})=\bar{v}^{2K}_{n+\1}\\
&&\le E_{\mu^{2K,0}_{n+\1}(\cdot;\xi)}\Big[\sum_{0<k\le 2K}L^{(c)}(\gamma^k,t_{k-1},\xi^k_{m(\gamma^k)})\tau +u_\delta(\gamma^0)\Big]+\bar{H}_\delta(c )\\
&&\le E_{\mu^{2K,0}_{n+\1}(\cdot;\xi)}\Big[\sum_{0<k\le 2K}L^{(c)}(\eta^k(\gamma),t_{k-1},\xi^k_{m(\gamma^k)})\tau +u_\delta(\eta^0(\gamma))\Big]+\bar{H}_\delta(c )+b_4\sqrt{h}\\
&&\le E_{\mu^{2K,0}_{n+\1}(\cdot;\xi)}\Big[\sum_{0<k\le 2K}L^{(c)}(\gamma^\ast(t_k),t_{k-1},\gamma^\ast{}'(t_k))\dt +u_\delta(\gamma^\ast(0))\Big]+\bar{H}_\delta(c )+b_5\sqrt{h}\\
&&\le \int^1_0L^{(c)}(\gamma^\ast(s),s,\gamma^\ast{}'(s))ds+u_\delta(\gamma^\ast(0))+\bar{H}_\delta(c )+b_6\sqrt{h}.
\end{eqnarray*}
Since we have 
$$\bar{w}(x_{n+\1},1)=\int^1_0L^{(c)}(\gamma^\ast(s),s,\gamma^\ast{}'(s))ds+\bar{w}(\gamma^\ast(0),0)+\bar{H}(c),$$
with $ \bar{w}(\gamma^\ast(0),0)=\bar{w}(\gamma^\ast(0),1)$, we obtain 
\begin{eqnarray*}
u_\delta(x)-\bar{w}(x,1)&\le&  
u_\delta(x_{n+1})-\bar{w}(x_{n+1},1)+b_{7}h\\
&\le&u_\delta(\gamma^\ast(0))-\bar{w}(\gamma^\ast(0),1)+\bar{H}_\delta(c)-\bar{H}(c)+b_6\sqrt{h}+b_{7}h.
\end{eqnarray*}
By the choice of $x$, we see that 
\begin{eqnarray*}
\bar{H}_\delta(c)-\bar{H}(c)&\ge& \{u_\delta(x)-\bar{w}(x,1)\}-
\{u_\delta(\gamma^\ast(0))-\bar{w}(\gamma^\ast(0),1)\}-b_8\sqrt{h}\\
&\ge&-b_8\sqrt{h}.
\end{eqnarray*}
\indent We prove the converse inequality. 
We switch the above $x$ to the one attaining $\min_{x\in\T^d} (u_\delta(x)-\bar{w}(x,1))$. 
Let  $\xi^\ast$ be the minimizing control for $u_\delta(x_{n+\1})=\bar{v}^{2K}_{n+\1}$. Then, we have
\begin{eqnarray*}
u_\delta(x_{n+\1})&=&E_{\mu^{2K,0}_{n+\1}(\cdot;\xi^\ast)}\Big[\sum_{0<k\le 2K}L^{(c)}(\gamma^k,t_{k-1},\xi^\ast{}^k_{m(\gamma^k)})\dt +u_\delta(\gamma^0)\Big]+\bar{H}_\delta(c)\\
&\ge& E_{\mu^{2K,0}_{n+\1}(\cdot;\xi^\ast)}\Big[\sum_{0<k\le 2K}L^{(c)}(\eta^k(\gamma),t_{k-1},\xi^\ast{}^k_{m(\gamma^k)})\dt +u_\delta(\eta^0(\gamma))\Big]+\bar{H}_\delta(c )-b_9\sqrt{\dx}.
\end{eqnarray*}
Let $\eta_\delta(\gamma)$ be the linear interpolation of $\eta^k(\gamma)$, where $\eta_\delta(\gamma)'(t)=\xi^\ast{}^k_{m(\gamma^k)}$ for $t\in(t_{k-1},t_k)$. 
Then, for each $\gamma$, we have
\begin{eqnarray*}
\bar{w}(x_{n+\1},1)&\le&\int_0^1L^{(c)}(\eta_\delta(\gamma)(s),s,\eta_\delta(\gamma)'(s))ds+\bar{w}(\eta_\delta(\gamma)(0),0)+\bar{H}(c )\\
&\le&\sum_{0<k\le 2K}L^{(c)}(\eta^k(\gamma),t_{k-1},\xi^\ast{}^k_{m(\gamma^k)})\tau +\bar{w}(\eta^0(\gamma),0)+\bar{H}(c )+b_9h.
\end{eqnarray*}
Therefore, we see that 
\begin{eqnarray*}
u_\delta(x)-\bar{w}(x,1)&\ge& 
u_\delta(x_{n+1})-\bar{w}(x_{n+1},1)-b_{10}h\\
&\ge& E_{\mu(\cdot;\xi^\ast)}\Big[u_\delta(\eta^0(\gamma))-\bar{w}(\eta^0(\gamma),1)\Big]+\bar{H}_\delta(c )-\bar{H}(c )-b_{11}\sqrt{h}.
\end{eqnarray*}
Due to the choice of $x$, we conclude $\bar{H}_\delta(c )-\bar{H}(c )\le b_{11}\sqrt{h}$.
\end{proof}
As for  $\tilde{\bar{H}}_\delta(c)$, we have the same convergence result.
\begin{Thm}\label{43434343}
Let $\{\bar{v}_\delta(c)\}_\delta$ be a uniformly bounded sequence of time-$1$-periodic solution of \eqref{HJ-delta-cell} with $\delta\to0$. Then, there exists a subsequence of $\{\bar{v}_\delta(c)\}_\delta$, still denoted by the same symbol, such that  $\{\bar{v}_\delta(c)\}_\delta$ tends to a viscosity solution $\bar{v}$ of \eqref{4cell} as $\delta\to0$ in the sense that 
\begin{eqnarray}\label{maxmax}
\max_{\{(m,k+1)|\,(x_m,t_{k+1})\in\pr \tG_\delta\}}|\bar{v}_\delta(c)^{k+1}_m-\bar{v}(x_m,t_{k+1})|\to0\mbox{\quad as $\delta\to0$}.
\end{eqnarray}
\end{Thm}
\begin{proof}
Let $w_\delta:\T^d\to\R$ be the Lipschitz interpolation of $\bar{v}_\delta(c)^0$. Then, $\{w_\delta\}_\delta$ is uniformly bounded and equi-Lipschitz to have convergent subsequence with the limit $w$. Let $\bar{v}$ be the viscosity solution of 
$$v_t+H(x,t,c+v_x)=\bar{H}(c)\quad{\rm in}\,\,\,\T^d\times(0,1],\quad v(\cdot,0)=w(\cdot)\quad{\rm on}\,\,\,\T^d.$$
Since $\bar{H}_\delta(c)\to\bar{H}(c)$ as $\delta\to0$, we see that \eqref{maxmax} follows from the same reasoning as the proof of Theorem \ref{convergence-Lax-Oleinik}, where the time periodicity of $\bar{v}_\delta$ implies that  
 $\bar{v}$ is time-$1$-periodic.
\end{proof}
As for  a sequence of time-$1$-periodic solutions of \eqref{HJ-delta2+}, it tends to a semiconvex a.e. solution of \eqref{4cell} as $\delta\to0$ (up to a subsequence).

The same reasonings as the proofs of Theorem \ref{convergence-derivative} and \ref{in-probability1} yield the following theorems:  
\begin{Thm}
Consider the subsequence $\{\bar{v}_\delta(c)\}_\delta$ and $\bar{v}$ in  Theorem \ref{43434343}. 
Let $(x,t)\in \T^d\times(0,1]$ be such that $\bar{v}_{x^i}(x,t)$ exists. Note that a.e. points of $\T^d\times[0,1]$ have such a property. Let $(x_n,t_{l})\in{\mathcal{G}}_\delta$ be such that $(x_n,t_{l})\to(x,t)$ as $\delta\to0$.   Then, we have      
$$|(D_{x^i}\bar{v}_\delta(c))^l_n-\bar{v}_{x^i}(x,t)|\to0\mbox{\quad as $\delta\to0$}.$$
\end{Thm}
\begin{Thm}
Consider the subsequence $\{\bar{v}_\delta(c)\}_\delta$ and $\bar{v}$ in  Theorem \ref{43434343}. 
Let $(x,t)\in \R^d\times(0,1]$, $\tilde{t}\ge0$ be  arbitrary. Let $\Gamma^\ast(x,t)$ be the set of all minimizing curves $\gamma^\ast:[-\tilde{t},t]\to\R^d$ for $\bar{v}(x,t)$. 
Let $(x_n,t_{l+1})\in\tilde{\mathcal{G}}_\delta$, $\tilde{l}$ be such that $(x_n,t_{l+1})\to(x,t)$, $t_{-\tilde{l}}\to-\tilde{t}-0$ as $\delta\to0$. For each $\delta$, let $\gamma\in \Omega^{l+1,-\tilde{l}}_n$ be the random walk generated by the minimizing control $\xi^\ast$ for $v^{l+1}_n$.  
\begin{enumerate}
\item Fix  an arbitrary $\ep_1>0$ and define the set
\begin{eqnarray*}
\check{\Omega}^{\ep_1}_\delta&:=&\{ \gamma\in  \Omega^{l+1,-\tilde{l}}_n\,|\,\mbox{there exists $\gamma^\ast=\gamma^\ast(\gamma)\in\Gamma^\ast(x,t)$} \\
&&\qquad\qquad\qquad\qquad\qquad\mbox{ such that $\norm \eta_\delta(\gamma)'-\gamma^\ast{}'\norm_{L^2([-\tilde{t},t])}\le \ep_1$}  \}.
\end{eqnarray*}
Then, we have Prob$(\check{\Omega}_\delta^{\ep_1})\to1$ as $\delta\to0$.
\item Fix an arbitrary $\ep_2>0$ and define the set
\begin{eqnarray*}
\tilde{\Omega}^{\ep_2}_\delta&:=&\{ \gamma\in  \Omega^{l+1,-\tilde{l}}_n\,|\,\mbox{there exists $\gamma^\ast=\gamma^\ast(\gamma)\in\Gamma^\ast$}\\
&&\qquad\qquad\qquad\qquad\qquad\mbox{ such that $\norm\eta_\delta(\gamma)-\gamma^\ast\norm_{C^0([-\tilde{t},t])}\le \ep_2$}\}.
\end{eqnarray*}
Then, we have $\mbox{Prob}(\tilde{\Omega}_\delta^{\ep_2})\to1$ as $\delta\to0$. 
\item Fix  an arbitrary $\ep_3>0$ and  define the set
\begin{eqnarray*}
\Omega^{\ep_3}_\delta&:=&\{ \gamma\in  \Omega^{l+1,-\tilde{l}}_n\,|\,\mbox{there exists $\gamma^\ast=\gamma^\ast(\gamma)\in\Gamma^\ast(x,t)$} \\
&&\qquad\qquad\qquad\qquad\qquad\,\,\,\,\,\,\,\,\mbox{ such that $\norm\gamma_\delta-\gamma^\ast\norm_{C^0([-\tilde{t},t])}\le \ep_3$}\}. 
\end{eqnarray*}
Then, we have Prob$(\Omega_\delta^{\ep_3})\to1$ as $\delta\to0$.
\end{enumerate}
\end{Thm}
\begin{Thm}
Let $\{\mu^\ast_\delta\}_\delta$ with $\delta\to0$ be a sequence of minimizing measures for \eqref{33minimizing}. Then, there exists a subsequence of  $\{\mu^\ast_\delta\}_\delta$, still denoted by the same symbol and an exact  Mather measure $\mu^\ast$ for $L^{(c)}$ such that $\mu^\ast_\delta$ converges weakly to $\mu^\ast$. 
\end{Thm}
\begin{proof}
Since $\{\mu^\ast_\delta\}_\delta$ is supported by a compact set independently from $\delta$, we find a weakly  convergent subsequence with the limit $\mu^\ast$.  Proposition \ref{holonomic-delta} implies the holonomic constraint of $\mu^\ast$. Theorem \ref{Mather-delta} and Theorem \ref{effective222} imply that $\mu^\ast$ satisfies 
$$\int_{\T^d\times\T\times\R^d}L^{(c)}(x,t,\zeta)d\mu^\ast=-\bar{H}(c),$$
which shows that $\mu^\ast$ is a Mather measure for $L^{(c)}$ \cite{Itu}, \cite{Mane}. 
\end{proof} 
\setcounter{section}{3}
\setcounter{equation}{0}
\section{Weak KAM theory - autonomous case}
We investigate the case with a time-independent Lagrangian $L(x,\zeta)$/Hamiltonian $H(x,p)$. The presentation of Section 3 becomes simpler in the autonomous case. We state several things which differ from Section 3. 

\subsection{Weak KAM solution on grid}
Since the configuration grid switches between $G_{\rm even}$ and $G_{\rm odd}$ depending on the time index, the discrete Hamilton-Jacobi equation always needs at least two unit-time-evolutions, i.e., with $\phi_\delta:=\varphi_\delta^2$, we have 
$$\mbox{$\varphi_{\delta}:=\varphi_\delta^{2K}=\phi_\delta\circ\phi_\delta\circ\cdots\circ\phi_\delta=(\phi_\delta)^K$ ($K$-iteration of $\phi_\delta$). }$$
We say that {\it $v$ is a stationary solution of  \eqref{HJ-delta-cell}, if it satisfies $v(\cdot,t_{2k})=v(\cdot,0)$, $v(\cdot,t_{2k+1})=v(\cdot,t_1)$ for all $k$} (one cannot have $v^0=v(\cdot,t_k)=v(\cdot,t_{k+1})$). 
Note that $v^0\in X_\delta$ yields a stationary solution, if and only if $v^0$ admits $\phi_\delta v^0+2\tau\bar{H}_\delta(c)= v^0$. 
The reasoning in Subsection 3.1 yields a pair of $\bar{H}_\delta(c)\in\R$ and $\bar{v}^0\in X_\delta$ such that 
$\varphi_\delta(\bar{v}^0;c)+\bar{H}_\delta(c)=\bar{v}^0.$
\begin{Prop}\label{kkkkaaaaa}
If $\bar{v}^0\in X_\delta$ satisfies $\varphi_\delta(\bar{v}^0;c)+\bar{H}_\delta(c)=\bar{v}^0$, it satisfies also $ \phi_\delta\bar{v}^0+2\tau\bar{H}_\delta(c)= \bar{v}^0$, i.e., $\bar{v}^k:= \varphi^k_\delta(\bar{v}^0;c)+t_k\bar{H}_\delta(c)$ is a stationary solution of \eqref{HJ-delta-cell}. 
\end{Prop}
\begin{proof}
For $\bar{v}^k:= \varphi^k_\delta(\bar{v}^0;c)+t_k\bar{H}_\delta(c)$ ($\bar{v}^k$ solves \eqref{HJ-delta-cell}), set  $Z^{k}_{m+\1}:=\bar{v}^{k+2}_{m+\1}-\bar{v}^{k}_{m+\1}$ and $S^k:=\max_{m}Z^{k}_{m+\1}$. It follows from \eqref{HJ-delta-cell}, Taylor's formula and (H2) that 
\begin{eqnarray*}
Z^{k+1}_m&=&\sum_{i=1}^d
\Big\{ \Big(\frac{1}{2d}-\frac{\lambda}{2}H_{p_i}(x_m,(D_x\bar{v})^k_m)\Big)Z^k_{m+e_i}+\Big(\frac{1}{2d}+\frac{\lambda}{2}H_{p_i}(x_m,(D_x\bar{v})^k_m)\Big)Z^k_{m-e_i}\Big\}\\
&&-\frac{\tau}{2}\kappa^{k+1}_m| (D_x\bar{v})^{k+2}_m-(D_x\bar{v})^k_m |^2, 
\end{eqnarray*}
where $\kappa^{k+1}_m>\kappa>0$ for all $k,m$. The CFL-type condition stated in Theorem \ref{main1} implies that $(\frac{1}{2d}\pm\frac{\lambda}{2}H_{p_i}(x_m,(D_x\bar{v})^k_m))>0$. Hence, we have $S^{k+1}\le S^k$ for all $k$.   
On the other hand,  since $\bar{v}^{2K}=\varphi_\delta(\bar{v}^0;c)+\bar{H}_\delta(c)=\bar{v}^0$, we have  $Z^{k+2K}_{m+\1}=Z^{k}_{m+\1}$ and $S^{k+2K}=S^k$ for all $k\ge0$. Therefore, it holds that $S^0=S^{2K}\le S^k\le S^0$ for all $0\le k\le 2K$, i.e., $S^k=S^0$ for all $k\ge0$. 

 Let $m$ be such that $S^{k+1}=Z^{k+1}_m$($=S^0$). Then, the recurrence equation of $Z^k_{m+\1}$ implies that we must have $Z^k_{m+\omega}=S^0$ for all $\omega\in B$, because otherwise $S^{k+1}<S^0$.  Let $\Lambda^k:=\{ x_{m+\1}\in \pr h\Z^d \,|\,(x_{m+1},t_k)\in \pr\tG,\,\,  Z^k_{m+\1}<S^0   \}$. 
If $x_{m+\1}\in\Lambda^k$, we must have $x_{m+\1+\omega}\in \Lambda^{k+1}$ for all $\omega\in B$. 
Hence, if $\Lambda^0\neq\emptyset$, it holds that $\sharp \Lambda^k<\sharp \Lambda^{k+1}$ for all $k\ge0$ until we have $\Lambda^{k+1}=\pr G_{\rm even}$ or $\pr G_{\rm odd}$, which means that $Z^k_{m+\1}<S^0$ for all $m$ at some $k$.   Therefore, we must have $\Lambda^0=\emptyset$. Furthermore,  the equality $S^k=S^0$ for all $k\ge0$ requires $(D_x\bar{v})^{k+2}_m=(D_x\bar{v})^{k}_m$ for all $m,k$. 
This implies that  there exists a constant $\beta\in\R$ such that $\bar{v}^0+\beta=\bar{v}^2$.  
Since $\bar{v}^2+\beta=\phi_\delta(\bar{v}^0+\beta)+2\tau \bar{H}_\delta(c)=\phi_\delta\bar{v}^2+2\tau \bar{H}_\delta(c)=\bar{v}^4$,  $\bar{v}^{4}+\beta=\bar{v}^{6}, \ldots,\bar{v}^{2K-2}+\beta=\bar{v}^{2K}$, we have $\bar{v}^0+K\beta=\bar{v}^{2K}=\bar{v}^0$. Thus, we conclude that $\beta=0$ and $\bar{v}^0=\bar{v}^2=\phi_\delta\bar{v}^0+2\tau\bar{H}_0(c)$.
\end{proof}
 
\subsection{Long-time behavior of Lax-Oleinik type solution map}
In autonomous weak KMA theory,   Fathi  \cite{Fathi98-2} elegantly makes clear the long time behavior of the Lax-Oleinik semigroup: $\mathcal{T}^t_-u+t\bar{H}(c)$ converges to a weak KAM solution (the limit depends on $u$) as $t\to\infty$ for all $u\in C^0(\T^d;\R)$. In other words, the viscosity solution $v$ of 
$$v_t+H(x,c+v_x)=\bar{H}(c),\,\,\,v(x,0)=u(x),\,\,\,\,x\in\T^d,\,\,\,t>0$$
converges to a stationary solution as $t\to\infty$. 
We remark that this is not true in a time-dependent case in general  \cite{Fathi-Mather}. We prove a similar result for $\varphi_\delta^k(\cdot;c)+t_k\bar{H}_\delta(c)$ as $k\to\infty$.
\begin{Thm}
For each $v^0\in X_\delta$, the solution $v^k=\varphi_\delta^k(v^0;c)+t_k\bar{H}_\delta(c)$ of \eqref{HJ-delta-cell} converges to some stationary solution of \eqref{HJ-delta-cell}   as $k\to\infty$, i.e., there exist $\bar{v}^0$ and $\bar{v}^1$ depending on $v^0$ such that $\bar{v}^0=\phi_\delta\bar{v}^0+2\tau\bar{H}_\delta(c)$, $\bar{v}^1=\varphi_\delta^1(\bar{v}^0;c)+\tau\bar{H}_\delta(c)$ for which we have 
$$\max_{x}|v^{2k}-\bar{v}^0|\to0,\,\,\,\max_{x}|v^{2k+1}-\bar{v}^1|\to0\mbox{\quad as $k\to\infty$}.$$
\end{Thm}
\begin{proof}
Fix $v^0\in X_\delta$ arbitrarily. 
By Proposition \ref{kkkkaaaaa}, we know that there exists $\bar{w}^0\in X_\delta$ such that $\phi_\delta\bar{w}^0+2\tau\bar{H}_\delta(c)=\bar{w}^0$. If $\bar{w}^0\le v^0$ does not hold, we add a constant to $\bar{w}^0$ and rewrite it as $\bar{w}^0$, to have $\bar{w}^0\le v^0$ and $\phi_\delta\bar{w}^0+2\tau\bar{H}_\delta(c)=\bar{w}^0$. We find a constant $A\in\R$ such that 
$$ \bar{w}^0\le v^0 \le \bar{w}^0+A.$$
For simpler presentation, set $\bar{\varphi}^k_\delta\,\,\cdot:=\varphi_\delta^k(\cdot;c)+t_k\bar{H}_\delta(c)$, $\bar{\phi}_\delta\,\,\cdot:=\phi_\delta \cdot +2\tau\bar{H}_\delta(c)$. 
Then, $\bar{w}^0$ satisfies $\bar{\varphi}^{2K}_\delta\bar{w}^0=\bar{w}^0$ and $\bar{\phi}_\delta\bar{w}^0=\bar{w}^0$. Since $\bar{\varphi}^k_\delta$ and $\bar{\phi}_\delta$  preserve the order,   we see that 
\begin{eqnarray}\label{huki}
\bar{w}^0\le v^{2l}=\bar{\varphi}_\delta^{2l}v^0=(\bar{\phi}_\delta)^lv^0 \le \bar{w}^0+A \mbox{\quad for all $l\in\N$}.
\end{eqnarray}
\indent We look at the increment within the small time $2\tau$. Set $Z^{k}_{m+\1}:=v^{k+2}_{m+\1}-v^{k}_{m+\1}$ and $S^k:=\max_{m}Z^{k}_{m+\1}$. As we already saw in the proof of Proposition \ref{kkkkaaaaa}, it holds that  
\begin{eqnarray*}
Z^{k+1}_m&=&\sum_{i=1}^d
\Big\{ \Big(\frac{1}{2d}-\frac{\lambda}{2}H_{p_i}(x_m,(D_xv)^k_m)\Big)Z^k_{m+e_i}+\Big(\frac{1}{2d}+\frac{\lambda}{2}H_{p_i}(x_m,(D_xv)^k_m)\Big)Z^k_{m-e_i}\Big\}\\
&&-\frac{\tau}{2}\kappa^{k+1}_m| (D_xv)^{k+2}_m-(D_xv)^k_m |^2, 
\end{eqnarray*}
where $\kappa^{k+1}_m>\kappa>0$ for all $k,m$ and  $S^{k+1}\le S^k$ for all $k$.   Since $|S^k|$ is uniformly bounded, we have $S^\ast\in\R$ such that $S^k\to S^\ast$ as $k\to\infty$. 

Since $\{v^{2K\cdot l}\}_{l\in\N}$ is bounded (in $\R^{\sharp G_{\rm odd} }$) due to \eqref{huki}, we have a convergent subsequence $\{v^{2K\cdot a_l}\}_{l\in N}$ such that $v^{2K\cdot a_l}\to \bar{v}^0\in X_\delta$ as $l\to\infty$. Note that 
\begin{eqnarray}\label{35355353}
  \bar{w}^0\le \bar{v}^0 \le \bar{w}^0+A.
\end{eqnarray}
Since $v^{2K\cdot a_l+k}=\bar{\varphi}^k_\delta v^{2K\cdot a_l}$ and 
$\bar{\varphi}^k_\delta$ is continuous, we have 
$v^{2K\cdot a_l+k}=\bar{\varphi}^k_\delta v^{2K\cdot a_l}\to \bar{v}^k:=\bar{\varphi}^k_\delta \bar{v}^0$ as $l\to\infty$ for all $k\in\N$. 
Furthermore, for $\tilde{Z}^{k}_{m+\1}:=\bar{v}^{k+2}_{m+\1}-\bar{v}^{k}_{m+\1}$ and $\tilde{S}^k:=\max_{m}\tilde{Z}^{k}_{m+\1}$, we have $\tilde{S}^k=S^\ast$ for all $k\ge0$.  
%
Hence, due to the same reasoning as the proof of Proposition \ref{kkkkaaaaa}, we must have $(D_x\bar{v})^{k+2}_m=(D_x\bar{v})^{k}_m$ for all $m,k$, which imply that there exists some constant $\beta\in\R$ such that 
$$\bar{\phi}_\delta \bar{v}^0+\beta=\bar{v}^0, \,\,\,\,\,\,
(\bar{\phi}_\delta)^k \bar{v}^0+2k\tau\beta=\bar{v}^0\mbox{ for all $k\ge0$}.$$ 
Since \eqref{35355353} implies that $(\bar{\phi}_\delta)^k \bar{v}^0=\bar{v}^0-2k\tau\beta$ is bounded for all $k\ge0$, we must have $\beta=0$. 

Now, we know that $\bar{v}^k=\bar{\varphi}_\delta^k\bar{v}^0$ is a stationary solution of  \eqref{HJ-delta-cell} and $\{v^{2K\cdot a_l+k}\}_{l\in\N}$ converges to $\bar{v}^k$ as $l\to\infty$. Note that $\bar{\varphi}_\delta^k$ is non-expanding, i.e., $\max_x|\bar{\varphi}_\delta^kw^0-\bar{\varphi}_\delta^k\tilde{w}^0 |\le \max_x| w^0- \tilde{w}^0|$ for all $k\ge0$ and  $w^0,\tilde{w}^0\in X_\delta$. For any $\ep>0$, we have $l=l(\ep)\in\N$ such that $\max_x|v^{2K\cdot a_l}-\bar{v}^0|<\ep$ and $\max_x|v^{2K\cdot a_l+1}-\bar{v}^1|<\ep$. If $k\ge K\cdot  a_{l}$ with $l=l(\ep)$, we have 
\begin{eqnarray*}
\max_x|\bar{\varphi}_\delta^{2k}v^0-\bar{v}^0|&=&
\max_x|\bar{\varphi}_\delta^{2k-2K\cdot a_l}\circ \bar{\varphi}_\delta^{2K\cdot a_l}v^0  
- \bar{\varphi}_\delta^{2k-2K\cdot a_l}\bar{v}^0|\\
&\le& \max_x|\bar{\varphi}_\delta^{2K\cdot a_l}v^0-\bar{v}^0|<\ep,  \\
\max_x|\bar{\varphi}_\delta^{2k+1}v^0-\bar{v}^1|&=&
\max_x|\bar{\varphi}_\delta^{2k-2K\cdot a_l}\circ \bar{\varphi}_\delta^{2K\cdot a_l+1}v^0  
- \bar{\varphi}_\delta^{2k-2K\cdot a_l}\bar{v}^1|\\
&\le& \max_x|\bar{\varphi}_\delta^{2K\cdot a_l+1}v^0-\bar{v}^1|<\ep.  
\end{eqnarray*}
\end{proof}
%

\subsection{Mather measure}
We re-formulate the  analogue of Mather's minimizing problem in the autonomous case.  
We call  $\xi: \pr h\Z^d=\pr(G_{\rm even}\cup G_{\rm odd})\to [-(d\lambda)^{-1},(d\lambda)^{-1}]^d$ an admissible control, where we identify $\xi$ as the function $\xi:\pr\tG \to[-(d\lambda)^{-1},(d\lambda)^{-1}]^d$ such that $\xi^{2k}_{m+\1}=\xi(x_{m+\1})$ and $\xi^{2k+1}_{m}=\xi(x_m)$ for all $k$ ($\xi$ also stands for its $1$-periodic extension to $\tG$). Define the functional for each $l\in\N$ as 
\begin{eqnarray*}
&&\mathcal{F}^l_\delta(\cdot;\xi): C_c(\T^d\times\R^d;\R)\to \R,\\
&&\mathcal{F}^l_\delta(f;\xi):=E_{\mu^{0,-l}_{n+\1}(\cdot; \xi)}\Big[\frac{1}{t_{l}}\sum_{-l<k\le0}f(\gamma^k,\xi^k_{m(\gamma^k)})\tau\Big].
\end{eqnarray*}
The Riesz representation theorem yields the unique  probability measure $\mu^l_\delta(\xi)$ of $\T^d\times\R^d$ such that 
\begin{eqnarray*}
\mathcal{F}^l_\delta(f;\xi)=\int_{\T^d\times\R^d}f \,\,d \mu^{l}_\delta(\xi)\mbox{\quad for any $f\in C_c(\T^d\times\R^d;\R)$}.
\end{eqnarray*}
We give the representation of $\mu^l_\delta(\xi)$ through $p(\cdot;\xi):\tG_\delta\to[0,1]$. 
Define $p^l_\delta(\cdot;\xi):\pr h\Z^d\to[0,1]$ as 
\begin{eqnarray*}
p^l_\delta(x;\xi):=\left\{
\begin{array}{lll}
&\dis \frac{1}{t_l}\sum_{\{(m,k)\,|\, \pr x_{m+\1}=x,\,\,k={\rm even}, \,\,-l<k\le0  \}} p(x_{m+\1},t_{k}; \xi)\tau\mbox{\quad for $x\in\pr G_{\rm odd}$},\\
&\dis \frac{1}{t_l}\sum_{\{(m,k)\,|\, \pr x_{m}=x,\,\,k+1={\rm odd}, \,\,-l<k+1\le0  \}} p(x_{m},t_{k+1}; \xi)\tau\mbox{\quad for $x\in\pr G_{\rm even}$}
\end{array}
\right.
\end{eqnarray*}
where we note that 
$$\sum_{x\in\pr h\Z^d} p^l_\delta(x;\xi)=1\mbox{\quad for each $l\in\N$}.$$
Due to the periodicity of $f\in C_c(\T^d\times\R^d;\R)$ and $\xi$, we have
\begin{eqnarray*}
\mathcal{F}^l_\delta(f;\xi)&=&\sum_{x\in\pr h\Z^d} p^l_\delta(x;\xi)f(x,\xi(x)),\\
\mu^l_\delta(\xi)&=&  \sum_{x\in\pr h\Z^d} p^l_\delta(x;\xi) \mathfrak{d}_{x,\xi(x)},
\end{eqnarray*}
where $\mathfrak{d}_{x,\zeta}$ is the Dirac measure of $\T^d\times\R^d$ supported by the point $(x,\zeta)$.  
Let $\mathcal{Q}_\delta(\xi)$ be the set of all limits of convergent subsequences of  $\{p^l_\delta(\cdot;\xi)\}_{l\in\N}$ to have  
$$\mathcal{P}_\delta=\bigcup_{\xi}\Big\{   
 \mu_\delta(\xi)=\sum_{x\in\pr h\Z^d} p_\delta(x;\xi) \mathfrak{d}_{x,\xi(x)}\,\Big|\,
 p_\delta(\cdot;\xi)\in\mathcal{Q}_\delta(\xi)
 \Big\},$$ 
where the union is taken over all admissible controls.

We re-state the holonomic-like  constraint of $\mathcal{P}_\delta$ in the autonomous case. For each function $g:\pr h\Z^d \to\R$, consider a continuous function $f:\T^d\times\R^d\to\R$ whose restriction to $\pr h\Z^d\times\R^d$ is  
\begin{eqnarray}\label{holoholo2} 
f(x_{m}, \zeta):=\frac{1}{\tau}\Big( g(x_m)-\frac{1}{2d}\sum_{\omega\in B}g(x_{m+\omega})  \Big)  + \sum_{i=1}^d\frac{g(x_{m+e_i})-g(x_{m-e_i})}{2h}\zeta^i.
\end{eqnarray}
\begin{Prop}\label{holonomic-delta2}
Each $\mu_\delta(\xi)\in\mathcal{P}_\delta$ satisfies the constraint 
$$\int_{\T^d\times\R^d}f\,\,d\mu_\delta(\xi)=0$$
for any continuous function $f:\T^d\times\R^d\to\R$ whose restriction to $\pr h\Z^d\times\R^d$ is of the form \eqref{holoholo2}. 
\end{Prop}
\begin{proof}
We identify each $g:\pr h\Z^d \to\R$ with the function $g:\pr \tG\to\R$ such that $g^{2k}_{x_{m+\1}}=g(x_{m+1})$ and $g^{2k+1}_{x_{m}}=g(x_{m})$ for all $k$, where $g$ is also $1$-periodically extended to $\tG$. Then, a function $f$ with \eqref{holoholo2} can be seen as the function $f:\T^d\times\T\times\R^d\to\R$ whose restriction to $\pr\tG\times\R^d$ is  
\begin{eqnarray*} 
f(x_{m},t_{k+1}, \zeta):=(D_tg)^{k+1}_{m}+(D_x g)^{k}_{m}\cdot \zeta,\quad (x_{m},t_{k+1}, \zeta)\in \pr\tG\times\R^d.
\end{eqnarray*}
The rest is the same as the proof of Proposition \ref{holonomic-delta}. 
\end{proof}
\noindent If $g$ is the restriction of $g\in C^1(\T^d;\R)$ to $\pr h\Z^d$, we have $\frac{1}{\tau}( g(x_m)-\frac{1}{2d}\sum_{\omega\in B}g(x_{m+\omega})) \to 0$ as $\delta\to0$ with the hyperbolic scaling. Hence,  Proposition \ref{holonomic-delta2} implies the exact holonomic condition at  the hyperbolic scaling limit. 

The analogue of Mather's minimizing problem is the following: 
\begin{Thm}\label{Mather-delta2}
For each $c$, we have
\begin{eqnarray}\label{33minimizing2}
\inf_{\mu_\delta\in \mathcal{P}_\delta} \int_{\T^d\times\R^d} L^{(c)}(x,\zeta)\,d\mu_\delta=-\bar{H}_\delta(c), 
\end{eqnarray}
where there exists at least one minimizing probability measure $\mu_\delta^\ast\in \mathcal{P}_\delta$ that attains the infimum. 
\end{Thm}
We define an analogue of the Mather set $\mathcal{M}_\delta(c)$ for each $c$ as 
$$\mathcal{M}_\delta(c):=\bigcup_{\mu^\ast_\delta}\mbox{supp}(\mu^\ast_\delta)\subset \pr h\Z^d\times[-(d\lambda)^{-1},(d\lambda)^{-1}]^d,$$
where the union is taken over all minimizing probability measures $\mu^\ast_\delta$ of \eqref{33minimizing2}. 
\begin{Thm}\label{Aubry-delta}
Let $\bar{v}^0\in X_\delta$ be such that $\varphi_\delta(\bar{v}^0;c)+\bar{H}_\delta(c)=\bar{v}^0$ and 
$\bar{v}^1:=\varphi^1_\delta(\bar{v}^0;c)+\tau\bar{H}_\delta(c)$.  Define $\bar{v}=\bar{v}(c):h\Z^d\to\R$ as  $\bar{v}(x_{m+\1})=\bar{v}^0_{m+\1}$ on $\pr G_{\rm odd}$ and $\bar{v}(x_{m})=\bar{v}^1_{m}$ on $\pr G_{\rm even}$. 
For each $c$, the support of any minimizing probability measure for \eqref{33minimizing2} is contained in the set 
\begin{eqnarray*}
\mathcal{A}_\delta(c):=\bigcap_{\bar{v}(c)} \Big\{\big(x_{m},H_p(x_{m},c+(D_x\bar{v})_{m})\big)\,|\, x_{m}\in \pr h\Z^d     \Big\},
\end{eqnarray*}
where the intersection is taken over all such $\bar{v}(c)$. 
\end{Thm}
\medskip\medskip\medskip

\noindent{\bf Acknowledgement.} The author  is supported by JSPS Grant-in-aid for Young Scientists \#18K13443.

\end{document}